\documentclass[11pt, twoside, leqno]{article}
\usepackage{mathrsfs}
\usepackage{amssymb}
\usepackage{amsmath}
\usepackage{mathrsfs}
\usepackage{amsthm}
\usepackage{amsfonts}
\usepackage{color}
\usepackage{latexsym}
\usepackage{txfonts}
\usepackage{indentfirst}
\usepackage{soul}

\allowdisplaybreaks
\newtheorem{theorem}{Theorem}[section]
\newtheorem{lemma}[theorem]{Lemma}
\newtheorem{corollary}[theorem]{Corollary}
\newtheorem{proposition}[theorem]{Proposition}

\theoremstyle{definition}
\newtheorem{remark}[theorem]{Remark}
\newtheorem{definition}[theorem]{Definition}

\numberwithin{equation}{section}

\begin{document}

\title{\bf\Large Extension Theorem and
Bourgain--Brezis--Mironescu-Type
Characterization of
Ball Banach Sobolev Spaces on Domains
\footnotetext{\hspace{-0.35cm} 2020
{\it Mathematics Subject Classification}.
Primary 46E35; Secondary 35A23, 42B25, 26D10.\endgraf
{\it Key words and phrases.}
ball Banach function space, ball Banach Sobolev space,
extension, $(\varepsilon,\delta)$-domain,
Bourgain--Brezis--Mironescu-type characterization.
\endgraf
This project is supported by the National Key Research
and Development Program of China
(Grant No. 2020YFA0712900) and the National
Natural Science Foundation of China
(Grant Nos. 12122102, 11971058 and 12071197).}}
\date{}
\author{Chenfeng Zhu,
Dachun Yang\footnote{Corresponding author,
E-mail: \texttt{dcyang@bnu.edu.cn}/{\color{red}
August 1, 2023}/Final version.}
\ and Wen Yuan}

\maketitle

\vspace{-0.7cm}

\begin{center}
\begin{minipage}{13cm}
{\small {\bf Abstract}\quad
Let $\Omega\subset\mathbb{R}^n$ be a bounded
$(\varepsilon,\infty)$-domain
with $\varepsilon\in(0,1]$,
$X(\mathbb{R}^n)$ a ball Banach function space
satisfying some extra mild assumptions,
and $\{\rho_\nu\}_{\nu\in(0,\nu_0)}$
with $\nu_0\in(0,\infty)$
a $\nu_0$-radial decreasing approximation of the identity
on $\mathbb{R}^n$.
In this article,
the authors establish two extension theorems, respectively,
on the inhomogeneous ball Banach Sobolev space
$W^{m,X}(\Omega)$ and the homogeneous ball Banach Sobolev space
$\dot{W}^{m,X}(\Omega)$ for any $m\in\mathbb{N}$.
On the other hand, the authors
prove that, for any $f\in\dot{W}^{1,X}(\Omega)$,
$$
\lim_{\nu\to0^+}
\left\|\left[\int_\Omega\frac{|f(\cdot)-f(y)|^p}{
|\cdot-y|^p}\rho_\nu(|\cdot-y|)\,dy
\right]^\frac{1}{p}\right\|_{X(\Omega)}^p
=\frac{2\pi^{\frac{n-1}{2}}\Gamma
(\frac{p+1}{2})}{\Gamma(\frac{p+n}{2})}
\left\|\,\left|\nabla f\right|\,\right\|_{X(\Omega)}^p,
$$
where $\Gamma$ is the Gamma function
and $p\in[1,\infty)$ is related to $X(\mathbb{R}^n)$.
Using this asymptotics, the authors further
establish a characterization of $W^{1,X}(\Omega)$
in terms of the above limit.
To achieve these,
the authors develop a machinery via using
a method of the extrapolation,
two extension theorems on
weighted Sobolev spaces,
and some recently found profound properties of $W^{1,X}(\mathbb{R}^n)$
to overcome those difficulties caused by
that the norm of $X(\mathbb{R}^n)$ has no explicit expression
and that $X(\mathbb{R}^n)$ might be neither the
reflection invariance nor the
translation invariance.
This characterization has a wide range of generality
and can be applied
to various Sobolev-type spaces,
such as Morrey
[Bourgain--Morrey-type,
weighted (or mixed-norm or variable),
local (or global) generalized Herz,
Lorentz, and Orlicz (or Orlicz-slice)] Sobolev spaces,
all of which are new.
Particularly, when
$X:=L^p$ with $p\in(1,\infty)$, this
characterization coincides with the
celebrated results of J. Bourgain, H. Brezis, and P. Mironescu
in 2001 and 2002 on smooth bounded domains;
moreover, this characterization is also new even when $X:=L^q$
with $1\leq p<q<\infty$.
}
\end{minipage}
\end{center}

\vspace{0.2cm}

\tableofcontents

\vspace{0.2cm}

\section{Introduction}
Let $\Omega$ be an open subset of $\mathbb{R}^n$.
It is well known that
the \emph{homogeneous fractional Sobolev space}
$\dot{W}^{s,p}(\Omega)$,
with both $s\in(0,1)$ and $p\in[1,\infty)$,
plays a vital role in solving many problems arising in
both partial
differential equations
and harmonic analysis (see, for instance,
\cite{bbm2002,crs2010,cv2011,npv2012,h2008,ht2008,
ma2011,m2011}),
which is defined to be the set of
all the measurable functions $f$
on $\Omega$ having the
following finite \emph{Gagliardo semi-norm}
\begin{align}\label{Gnorm}
\|f\|_{\dot{W}^{s,p}(\Omega)}
:=\left[\int_{\Omega}\int_{\Omega}
\frac{|f(x)-f(y)|^p}{|x-y|^{sp+n}}\,dx
\,dy\right]^{\frac{1}{p}}.
\end{align}
Recall that, if $\alpha:=(\alpha_1,\ldots,\alpha_n)
\in\mathbb{Z}_+^n$ is a multi-index and $f$
is a locally integral function on $\Omega$,
then $D^\alpha f$,
the \emph{$\alpha^\mathrm{th}$-weak partial derivative}
of $f$, is defined by setting,
for any $\phi\in C_{\mathrm{c}}^\infty(\Omega)$
(the set of all the infinitely differentiable
functions on $\Omega$ with compact support in $\Omega$),
\begin{align}\label{wpd}
\int_{\Omega}f(x)D^\alpha\phi(x)\,dx
=(-1)^{|\alpha|}\int_{\Omega}D^\alpha f(x)\phi(x)\,dx,
\end{align}
where $|\alpha|:=\sum_{j=1}^n\alpha_j$ and
$D^\alpha\phi:=\frac{\partial^{\alpha_1}}{\partial x_1^{\alpha_1}}
\ldots\frac{\partial^{\alpha_n}}{\partial x_n^{\alpha_n}}\phi$.
Recall that the \emph{homogeneous Sobolev space}
$\dot{W}^{1,p}(\Omega)$ with $p\in[1,\infty)$
is defined to be the
set of all the locally integrable functions $f$ on $\Omega$
having the following finite \emph{homogeneous Sobolev semi-norm}
$$
\|f\|_{\dot{W}^{1,p}(\Omega)}
:=\|\,|\nabla f|\,\|_{L^p(\Omega)},
$$
where
$\nabla f:=(\frac{\partial f}{\partial x_1},\ldots,
\frac{\partial f}{\partial x_n})$.
As is well known,
$\|\cdot\|_{\dot{W}^{s,p}(\Omega)}$ does not converge
to $\|\cdot\|_{\dot{W}^{1,p}(\Omega)}$
as $s\to1^-$; here and thereafter,
$s\to1^-$ means $s\in(0,1)$ and $s\to1$.
Indeed,
as was proved in \cite{bbm2000,b2002},
the integral in \eqref{Gnorm} when $s=1$ is infinite
unless $f$ is a constant function on $\Omega$.
In 2001, Bourgain et al. \cite{bbm2001}
gave a suitable way to
recover $\|\cdot\|_{\dot{W}^{1,p}(\Omega)}$
as the limit of the expression
$(1-s)^{1/p}\|\cdot\|_{\dot{W}^{s,p}(\Omega)}$
as $s\to1^-$,
where $\Omega\subset\mathbb{R}^n$ is a bounded
smooth domain.
This result in \cite{bbm2001} was further extended
by Brezis \cite{b2002} to $\mathbb{R}^n$.
Indeed, Brezis \cite{b2002} proved that,
for any $p\in[1,\infty)$ and $f\in\dot{W}^{1,p}(\mathbb{R}^n)
\cap L^p(\mathbb{R}^n)$,
\begin{align}\label{1055}
\lim_{s\to1^-}(1-s)
\|f\|_{\dot{W}^{s,p}(\mathbb{R}^n)}^p
=\frac{\kappa(p,n)}{p}\|f\|_{\dot{W}^{1,p}(\mathbb{R}^n)}^p,
\end{align}
where
\begin{align}\label{kappaqn}
\kappa(p,n):=\int_{\mathbb{S}^{n-1}}
\left|e\cdot\omega\right|^p\,d\sigma(\omega)
=\frac{2\pi^{\frac{n-1}{2}}\Gamma
(\frac{p+1}{2})}{\Gamma(\frac{p+n}{2})}
\end{align}
with $e$ being any unit vector in $\mathbb{R}^n$,
$d\sigma(\omega)$ the surface Lebesgue measure on the unit
sphere $\mathbb{S}^{n-1}$ of $\mathbb{R}^n$,
and $\Gamma$ the Gamma function.
The identity \eqref{1055} is nowadays referred to
as the Bourgain--Brezis--Mironescu formula on $\mathbb{R}^n$.
It is also worth mentioning that another
alternative way to amend the aforementioned deficiency
is via letting $s=1$ in \eqref{Gnorm} and
replacing the product $L^p$ norm
by the weak product $L^p$ quasi-norm simultaneously.
Such a striking novel way was given recently
by Brezis et al. \cite{bvy2021}.
The formula \cite[(1.3)]{bvy2021}
is nowadays referred to
as the Brezis--Van Schaftingen--Yung formula.
This result in \cite{bvy2021} was further generalized
by Brezis et al. \cite{bsvy.arxiv}.
We refer the reader to \cite{bsy2023,bn2016,
dm.arXiv,gh2023,hp.arXiv,ls2011,l2014,m2005,n2006}
for more studies
and applications of the Bourgain--Brezis--Mironescu formula
and to \cite{bsvy,
dlyyz2022,dlyyz.arxiv,dm.arXiv1,dssvy,dt.arXiv,frank,zyy2023}
for more related works about
recovering $\|\cdot\|_{\dot{W}^{1,p}(\mathbb{R}^n)}$.

On the other hand,
Sawano et al. \cite{shyy2017}
introduced
the concept of
ball quasi-Banach function spaces to
unify several important function spaces in
both harmonic analysis and partial
differential equations.
There are a lot of examples of ball quasi-Banach function spaces,
such as Morrey,
Bourgain--Morrey-type,
mixed-norm (or variable or weighted) Lebesgue,
local (or global) generalized Herz,
Lorentz, and Orlicz (or Orlicz-slice) spaces,
but some of which might not be
the classical Banach function spaces
in the sense of Bennett and Sharpley \cite{bs1988};
see \cite[Section~7]{shyy2017}.
Function spaces based on ball quasi-Banach
function spaces have attracted a large amount of attention and
made considerable progress in recent years;
we refer the reader
to \cite{h2021,wyy2020,zwyy2021} for the
boundedness of operators on ball quasi-Banach function spaces,
to \cite{cwyz2020,lyh2320,s2018,shyy2017,wyy.arxiv,yhyy2022-1,
yhyy2022-2,yyy2020,zhyy2022} for Hardy spaces
associated with ball quasi-Banach function spaces,
and to \cite{hcy2021,ins2019,is2017,tyyz2021,wyyz2021}
for further applications of ball quasi-Banach function spaces.

Observe that both the translation invariance and the explicit
expression of the $L^p(\mathbb{R}^n)$ norm play an
essential role in the proof of the
Bourgain--Brezis--Mironescu formula
\eqref{1055} in \cite{b2002}
and that both the rotation invariance and the explicit
expression of the $L^p(\mathbb{R}^n)$ norm
play a vital role in the proof of
the Brezis--Van Schaftingen--Yung formula
\cite[(1.3)]{bvy2021}.
However, it is known that
some of the aforementioned ball Banach function spaces
are obviously neither the translation invariance
nor the rotation invariance.
Thus, these two absences
make the generalization of both \eqref{1055} and \cite[(1.3)]{bvy2021}
from $L^p(\mathbb{R}^n)$ to ball Banach function spaces
become essentially difficult.
Recently, Dai et al. \cite{dlyyz.arxiv} subtly
used the Poincar\'e inequality,
the exquisite geometry of $\mathbb{R}^n$,
a method of extrapolation,
and the exact norm
of the Hardy--Littlewood
maximal operator
on $[X^\frac{1}{p}(\mathbb{R}^n)]'$
[the associate
space of the $1/p$-convexification of $X(\mathbb{R}^n)$
with $p\in[1,\infty)$]
to successfully extend
\cite[Theorems~1.1 and~1.2]{bvy2021} to ball Banach function spaces,
which was further generalized by Zhu et al. \cite{zyy2023}.
In the meantime, Dai et al. \cite{dgpyyz2022}
established an analogue of \eqref{1055} on
ball Banach function spaces by borrowing some ideas from
both \cite{dlyyz.arxiv} and \cite{bbm2001}.
Inspired by these recent progress, a natural question here
is: does \eqref{1055} also
hold true in the setting of ball Banach function spaces
on a connected and open set $\Omega\subset\mathbb{R}^n$
satisfying some extra mild assumptions?
We give a positive answer to this
problem in this article.

Motivated by the aforementioned works,
let $\Omega\subset\mathbb{R}^n$ be a bounded
$(\varepsilon,\infty)$-domain
with $\varepsilon\in(0,1]$,
$X(\mathbb{R}^n)$ a ball Banach function space
having an absolutely continuous norm
and satisfying that the Hardy--Littlewood maximal operator
is bounded on the associate
space of its convexification,
and $\{\rho_\nu\}_{\nu\in(0,\nu_0)}$ with $\nu_0\in(0,\infty)$
being a $\nu_0$-radial decreasing approximation of the identity
on $\mathbb{R}^n$.
In this article, we establish the extension theorems, respectively,
on the inhomogeneous ball Banach Sobolev space
$W^{m,X}(\Omega)$ and the homogeneous ball Banach Sobolev space
$\dot{W}^{m,X}(\Omega)$ for any $m\in\mathbb{N}$.
On the other hand,
we prove that, for any $f\in\dot{W}^{1,X}(\Omega)$,
\begin{align}\label{1942}
\lim_{\nu\to0^+}
\left\|\left[\int_\Omega\frac{|f(\cdot)-f(y)|^p}{
|\cdot-y|^p}\rho_\nu(|\cdot-y|)\,dy
\right]^\frac{1}{p}\right\|_{X(\Omega)}
=\left[\kappa(p,n)\right]^\frac{1}{p}
\left\|\,\left|\nabla f\right|\,\right\|_{X(\Omega)},
\end{align}
where $p\in[1,\infty)$ is related to $X(\mathbb{R}^n)$
and $\kappa(p,n)$ is the same as in \eqref{kappaqn}.
Here and thereafter, $\nu\to0^+$ means that there
exists $\nu_0\in(0,\infty)$ such that $v\in(0,v_0)$
and $\nu\to0$.
Using this asymptotics, we further
establish a characterization of $W^{1,X}(\Omega)$
in terms of the limit in the left hand side of \eqref{1942}.
Preciously,
under the extra assumption that $X'(\mathbb{R}^n)$
has an absolutely continuous norm, we show that
$f\in W^{1,X}(\Omega)$ if and only if $f\in X(\Omega)$
and
$$
\liminf_{\nu\to0^+}
\left\|\left[\int_\Omega\frac{|f(\cdot)-f(y)|^p}{
|\cdot-y|^p}\rho_\nu(|\cdot-y|)\,dy
\right]^\frac{1}{p}\right\|_{X(\Omega)}<\infty;
$$
moreover,
\eqref{1942} holds true for such $f$.
To achieve these, we develop a machinery via using
a method of the extrapolation,
two extension theorems on
weighted Sobolev spaces,
and some recently found
profound properties of $W^{1,X}(\mathbb{R}^n)$
to overcome those difficulties caused by
that the norm of $X(\mathbb{R}^n)$ has no explicit expression
and that $X(\mathbb{R}^n)$ might be
neither the reflection invariance nor the
translation invariance.
This characterization has a wide range of generality
and can be applied
to various Sobolev-type spaces,
including Morrey
[Bourgain--Morrey-type,
weighted (or mixed-norm or variable),
local (or global) generalized Herz,
Lorentz, and Orlicz (or Orlicz-slice)] Sobolev spaces,
all of which are new.
Particularly, when
$X:=L^p$ with $p\in(1,\infty)$, this
characterization coincides with the
celebrated result of Bourgain, Brezis, and Mironescu \cite[Theorem~2]{bbm2001}
in which $\Omega\subset\mathbb{R}^n$ is
assumed to be a smooth bounded domain
which is stronger than the hypothesis on $\Omega$
in the present article;
moreover, this characterization is also new even when $X:=L^q$
with $1\leq p<q<\infty$.

Let $p\in[1,\infty)$, $\nu_0:=\min\{n/p,\,1\}$, and
$R\in(0,\infty)$ be such that
$\Omega\subset B(\mathbf{0},R/2)$.
For any $\nu\in(0,\nu_0)$ and $r\in(0,\infty)$, take
\begin{align*}
\rho_\nu(r):=\nu p(2R)^{-\nu p}
r^{-n+\nu p}\mathbf{1}_{(0,2R]}(r).
\end{align*}
In this case, \eqref{1942} becomes
\begin{align*}
\lim_{s\to1^-}(1-s)^\frac{1}{p}\left\|\left[\int_{\Omega}
\frac{|f(\cdot)-f(y)|^p}{|\cdot-y|^{n+sp}}
\,dy\right]^{\frac{1}{p}}\right\|_{X(\Omega)}
=\left[\frac{\kappa(p,n)}{p}\right]^\frac{1}{p}
\left\|\,\left|\nabla f\right|\,\right\|_{X(\Omega)}
\end{align*}
(see Corollary~\ref{2001} below)
which gives the asymptotics of fractional
ball Banach Sobolev semi-norm involving
differences as $s\to1^-$.
Recall that,
in approximation theory,
establishing such asymptotices
via finite differences
is quite difficult,
even for some simple weighted Lebesgue spaces
(see \cite{k2015,mt2001} and the references therein).
This is because the difference operator
$\Delta_hf:=f(\cdot+h)-f(\cdot)$ with
$h\in\mathbb{R}^n\setminus\{\mathbf{0}\}$
might not be bounded on weighted Lebesgue spaces.
As it turns out, Dai et al. \cite{dgpyyz2022} overcame
those difficulties caused by that the norm of $X(\mathbb{R}^n)$
has no explicit expression and that $X(\mathbb{R}^n)$
might not be translation invariance
by exploiting the locally doubling property of $X(\mathbb{R}^n)$
(see \cite[Definition~2.1]{dgpyyz2022})
and using a method of extrapolation,
the Poincar\'e inequality,
and the fine geometric properties of systems of adjacent
dyadic cubes in $\mathbb{R}^n$.
These useful tools and techniques and
some profound properties in \cite{dgpyyz2022}
are taken full advantage of to obtain the main results of
this article.
On the other hand, it is remarkable that
\cite[Theorem~1]{bbm2001},
a key upper estimate to establish
\eqref{1055} on a bounded smooth domain $\Omega$,
relies on a standard extension theorem for $W^{1,p}(\Omega)$.
The proof of this extension theorem
depends on the reflection invariance
of $L^p(\mathbb{R}^n)$
(see, for instance, the proof of
\cite[p.\,270, Theorem~1]{evans}),
which might not be true for ball Banach function spaces.
To overcome this difficulty,
we use a method of extrapolation,
an upper estimate on $\mathbb{R}^n$
from \cite[(3.12)]{dgpyyz2022}, and an extension theorem
on weighted Sobolev spaces $W^{1,p}_\omega(\Omega)$
to establish the corresponding key upper estimate
in the setting of ball Banach function spaces
(see Proposition~\ref{1537} below).
Moreover, we should point out that
Theorem~\ref{1622},
the main result of this article,
is built under the assumption that
$\Omega\subset\mathbb{R}^n$ is a
bounded $(\varepsilon,\infty)$-domain
which is weaker than the assumption that
$\Omega\subset\mathbb{R}^n$ is a smooth bounded domain
in \cite{bbm2001}.
This is thanks to both
a density argument
and the extension theorem
on weighted Sobolev spaces
of Chua \cite{c1992}.

The organization of the remainder of this article is as follows.

In Section~\ref{section2},
we first give some preliminaries
on ball quasi-Banach function spaces and
introduce the concept
of the restrictive space of
a ball quasi-Banach function space.
Then we give some basic properties related to
the restrictive space.
Finally, we recall the concept of
the inhomogeneous and the homogeneous
ball Banach Sobolev spaces,
the Muckenhoupt weight class,
the weighted Lebesgue space,
and the inhomogeneous and the homogeneous
weighted Sobolev spaces.

In Section~\ref{yantuo},
we first recall the concept of
$(\varepsilon,\delta)$-domains.
Then, in Subsection~\ref{sub3.1}, we prove the
density of $C^\infty(\overline{\Omega})\cap W^{m,X}(\Omega)$
[resp. $C^\infty(\overline{\Omega})\cap\dot{W}^{m,X}(\Omega)$]
functions in $W^{m,X}(\Omega)$ [resp. $\dot{W}^{m,X}(\Omega)$]
for any $m\in\mathbb{N}$ (see Theorem~\ref{2247} below).
Using this, a method of extrapolation,
and two extension theorems on weighted Sobolev spaces
of Chua \cite{c1992},
we further establish the extension theorems
for $W^{m,X}(\Omega)$ on
$(\varepsilon,\delta)$-domains
in Subsection~\ref{sub3.2} and
for $\dot{W}^{m,X}(\Omega)$ on
$(\varepsilon,\infty)$-domains
in Subsection~\ref{sub3.3} respectively
(see Theorems~\ref{extension} and~\ref{extension2} below).

Section~\ref{S3} is devoted to establishing \eqref{1942}
for any $f\in\dot{W}^{1,X}(\Omega)$.
In Subsection~\ref{sub4.1}, we
prove \eqref{1942} for any $f\in C^\infty(\overline{\Omega})$
(see Proposition~\ref{2043} below)
via using the Taylor expansion for such $f$,
the properties of the
radial decreasing approximation of the identity,
and the dominated convergence theorem on
ball Banach function spaces.
In Subsection~\ref{sub4.2},
we establish
a key upper estimate for any $f\in\dot{W}^{1,X}(\Omega)$
(see Proposition~\ref{1537} below).
To obtain this, we first prove an upper estimate
on the weighted Sobolev space $\dot{W}^{1,p}_\omega(\Omega)$ via using
both an upper estimate on
$\dot{W}^{1,p}_\omega(\mathbb{R}^n)$ from \cite[(3.12)]{dgpyyz2022}
and an extension theorem on $\dot{W}^{1,p}_\omega(\Omega)$,
where $p\in[1,\infty)$ and $\omega\in A_1(\mathbb{R}^n)$.
This weighted estimate for $\dot{W}^{1,p}_\omega(\Omega)$,
together with a method of extrapolation,
further implies the desired upper
estimate for any $f\in\dot{W}^{1,X}(\Omega)$.
In Subsection~\ref{sub4.3},
we use the aforementioned  upper estimate and
a standard density argument
to extend \eqref{1942} from
any $f\in C^\infty(\overline{\Omega})$
to any $f\in\dot{W}^{1,X}(\Omega)$
(see Theorem~\ref{2045} below).

In Section~\ref{S4},
borrowing some ideas from both \cite{bbm2001}
and \cite{dgpyyz2022},
we first give a
sufficient condition for any $f\in X(\Omega)$
such that $f\in W^{1,X}(\Omega)$
(see Theorem~\ref{BBMdomain} below).
Combining this and the asymptotics
of $W^{1,X}(\Omega)$ proved in
Theorem~\ref{2045},
we further establish the
Bourgain--Brezis--Mironescu-type characterization
of $W^{1,X}(\Omega)$
(see Theorem~\ref{1622} below).

In Section~\ref{S5},
we apply Theorem~\ref{1622}
to various specific Sobolev-type spaces, such as
the weighted Sobolev space
and the Morrey--Sobolev space (see Subsection~\ref{5.4} below),
the Bourgain--Morrey-type Sobolev space
(see Subsection~\ref{BBMspace} below),
the local (or global) generalized Herz--Sobolev space
(see Subsection~\ref{Herz} below),
the mixed-norm Sobolev space (see Subsection~\ref{5.2} below),
the variable Sobolev space (see Subsection~\ref{5.3} below),
the Lorentz--Sobolev space (see Subsection~\ref{5.7} below),
the Orlicz--Sobolev space (see Subsection~\ref{5.5} below),
and the generalized amalgam Sobolev
space (see Subsection~\ref{5.6} below).
All of these results are new.
Obviously, due to the flexibility and the operability,
more applications of the main results of this
article to some newfound function spaces are predictable.

Finally, we point out that   the characterization of $W^{1,X}(\Omega)$
(see Corollary~\ref{2001} below)
can be further used to establish
the Brezis--Seeger--Van Schaftingen--Yung-type
characterization
of the homogeneous ball Banach Sobolev
space $\dot{W}^{1,X}(\mathbb{R}^n)$,
which is presented in the forthcoming article \cite{zyy20231}.

At the end of this section, we make some conventions on notation.
We always let $\mathbb{N}:=\{1,2,\ldots\}$,
$\mathbb{Z}_+:=\mathbb{N}\cup\{0\}$, and $\mathbb{S}^{n-1}$
be the unit sphere of $\mathbb{R}^n$.
We denote by $C$ a \emph{positive constant} which is independent
of the main parameters involved, but may vary from line to line.
We use $C_{(\alpha,\dots)}$ to denote a positive constant depending
on the indicated parameters $\alpha,\, \dots$.
The symbol $f\lesssim g$ means $f\le Cg$
and, if $f\lesssim g\lesssim f$, we then write $f\sim g$.
If $f\le Cg$ and $g=h$ or $g\le h$,
we then write $f\lesssim g=h$ or $f\lesssim g\le h$.
For any index $q\in[1,\infty]$,
we denote its \emph{conjugate index} by $q'$,
that is, $1/q+1/q'=1$.
We use $\mathbf{0}$ to denote the \emph{origin} of $\mathbb{R}^n$.
If $\Omega\subset\mathbb{R}^n$, we denote by ${\mathbf{1}}_\Omega$ its
\emph{characteristic function} and
by $\Omega^\complement$ the set $\mathbb{R}^n\setminus\Omega$.
For any measurable
set $E\subset\mathbb{R}^n$, $|E|$
denotes its $n$-dimensional Lebesgue measure.
Moreover, we denote by $d\sigma$
the surface Lebesgue measure on the unit sphere
$\mathbb{S}^{n-1}$ of $\mathbb{R}^n$.
For any $x\in\mathbb{R}^n$ and $r\in(0,\infty)$,
let $B(x,r):=\{y\in\mathbb{R}^n:\ |x-y|<r\}$.
For any $\lambda\in(0,\infty)$
and any ball $B:=B(x_B,r_B)$ in
$\mathbb{R}^n$ with both center $x_B\in\mathbb{R}^n$
and radius $r_B\in(0,\infty)$, let $\lambda B:=B(x_B,\lambda r_B)$.
For any measurable function $f$ on $\Omega\subset\mathbb{R}^n$,
its \emph{support} $\mathrm{supp\,}(f)$ is
defined by setting
$\mathrm{supp\,}(f):=\overline{\{x\in\Omega:\ f(x)\neq0\}}$.
The \emph{distance} $\mathrm{dist\,}(E,F)$ between two
arbitrary and nonempty sets $E,F\subset\mathbb{R}^n$
is defined by setting
$$
\mathrm{dist\,}(E,F):=\inf\left\{|x-y|:\
x\in E,\,y\in F\right\};
$$
if $E=\{x\}$ for some $x\in\mathbb{R}^n$,
we simply write
$\mathrm{dist\,}(x,F):=\mathrm{dist\,}(\{x\},F)$.
For any $p\in(0,\infty]$ and any
measurable set $\Omega\subset\mathbb{R}^n$,
the \emph{Lebesgue space}
$L^p(\Omega)$
is defined to be the set of
all the measurable functions $f$
on $\Omega$ such that
\begin{align*}
\|f\|_{L^p(\Omega)}:=
\left[\int_{\Omega}|f(x)|^p\,dx\right]^\frac{1}{p}<\infty.
\end{align*}
For any $p\in(0,\infty)$,
the set $L_{\mathrm{loc}}^p(\Omega)$ of all the locally $p$-integrable
functions on $\Omega$
is defined to be the set of all the
measurable functions $f$
on $\Omega$ such that,
for any open set $V\subset\Omega$
satisfying that $V$ is compactly contained in $\Omega$,
$f\in L^p(V)$.
For any $f\in L^1_{\mathrm{loc}}(\mathbb{R}^n)$,
its \emph{Hardy--Littlewood maximal function} $\mathcal{M}(f)$
is defined by setting, for any $x\in\mathbb{R}^n$,
\begin{align*}
\mathcal{M}(f)(x):=\sup_{B\ni x}
\frac{1}{|B|}\int_B\left|f(y)\right|\,dy,
\end{align*}
where the supremum is taken over all the
balls $B\subset\mathbb{R}^n$ containing $x$.
Finally, when we prove a theorem (or the like),
in its proof we always use the same symbols as in the
statement itself of that theorem (or the like).

\section{Ball Banach Function Spaces}
\label{section2}

Throughout this section,
let $\Omega\subset\mathbb{R}^n$ be an open set.
In this section, we first give some preliminaries
on ball quasi-Banach function spaces and
then introduce the concepts
of both the restrictive space of
a ball quasi-Banach function space and
the (in)homogeneous ball Banach Sobolev space on $\Omega$.
Finally, we recall the concepts of
the Muckenhoupt weight class,
the weighted Lebesgue space,
and the (in)homogeneous weighted Sobolev space.
We begin with the following definition
of ball quasi-Banach function
spaces, which was introduced in \cite[Definition~2.2]{shyy2017}.
For $E\in\{\mathbb{R}^n,\,\Omega\}$,
the \emph{symbol} $\mathscr{M}(E)$ denotes
the set of all measurable functions
on $E$.

\begin{definition}\label{1659}
A quasi-Banach space $X(\mathbb{R}^n)
\subset\mathscr{M}(\mathbb{R}^n)$,
equipped with a quasi-norm $\|\cdot\|_{X(\mathbb{R}^n)}$
which makes sense for all functions
in $\mathscr{M}(\mathbb{R}^n)$,
is called a \emph{ball quasi-Banach function space} if
\begin{enumerate}
\item[\textup{(i)}]
for any $f\in\mathscr{M}(\mathbb{R}^n)$,
if $\|f\|_{X(\mathbb{R}^n)}=0$, then $f=0$
almost everywhere in $\mathbb{R}^n$;
\item[\textup{(ii)}]
if $f,g\in\mathscr{M}(\mathbb{R}^n)$
with $|g|\leq|f|$ almost everywhere in $\mathbb{R}^n$,
then $\|g\|_{X(\mathbb{R}^n)}\leq\|f\|_{X(\mathbb{R}^n)}$;
\item[\textup{(iii)}]
if a sequence $\{f_m\}_{m\in\mathbb{N}}
\subset\mathscr{M}(\mathbb{R}^n)$
and $f\in\mathscr{M}(\mathbb{R}^n)$ satisfy
that $0\leq f_m\uparrow f$ almost everywhere in $\mathbb{R}^n$
as $m\to\infty$, then $\|f_m\|_{X(\mathbb{R}^n)}
\uparrow\|f\|_{X(\mathbb{R}^n)}$ as $m\to\infty$;
\item[\textup{(iv)}]
for any ball $B\subset\mathbb{R}^n$,
$\mathbf{1}_{B}\in X(\mathbb{R}^n)$.
\end{enumerate}
Moreover, a ball quasi-Banach function
space $X(\mathbb{R}^n)$ is called a
\emph{ball Banach function space} if
$X(\mathbb{R}^n)$ satisfies the following extra conditions:
\begin{enumerate}
\item[\textup{(v)}]
for any $f,g\in X(\mathbb{R}^n)$,
$\|f+g\|_{X(\mathbb{R}^n)}\leq
\|f\|_{X(\mathbb{R}^n)}+\|g\|_{X(\mathbb{R}^n)}$;
\item[\textup{(vi)}]
for any ball $B\subset\mathbb{R}^n$,
there exists a positive constant $C_{(B)}$, depending on $B$,
such that, for any $f\in X(\mathbb{R}^n)$,
$$
\int_{B}\left|f(x)\right|\,dx\leq C_{(B)}\|f\|_{X(\mathbb{R}^n)}.
$$
\end{enumerate}
\end{definition}

\begin{remark}\label{1052}
\begin{enumerate}
\item[\textup{(i)}]
Let $X(\mathbb{R}^n)$ be a ball quasi-Banach function space.
By \cite[Remark~2.5(i)]{yhyy2022-1},
we find that, for any $f\in\mathscr{M}(\mathbb{R}^n)$,
$\|f\|_{X(\mathbb{R}^n)}=0$ if and only if
$f=0$ almost everywhere in $\mathbb{R}^n$.
\item[\textup{(ii)}]
As was mentioned in \cite[Remark~2.5(ii)]{yhyy2022-1},
we obtain an equivalent formulation
of Definition~\ref{1659} via replacing any
ball $B\subset\mathbb{R}^n$ therein
by any bounded measurable set $S\subset\mathbb{R}^n$.
\item[\textup{(iii)}]
In Definition~\ref{1659}, if we replace any ball $B$ by any
measurable set $E$ with $|E|<\infty$,
then we obtain the definition of
quasi-Banach function spaces, which was originally introduced
by Bennett and Sharpley in
\cite[Chapter 1, Definitions~1.1 and~1.3]{bs1988}.
Thus, a quasi-Banach function space is always a ball
quasi-Banach function space, but the converse is
not necessary to be true.
\item[\rm(iv)]
In Definition~\ref{1659},
if we replace (iv)
by the following \emph{saturation property}:
\begin{enumerate}
\item[\rm(a)]
for any measurable set $E\subset\mathbb{R}^n$
of positive measure, there exists a measurable set $F\subset E$
of positive measure satisfying that $\mathbf{1}_F\in X(\mathbb{R}^n)$,
\end{enumerate}
then we obtain the definition of quasi-Banach function spaces
in Lorist and Nieraeth
\cite{lz20231}.
Moreover, by \cite[Proposition~2.5]{zyy20231}
(see also \cite[Proposition~4.22]{narxiv}),
we find that, if the
quasi-normed vector space $X(\mathbb{R}^n)$ satisfies
the extra assumption that
the Hardy--Littlewood maximal operator is weakly bounded on
its convexification,
then the definition of quasi-Banach function spaces in \cite{lz20231} coincides with
the definition of ball quasi-Banach function spaces. Thus,
under this extra assumption,
working with ball quasi-Banach function
spaces in the sense of Definition~\ref{1659}
or quasi-Banach function spaces in
the sense of \cite{lz20231} would yield
exactly the same results.
\item[\textup{(v)}]
From \cite[Proposition~1.2.36]{lyh2320},
we infer that both (ii) and (iii) of
Definition~\ref{1659} imply that any ball quasi-Banach function
space is complete.
\end{enumerate}
\end{remark}

Now, we recall the
concept of the $p$-convexification of a ball
quasi-Banach function space;
see, for instance, \cite[Definition~2.6]{shyy2017}.

\begin{definition}\label{tuhua}
Let $p\in(0,\infty)$.
The \emph{$p$-convexification} $X^p(\mathbb{R}^n)$ of
the ball quasi-Banach function space
$X(\mathbb{R}^n)$ is defined by setting
$$
X^p(\mathbb{R}^n):=\{f\in\mathscr{M}(\mathbb{R}^n):\
|f|^p\in X(\mathbb{R}^n)\}
$$
equipped with the quasi-norm
$\|f\|_{X^p(\mathbb{R}^n)}:=\|\,|f|^p\|_{X(\mathbb{R}^n)}^\frac{1}{p}$
for any $f\in X^p(\mathbb{R}^n)$.
\end{definition}

We recall the
concept of the associate space of a ball
Banach function space, which can be found in \cite[p.\,9]{shyy2017};
see \cite[Chapter 1, Section 2]{bs1988} for more details.

\begin{definition}\label{associte}
The \emph{associate space}
(also called the \emph{K\"othe dual}) $X'(\mathbb{R}^n)$ of
the ball Banach function space $X(\mathbb{R}^n)$
is defined by setting
\begin{align*}
X'(\mathbb{R}^n):=\left\{f\in\mathscr{M}(\mathbb{R}^n):\
\|f\|_{X'(\mathbb{R}^n)}:=\sup_{\{g\in X(\mathbb{R}^n):\
\|g\|_{X(\mathbb{R}^n)}=1\}}
\left\|fg\right\|_{L^1(\mathbb{R}^n)}<\infty\right\},
\end{align*}
where $\|\cdot\|_{X'(\mathbb{R}^n)}$ is called the
\emph{associate norm} of $\|\cdot\|_{X(\mathbb{R}^n)}$.
\end{definition}

\begin{remark}\label{dual}
Let $X(\mathbb{R}^n)$ be a ball Banach function space.
\begin{enumerate}
\item[\rm(i)]
From \cite[Proposition~2.3]{shyy2017},
we infer that $X'(\mathbb{R}^n)$ is also a ball Banach function space.
\item[\rm(ii)]
By \cite[Lemma~2.6]{zwyy2021},
we find that $X(\mathbb{R}^n)$ coincides with its second
associate space $X''(\mathbb{R}^n)$.
\end{enumerate}
\end{remark}

Next, we introduce the concept of the restrictive
space of a ball quasi-Banach function space.
For any $g\in\mathscr{M}(\mathbb{R}^n)$,
denote by $g|_\Omega$ the restriction of $g$ on $\Omega$.

\begin{definition}\label{1635}
The \emph{restrictive space} $X(\Omega)$ of
the ball quasi-Banach function space $X(\mathbb{R}^n)$
is defined by setting
\begin{align*}
X(\Omega):=\left\{f\in\mathscr{M}(\Omega):\
f=g|_\Omega\text{ for some }g\in X(\mathbb{R}^n)\right\};
\end{align*}
moreover, for any $f\in X(\Omega)$, let
\begin{align}\label{1035}
\left\|f\right\|_{X(\Omega)}:=\inf\left\{\|g\|_{X(\mathbb{R}^n)}:\
f=g|_\Omega,\,g\in X(\mathbb{R}^n)\right\}.
\end{align}
\end{definition}

The following proposition shows that
the infimum in \eqref{1035} can be attained.

\begin{proposition}\label{norm}
Let $X(\mathbb{R}^n)$ be a ball quasi-Banach function space
and $X(\Omega)$ its restrictive space.
Then, for any $f\in X(\Omega)$,
\begin{align*}
\left\|f\right\|_{X(\Omega)}=
\left\|\widetilde{f}\right\|_{X(\mathbb{R}^n)},
\end{align*}
where
\begin{align}\label{1448}
\widetilde{f}(x):=
\begin{cases}
f(x)&\textup{ if }x\in\Omega\\
0&\textup{ if }x\in\Omega^\complement.
\end{cases}
\end{align}
\end{proposition}

\begin{proof}
Let $f\in X(\Omega)$.
We first show that
\begin{align}\label{1041}
\left\|f\right\|_{X(\Omega)}\ge
\left\|\widetilde{f}\right\|_{X(\mathbb{R}^n)}.
\end{align}
Indeed, for any given $g\in X(\mathbb{R}^n)$
satisfying that $f=g|_\Omega$, it is easy to show that,
for any $x\in\mathbb{R}^n$,
$|\widetilde{f}(x)|\leq|g(x)|$,
which, together with Definition~\ref{1659}(ii),
further implies that
\begin{align}\label{1053}
\left\|\widetilde{f}\right\|_{X(\mathbb{R}^n)}
\leq\|g\|_{X(\mathbb{R}^n)}.
\end{align}
Taking the infimum over all $g\in X(\mathbb{R}^n)$
satisfying that $g|_\Omega=f$,
we conclude that \eqref{1041} holds true.

Next, we show that
\begin{align}\label{1056}
\left\|f\right\|_{X(\Omega)}\leq
\left\|\widetilde{f}\right\|_{X(\mathbb{R}^n)}.
\end{align}
From \eqref{1053}, we deduce that
$\|\widetilde{f}\|_{X(\mathbb{R}^n)}<\infty$ and hence
$\widetilde{f}\in X(\mathbb{R}^n)$.
Notice that $\widetilde{f}$ satisfies $\widetilde{f}|_\Omega=f$.
By this and \eqref{1035}, we conclude that \eqref{1056} holds true,
which, combined with \eqref{1041},
then completes the proof of Proposition~\ref{norm}.
\end{proof}

The following proposition shows that $X(\Omega)$
has similar properties to those in Definition~\ref{1659}.

\begin{proposition}\label{1533}
Let $X(\mathbb{R}^n)$ be a ball quasi-Banach function space
and $X(\Omega)$ its restrictive space.
Then $X(\Omega)$ is a quasi-Banach space satisfying that
\begin{enumerate}
\item[\rm(i)]
for any $f\in\mathscr{M}(\Omega)$,
if $\|f\|_{X(\Omega)}=0$, then $f=0$
almost everywhere in $\Omega$;
\item[\rm(ii)]
if $f,g\in\mathscr{M}(\Omega)$
with $|g|\leq|f|$ almost everywhere in $\Omega$,
then $\|g\|_{X(\Omega)}\leq\|f\|_{X(\Omega)}$;
\item[\rm(iii)]
if a sequence $\{f_m\}_{m\in\mathbb{N}}
\subset\mathscr{M}(\Omega)$
and $f\in\mathscr{M}(\Omega)$ satisfy
that $0\leq f_m\uparrow f$ almost everywhere in $\Omega$
as $m\to\infty$, then $\|f_m\|_{X(\Omega)}
\uparrow\|f\|_{X(\Omega)}$ as $m\to\infty$;
\item[\rm(iv)]
for any bounded measurable set $E\subset\Omega$,
$\mathbf{1}_{E}\in X(\Omega)$;
particularly, for any ball $B\subset\mathbb{R}^n$,
$\mathbf{1}_{B\cap\Omega}\in X(\Omega)$.
\end{enumerate}
Assume further that $X(\mathbb{R}^n)$ is a ball Banach function space,
then $X(\Omega)$ also satisfies that
\begin{enumerate}
\item[\rm(v)]
for any $f,g\in X(\Omega)$,
$\|f+g\|_{X(\Omega)}\leq\|f\|_{X(\Omega)}+\|g\|_{X(\Omega)}$;
\item[\rm(vi)]
for any bounded measurable set $E\subset\Omega$,
there exists a positive constant $C_{(E)}$, depending on $E$,
such that, for any $f\in X(\Omega)$,
$$
\int_{E}\left|f(x)\right|\,dx\leq C_{(E)}\|f\|_{X(\Omega)};
$$
particularly, for any ball $B\subset\mathbb{R}^n$,
there exists a positive constant $C_{(B)}$, depending on $B$,
such that, for any $f\in X(\Omega)$,
$$
\int_{B\cap\Omega}\left|f(x)\right|\,dx\leq C_{(B)}\|f\|_{X(\Omega)}.
$$
\end{enumerate}
\end{proposition}

\begin{proof}
For any $f\in X(\Omega)$, let $\widetilde{f}$
be the same as in \eqref{1448}.
We first show that $X(\Omega)$ is complete.
Let $\{f_k\}_{k\in\mathbb{N}}$ be a Cauchy sequence in $X(\Omega)$.
By Proposition~\ref{norm}, we find that, for any $k,m\in\mathbb{N}$,
$$
\left\|\widetilde{f_k}-\widetilde{f_m}\right\|_{X(\mathbb{R}^n)}
=\left\|f_k-f_m\right\|_{X(\Omega)},
$$
which, further implies that $\{\widetilde{f_k}\}_{k\in\mathbb{N}}$
is a Cauchy sequence in $X(\mathbb{R}^n)$.
From this and Remark~\ref{1052}(iv),
we deduce that there exists $f\in X(\mathbb{R}^n)$
such that
\begin{align}\label{1637}
\left\|f-\widetilde{f_k}\right\|_{X(\mathbb{R}^n)}\to0
\end{align}
as $k\to\infty$.
Let $f_0:=f|_\Omega$. Then, by Definition~\ref{1635},
we conclude that $f_0\in X(\Omega)$.
Moreover, from Proposition~\ref{norm},
Definition~\ref{1659}(ii), and \eqref{1637},
we deduce that
\begin{align*}
\left\|f_0-f_k\right\|_{X(\Omega)}
=\left\|\widetilde{f_0}-\widetilde{f_k}\right\|_{X(\mathbb{R}^n)}
\leq\left\|f-\widetilde{f_k}\right\|_{X(\mathbb{R}^n)}\to0
\end{align*}
as $k\to\infty$. Thus, $X(\Omega)$ is complete.

By Proposition~\ref{norm} and (i), (ii), (iii), and (v)
of Definition~\ref{1659} with $f$, $g$, $f_m$ replaced,
respectively, by
$\widetilde{f}$, $\widetilde{g}$, and $\widetilde{f_m}$,
we conclude that (i), (ii), (iii), and (v) of the present proposition
hold true. Now, we show (iv).
Assume that $E$ is a bounded measurable subset of $\Omega$.
Then $E$ is also a bounded measurable subset of $\mathbb{R}^n$.
By Remark~\ref{1052}(ii) and Definition~\ref{1659}(iv),
we find that $\widetilde{\mathbf{1}}_E\in X(\mathbb{R}^n)$.
Thus, $\mathbf{1}_E=\widetilde{\mathbf{1}}_E|_\Omega\in X(\Omega)$
by Definition~\ref{1635}.
This finishes the proof of (iv).
Next, we show (vi). From  Definition~\ref{1659}(vi),
Remark~\ref{1052}(ii), and Proposition~\ref{norm},
we infer that, for any bounded measurable set $E\subset\Omega$,
\begin{align*}
\int_E|f(x)|\,dx=\int_E\left|\widetilde{f}(x)\right|\,dx
\leq C_{(E)}\left\|\widetilde{f}\right\|_{X(\mathbb{R}^n)}
=C_{(E)}\left\|f\right\|_{X(\Omega)},
\end{align*}
which completes the proof of (vi) and hence Proposition~\ref{1533}.
\end{proof}

The following concept
can be found
in \cite[Chapter 1, Definition~3.1]{bs1988};
see also \cite[Definition~3.2]{wyy2020}.

\begin{definition}\label{2029}
Let $X(\mathbb{R}^n)$ be a ball quasi-Banach function space
and $X(\Omega)$ its restrictive space.
For $E\in\{\mathbb{R}^n,\,\Omega\}$,
$X(E)$ is said
to have an \emph{absolutely continuous quasi-norm on $E$}
if, for any $f\in X(E)$ and any sequence $\{E_j\}_{j\in\mathbb{N}}$
of measurable subsets of $E$
satisfying that $\mathbf{1}_{E_j}\to0$
almost everywhere in $E$ as $j\to\infty$, one has
$\|f\mathbf{1}_{E_j}\|_{X(E)}\to0$ as $j\to\infty$.
\end{definition}

The following result shows that
the restriction preserves the
absolutely continuous quasi-norm.

\begin{proposition}\label{abs}
Let $X(\mathbb{R}^n)$ be a ball quasi-Banach function space
having an absolutely continuous quasi-norm.
Then $X(\Omega)$ also has an absolutely continuous quasi-norm.
\end{proposition}

\begin{proof}
Assume that $f\in X(E)$
and $\{E_j\}_{j\in\mathbb{N}}$ is a sequence of
measurable subsets of $E$
satisfying that $\mathbf{1}_{E_j}\to0$
almost everywhere in $E$ as $j\to\infty$.
Let $\widetilde{f}$ and $\widetilde{1_{E_j}}$
be defined the same as in \eqref{1448}.
Then we find that $\widetilde{f}\in X(\mathbb{R}^n)$
and $\{E_j\}_{j\in\mathbb{N}}$ is also a sequence of
measurable subsets of $\mathbb{R}^n$
satisfying that $\widetilde{\mathbf{1}_{E_j}}\to0$
almost everywhere in $\mathbb{R}^n$ as $j\to\infty$.
By Proposition~\ref{norm} and the assumption that
$X(\mathbb{R}^n)$ has an absolutely
continuous quasi-norm,
we conclude that
$$
\left\|f\mathbf{1}_{E_j}\right\|_{X(E)}
=\left\|\widetilde{f}\widetilde{1_{E_j}}\right\|_{X(\mathbb{R}^n)}
\to0
$$
as $j\to\infty$.
This finishes the proof of Proposition~\ref{abs}.
\end{proof}

Next, we introduce the associate space of $X(\Omega)$.

\begin{definition}\label{2036}
Let $X(\mathbb{R}^n)$ be a ball Banach function space
and $X(\Omega)$ its restrictive space.
The \emph{associate space} $[X(\Omega)]'$ of
$X(\Omega)$
is defined by setting
\begin{align*}
[X(\Omega)]':=\left\{f\in\mathscr{M}(\Omega):\
\|f\|_{[X(\Omega)]'}:=\sup_{\{g\in X(\Omega):\ \|g\|_{X(\Omega)}=1\}}
\left\|fg\right\|_{L^1(\Omega)}<\infty\right\},
\end{align*}
where $\|\cdot\|_{[X(\Omega)]'}$ is called the
\emph{associate norm} of $\|\cdot\|_{X(\Omega)}$.
\end{definition}

The following proposition shows that
$[X(\Omega)]'$ coincides with $X'(\Omega)$
with same norms.

\begin{proposition}\label{2034}
Let $X(\mathbb{R}^n)$ be a ball Banach function space,
$X'(\mathbb{R}^n)$ its associate space,
and $X(\Omega)$ its restrictive space on $\Omega$.
Then $[X(\Omega)]'=X'(\Omega)$
with same norms,
where $X'(\Omega)$ denotes the restrictive space of $X'(\mathbb{R}^n)$.
\end{proposition}

\begin{proof}
On the one hand, assume that $f\in[X(\Omega)]'$.
Let $\widetilde{f}$ be the same as in \eqref{1448}.
By Definitions~\ref{associte}
and~\ref{2036}, we find that
\begin{align*}
\left\|\widetilde{f}\right\|_{X'(\mathbb{R}^n)}
&=\sup_{\{g\in X(\mathbb{R}^n):\ \|g\|_{X(\mathbb{R}^n)}=1\}}
\int_{\mathbb{R}^n}\left|\widetilde{f}(x)g(x)\right|\,dx\\
&=\sup_{\{g\in X(\mathbb{R}^n):\ \|g\|_{X(\mathbb{R}^n)}=1\}}
\int_{\Omega}\left|f(x)g(x)\right|\,dx\\
&\leq\sup_{\{h\in X(\Omega):\ \|h\|_{X(\Omega)}=1\}}
\int_{\Omega}\left|f(x)h(x)\right|\,dx=\left\|f\right\|_{[X(\Omega)]'}
<\infty.
\end{align*}
Thus, $\widetilde{f}\in X'(\mathbb{R}^n)$
and $f$, the restriction of $\widetilde{f}$ on $\Omega$,
belongs to $X'(\Omega)$.

On the other hand, assume that $f\in X'(\Omega)$.
From Proposition~\ref{norm}, we infer that
$\widetilde{f}\in X'(\mathbb{R}^n)$ and
$\|\widetilde{f}\|_{X'(\mathbb{R}^n)}=\|f\|_{X'(\Omega)}$.
By this and Definitions~\ref{2036} and~\ref{associte}, we conclude that
\begin{align*}
\left\|f\right\|_{[X(\Omega)]'}
&=\sup_{\{g\in X(\Omega):\ \|g\|_{X(\Omega)}=1\}}
\int_{\Omega}\left|f(x)g(x)\right|dx\\
&=\sup_{\{g\in X(\Omega):\ \|g\|_{X(\Omega)}=1\}}
\int_{\mathbb{R}^n}\left|\widetilde{f}(x)\widetilde{g}(x)\right|dx\\
&\leq\sup_{\{h\in X(\mathbb{R}^n):\ \|h\|_{X(\mathbb{R}^n)}=1\}}
\int_{\mathbb{R}^n}\left|\widetilde{f}(x)h(x)\right|dx\\
&=\left\|\widetilde{f}\right\|_{X'(\mathbb{R}^n)}
=\left\|f\right\|_{X'(\Omega)}<\infty.
\end{align*}
Thus, $f\in[X(\Omega)]'$.
This finishes the proof of Proposition~\ref{2034}.
\end{proof}

\begin{remark}\label{1104}
Let $X(\mathbb{R}^n)$ be a ball Banach function space
and $X(\Omega)$ its restrictive space.
Assume that $X(\mathbb{R}^n)$ has an absolutely continuous norm.
For $E\in\{\mathbb{R}^n,\,\Omega\}$,
we use $X^*(E)$ to denote the \emph{dual space} of $X(E)$.
Then, from \cite[Lemma~1.7.7]{lyh2320},
we deduce that $X'(E)$ coincides with $X^*(E)$;
see also \cite[p.\,23, Corollary~4.3]{bs1988}.
\end{remark}

We introduce the concepts of both
the inhomogeneous
ball Banach Sobolev space $W^{m,X}(\Omega)$
and the homogeneous
ball Banach Sobolev space $\dot{W}^{m,X}(\Omega)$
as follows;
see \cite[Definition~2.6]{dgpyyz2022}
for the definition of $W^{1,X}(\mathbb{R}^n)$
and \cite[Definition~2.4]{dlyyz.arxiv}
for the definition of $\dot{W}^{1,X}(\mathbb{R}^n)$.

\begin{definition}\label{2.7}
Let $X(\mathbb{R}^n)$ be a ball Banach function space
and $m\in\mathbb{N}$.
\begin{enumerate}
\item[\rm(i)]
The \emph{inhomogeneous ball Banach
Sobolev space} $W^{m,X}(\Omega)$
is defined to be the set of all the
$f\in X(\Omega)$
satisfying that, for any multi-index
$\alpha\in\mathbb{Z}_+^n$ with $|\alpha|\leq m$,
the $\alpha^\mathrm{th}$-weak partial derivative $D^\alpha f$ of $f$
exists and $D^\alpha f\in X(\Omega)$.
Moreover, for any $f\in W^{m,X}(\Omega)$, let
\begin{align*}
\|f\|_{W^{m,X}(\Omega)}:=
\sum_{\alpha\in\mathbb{Z}_+^n,\,
|\alpha|\leq m}\left\|D^\alpha f\right\|_{X(\Omega)}.
\end{align*}
\item[\rm(ii)]
The \emph{homogeneous ball Banach
Sobolev space} $\dot{W}^{m,X}(\Omega)$
is defined to be the set of all the
$f\in L^1_{\mathrm{loc}}(\Omega)$
satisfying that, for any multi-index
$\alpha\in\mathbb{Z}_+^n$ with $|\alpha|=m$,
the $\alpha^\mathrm{th}$-weak partial derivative $D^\alpha f$ of $f$
exists and $D^\alpha f\in X(\Omega)$.
Moreover, for any $f\in\dot{W}^{m,X}(\Omega)$, let
\begin{align*}
\|f\|_{\dot{W}^{m,X}(\Omega)}:=
\sum_{\alpha\in\mathbb{Z}_+^n,\,
|\alpha|=m}\left\|D^\alpha f\right\|_{X(\Omega)}.
\end{align*}
\end{enumerate}
\end{definition}

Now, we recall the concept of the
Muckenhoupt $A_p(\mathbb{R}^n)$ class
(see, for instance, \cite[Definitions~7.1.1 and~7.1.3]{g2014}).

\begin{definition}\label{1557}
An \emph{$A_p(\mathbb{R}^n)$-weight} $\omega$, with $p\in[1,\infty)$,
is a nonnegative locally integrable function
on $\mathbb{R}^n$ satisfying that,
when $p=1$,
\begin{align}\label{A1}
[\omega]_{A_1(\mathbb{R}^n)}:=\sup_{Q\subset\mathbb{R}^n}
\frac{\|\omega^{-1}\|_{L^\infty(Q)}}{|Q|}\int_Q\omega(x)\,dx<\infty
\end{align}
and, when $p\in(1,\infty)$,
\begin{align}\label{Ap}
[\omega]_{A_p(\mathbb{R}^n)}:=\sup_{Q\subset\mathbb{R}^n}
\frac{1}{|Q|}\int_Q\omega(x)\,dx
\left\{\frac{1}{|Q|}\int_Q
[\omega(x)]^{1-p'}\,dx\right\}^{p-1}<\infty,
\end{align}
where $p'$ denotes the conjugate index of $p$ and
the suprema in both \eqref{A1} and \eqref{Ap}
are taken over all cubes $Q\subset\mathbb{R}^n$.
Moreover, let
$A_\infty(\mathbb{R}^n):=\bigcup_{p\in[1,\infty)}A_p(\mathbb{R}^n).$
\end{definition}

Next, we recall the concepts of both the weighted Lebesgue space and
the weighted Sobolev space as follows.

\begin{definition}\label{1556}
Let $p\in[1,\infty]$,
$\omega\in A_p(\mathbb{R}^n)$,
and $m\in\mathbb{N}$.
\begin{enumerate}
\item[\textup{(i)}]
The \emph{weighted Lebesgue space} $L^p_\omega(\Omega)$
is defined to be the set of all the $f\in\mathscr{M}(\Omega)$ such that
\begin{align*}
\|f\|_{L^p_\omega(\Omega)}:=
\left[\int_{\Omega}\left|f(x)\right|^p
\omega(x)\,dx\right]^{\frac{1}{p}}<\infty.
\end{align*}
\item[\textup{(ii)}]
The \emph{inhomogeneous weighted Sobolev space}
$W^{m,p}_\omega(\Omega)$
is defined to be the set of all the $f\in L^p_\omega(\Omega)$
satisfying that, for any multi-index
$\alpha\in\mathbb{Z}_+^n$ with $|\alpha|\leq m$,
the $\alpha^\mathrm{th}$-weak partial derivative $D^\alpha f$ of $f$
exists and $D^\alpha f\in L^p_\omega(\Omega)$.
Moreover, for any $f\in W^{m,p}_\omega(\Omega)$, let
\begin{align*}
\|f\|_{W^{m,p}_\omega(\Omega)}:=
\sum_{\alpha\in\mathbb{Z}_+^n,\,
|\alpha|\leq m}\left\|D^\alpha f\right\|_{L^p_\omega(\Omega)}.
\end{align*}
\item[\textup{(iii)}]
The \emph{homogeneous weighted Sobolev space}
$\dot{W}^{m,p}_\omega(\Omega)$
is defined to be the set of all the $f\in L^1_{\mathrm{loc}}(\Omega)$
satisfying that, for any multi-index
$\alpha\in\mathbb{Z}_+^n$ with $|\alpha|=m$,
the $\alpha^\mathrm{th}$-weak partial derivative $D^\alpha f$ of $f$
exists and $D^\alpha f\in L^p_\omega(\Omega)$.
Moreover, for any $f\in\dot{W}^{m,p}_\omega(\Omega)$, let
\begin{align*}
\|f\|_{\dot{W}^{m,p}_\omega(\Omega)}:=
\sum_{\alpha\in\mathbb{Z}_+^n,\,
|\alpha|=m}\left\|D^\alpha f\right\|_{L^p_\omega(\Omega)}.
\end{align*}
\end{enumerate}
\end{definition}

\section{Two Extension Theorems on Ball Banach Sobolev Spaces}
\label{yantuo}

The main target of this section is to establish the
extension theorems, respectively, on the inhomogeneous
ball Banach Sobolev space $W^{m,X}(\Omega)$
and the homogeneous ball Banach Sobolev space $\dot{W}^{m,X}(\Omega)$
(see Theorems~\ref{extension} and~\ref{extension2} below),
where $m\in\mathbb{N}$,
$\Omega\subset\mathbb{R}^n$ is an $(\varepsilon,\delta)$-domain,
and $X(\mathbb{R}^n)$ is a ball Banach function space
satisfying some extra mild assumptions.
To this end, we begin with recalling the concept
of $(\varepsilon,\delta)$-domains,
which can be found in \cite[p.\,73]{j1981}.

\begin{definition}\label{2121}
Let $\varepsilon\in(0,1]$ and $\delta\in(0,\infty]$.
An open set $\Omega\subset\mathbb{R}^n$
is called an \emph{$(\varepsilon,\delta)$-domain}
if, for any $x,y\in\Omega$ with $|x-y|\in(0,\delta)$,
there exists a rectifiable curve $\Gamma$ such that
\begin{enumerate}
\item[\textup{(i)}]
$\Gamma$ lies in $\Omega$;
\item[\textup{(ii)}]
$\Gamma$ connects $x$ and $y$;
\item[\textup{(iii)}]
$\ell(\Gamma)\leq\varepsilon^{-1}|x-y|$,
where $\ell(\Gamma)$ denotes
the Euclidean arc length of $\Gamma$;
\item[\textup{(iv)}]
for any point $z$ on $\Gamma$,
\begin{align*}
\mathrm{dist\,}(z,\partial\Omega)\ge\frac{\varepsilon|x-z||y-z|}{|x-y|}.
\end{align*}
\end{enumerate}
\end{definition}

\begin{remark}\label{1638}
\begin{enumerate}
\item[\textup{(i)}]
As was pointed out in \cite[p.\,73]{j1981},
both (ii) and (iii) of Definition~\ref{2121} imply that
any $(\varepsilon,\delta)$-domain
is locally connected in some quantitative sense.
\item[\textup{(ii)}]
As was pointed out in \cite[p.\,276]{dhhr2011},
$(\varepsilon,\infty)$-domains are also known
as \emph{uniform domains} which were
introduced by Martio and Sarvas \cite{ms1979}.
We refer the reader to \cite{ge,hk1991,v3} for more
studies of uniform domains.
\item[\textup{(iii)}]
As was pointed out in \cite[p.\,73]{j1981},
any Lipschitz domain
(see \cite[Definition~4.4]{eg2015} for its definition)
is an $(\varepsilon,\delta)$-domain
for some $\varepsilon\in(0,1]$ and $\delta\in(0,\infty]$;
the classical snowflake domain
(see, for instance, \cite[pp.\,104--105]{lv1973})
is an $(\varepsilon,\infty)$-domain
for some $\varepsilon\in(0,1]$ and hence $(\varepsilon,\delta)$-domains
for the same $\varepsilon\in(0,1]$ and any $\delta\in(0,\infty)$.
As was pointed out in \cite[p.\,276]{dhhr2011},
any bounded Lipschitz domain
and the half-space are $(\varepsilon,\infty)$-domains
for some $\varepsilon\in(0,1]$
and hence $(\varepsilon,\delta)$-domains
for the same $\varepsilon\in(0,1]$ and any $\delta\in(0,\infty)$.
\end{enumerate}
\end{remark}

For any $(\varepsilon,\delta)$-domain $\Omega$,
let $\{\Omega_i\}_{i}$ be the set of
the connected components of $\Omega$.
The \emph{radius} of $\Omega$ is defined by setting
$$
\mathrm{rad\,}(\Omega):=\inf_{i}\inf_{x\in\Omega_i}
\sup_{y\in\Omega_i}|x-y|;
$$
see, for instance, \cite[(5.3)]{c1992}.

\begin{remark}\label{1711}
Let $\Omega\subset\mathbb{R}^n$
be an $(\varepsilon,\infty)$-domain
for some $\varepsilon\in(0,1]$.
Then Definition~\ref{2121}(ii) implies that
$\Omega$ is connected.
Moreover,
by this, we find
$\mathrm{rad\,}(\Omega)\in(0,\infty]$.
\end{remark}

\subsection{Approximation of Functions
in Ball Banach Sobolev Spaces
by Functions\\ in $C^\infty(\overline{\Omega})$}
\label{sub3.1}

For any open set $\Omega\subset\mathbb{R}^n$,
let $C^\infty(\Omega)$ denote the space
of all infinitely differentiable functions on $\Omega$
and let $C^\infty(\overline{\Omega})$ denote the
restriction of $C^\infty(\mathbb{R}^n)$ on $\Omega$
(see, for instance, \cite[pp.\,1--2]{ma2011}
or \cite[p.\,702]{evans}).
The main target of this subsection is to
approximate functions in
the inhomogeneous ball Banach Sobolev space $W^{m,X}(\Omega)$
[or the homogeneous ball Banach Sobolev space $\dot{W}^{m,X}(\Omega)$]
by functions in $C^\infty(\overline{\Omega})$, where $m\in\mathbb{N}$
and $\Omega$ is an $(\varepsilon,\delta)$-domain.
We present the main theorem of this subsection
as follows.

\begin{theorem}\label{2247}
Let $\Omega\subset\mathbb{R}^n$ be an $(\varepsilon,\delta)$-domain
with $\varepsilon\in(0,1]$, $\delta\in(0,\infty]$,
and $\mathrm{rad\,}(\Omega)\in(0,\infty]$.
Let $X(\mathbb{R}^n)$ be a ball Banach function space
and $m\in\mathbb{N}$. Assume that
\begin{enumerate}
\item[\rm(a)]
$X(\mathbb{R}^n)$ has an absolutely continuous norm;
\item[\rm(b)]
there exists $p\in[1,\infty)$ such that
$X^\frac{1}{p}(\mathbb{R}^n)$
is a ball Banach function space
and that the Hardy--Littlewood maximal operator
is bounded on $[X^\frac{1}{p}(\mathbb{R}^n)]'$.
\end{enumerate}
Then
\begin{enumerate}
\item[\rm(i)]
$C^{\infty}(\overline{\Omega})\cap W^{m,X}(\Omega)$
is dense in $W^{m,X}(\Omega)$.
\item[\rm(ii)]
$C^{\infty}(\overline{\Omega})\cap\dot{W}^{m,X}(\Omega)$
is dense in $\dot{W}^{m,X}(\Omega)$.
\end{enumerate}
\end{theorem}

\begin{remark}
On the assumption (b) of Theorem~\ref{2247},
if $p\in(1,\infty)$, then Lorist and Nieraeth \cite[Theorem~3.1]{lz2023}
gave two equivalent characterizations on the boundedness of
the Hardy--Littlewood maximal operator $\mathcal M$ on
$[X^{\frac 1p}({\mathbb R}^n)]'$, which further implies
that $\mathcal M$ is bounded on both $X(\mathbb{R}^n)$ and $X'(\mathbb{R}^n)$.
\end{remark}

To show Theorem~\ref{2247},
we need two technical lemmas.
The following lemma
is a generalization of \cite[Lemma~4.1]{c1992}
from inhomogeneous weighted Sobolev spaces
to inhomogeneous ball Banach Sobolev spaces.
For any $\Omega\subset\mathbb{R}^n$ and $s\in(0,\infty)$,
let
\begin{align}\label{1506}
\Omega_s:=\left\{x\in\Omega:\ \mathrm{dist\,}(x,\Omega^\complement)>s\right\}.
\end{align}

\begin{lemma}\label{943}
Let $m\in\mathbb{N}$
and $X(\mathbb{R}^n)$ be a ball Banach function space
satisfying the assumptions (a) and (b) of Theorem~\ref{2247}.
Let $\xi\in C_{\mathrm{c}}^\infty(\mathbb{R}^n)$
be such that $\mathrm{supp\,}(\xi)\subset B(\mathbf{0},1)$
and $\int_{\mathbb{R}^n}\xi(x)\,dx=1$.
Then, for any given $s\in(0,\infty)$,
\begin{enumerate}
\item[\rm(i)]
for any $f\in X(\Omega)$,
$f*\xi_t\to f$ in $X(\Omega_s)$ as $t\in(0,s)$ and $t\to0^+$;
\item[\rm(ii)]
for any $f\in W^{m,X}(\Omega)$,
$f*\xi_t\to f$ in $W^{m,X}(\Omega_s)$ as $t\in(0,s)$ and $t\to0^+$,
\end{enumerate}
where $\Omega_s$ is the same as in \eqref{1506} and
$\xi_t(\cdot):=t^{-n}\xi(t^{-1}\cdot)$ for any $t\in(0,\infty)$.
\end{lemma}

To show Lemma~\ref{943}, we need the following
dominated convergence theorem on ball quasi-Banach function spaces,
which can be found in \cite[Lemma~6.3]{yhyy2022-2}; see also
\cite[Chapter 1, Proposition~3.6]{bs1988}.

\begin{lemma}\label{Ldct}
Let $X(\mathbb{R}^n)$ be a ball quasi-Banach function space
having an absolutely continuous quasi-norm.
Assume that $f\in\mathscr{M}(\mathbb{R}^n)$ and
$\{f_m\}_{m\in\mathbb{N}}\subset\mathscr{M}(\mathbb{R}^n)$
satisfy that
\begin{enumerate}
\item[\textup{(i)}]
there exists $g\in X(\mathbb{R}^n)$
such that, for any $m\in\mathbb{N}$ and
almost every $x\in\mathbb{R}^n$, $|f_m(x)|\leq|g(x)|$;
\item[\textup{(ii)}]
for almost every $x\in\mathbb{R}^n$,
$\lim_{m\to\infty}f_m(x)=f(x)$.
\end{enumerate}
Then
$\lim_{m\to\infty}\|f_m-f\|_{X(\mathbb{R}^n)}=0$.
\end{lemma}

\begin{remark}\label{LdctOmega}
Let $X(\mathbb{R}^n)$ be a ball quasi-Banach function space
having an absolutely continuous quasi-norm.
Let $\Omega\subset\mathbb{R}^n$ be an open set
and $X(\Omega)$ the restrictive space of $X(\mathbb{R}^n)$ on $\Omega$.
Assume that both (i) and (ii) of Lemma~\ref{Ldct}
with $\mathbb{R}^n$ replaced by $\Omega$ hold true
for given $f\in\mathscr{M}(\Omega)$ and
$\{f_m\}_{m\in\mathbb{N}}\subset\mathscr{M}(\Omega)$;
then, by Proposition~\ref{norm} and Lemma~\ref{Ldct},
we conclude that
$\lim_{m\to\infty}\|f_m-f\|_{X(\Omega)}=0$.
\end{remark}

Recall that, for any given $r\in(0,\infty)$,
the \emph{centered ball average operator}
$\mathcal{B}_r$ is defined by setting,
for any $f\in L_{{\mathrm{loc}}}^1(\mathbb{R}^n)$ and $x\in\mathbb{R}^n$,
\begin{align*}
\mathcal{B}_r(f)(x):=\frac{1}{|B(x,r)|}\int_{B(x,r)}\left|f(y)\right|\,dy.
\end{align*}
The following lemma gives a criterion
for the uniform boundedness of the
centered ball average operators
$\{\mathcal{B}_r\}_{r\in(0,\infty)}$ on $X(\mathbb{R}^n)$,
which is just \cite[Lemma~3.11]{dgpyyz2022}.

\begin{lemma}\label{1009}
Let $X(\mathbb{R}^n)$ be a ball Banach
function space.
Assume that there exists $p\in[1,\infty)$ such that
$X^\frac{1}{p}(\mathbb{R}^n)$
is a ball Banach function space and that the Hardy--Littlewood
maximal operator is bounded on $[X^\frac{1}{p}(\mathbb{R}^n)]'$.
Then the centered ball average operators
$\{\mathcal{B}_r\}_{r\in(0,\infty)}$
are uniformly bounded on $X(\mathbb{R}^n)$.
\end{lemma}

Now, we use Remark~\ref{LdctOmega} and
Lemma~\ref{1009} to show Lemma~\ref{943}.

\begin{proof}[Proof of Lemma~\ref{943}]
We first show (i). Let $f\in X(\Omega)$.
By $\xi\in C_{\mathrm{c}}^\infty(\mathbb{R}^n)$
and the definition of $\mathcal{B}_t(f)$,
we find that,
for any $x\in\Omega_s$ and $t\in(0,s)$,
\begin{align}\label{1016}
\left|f*\xi_t(x)\right|
&=\left|t^{-n}\int_{B(x,t)}f(y)
\xi\left(\frac{x-y}{t}\right)\,dy\right|\\
&\lesssim\left\|\xi\right\|_{L^\infty(\mathbb{R}^n)}
\fint_{B(x,t)}|f(y)|\,dy
\sim\mathcal{B}_t(f)(x).\nonumber
\end{align}
From the assumption of the present lemma
and Lemma~\ref{1009},
we infer that the
centered ball average operators
$\{\mathcal{B}_r\}_{r\in(0,\infty)}$
are uniformly bounded on $X(\mathbb{R}^n)$.
By this and $f\in X(\Omega)$,
we find that $\mathcal{B}_t(f)\in X(\Omega_s)$.
On the other hand,
from $\int_{\mathbb{R}^n}\xi(x)\,dx=1$,
the Lebesgue differentiation theorem
(see, for instantce, \cite[(2.10)]{duo}),
and $f\in X(\Omega)$
[which implies that $f\in L_{\mathrm{loc}}^1(\Omega)$;
see Proposition~\ref{1533}(vi)],
we deduce that, for almost every $x\in\Omega_s$,
\begin{align*}
\left|f*\xi_t(x)-f(x)\right|
&=\left|t^{-n}\int_{B(x,t)}\left[f(y)-f(x)\right]
\xi\left(\frac{x-y}{t}\right)\,dy\right|\\
&\lesssim\|\xi\|_{L^\infty(\mathbb{R}^n)}
\fint_{B(x,t)}\left|f(y)-f(x)\right|\,dy\to0
\end{align*}
as $t\to0^+$.
By this, Remark~\ref{LdctOmega},
\eqref{1016}, and the proven conclusion that
$\mathcal{B}_t(f)\in X(\Omega_s)$,
we conclude that
\begin{align}\label{955}
\lim_{t\to0^+}\left\|f*\xi_t-f\right\|_{X(\Omega_s)}=0.
\end{align}
This finishes the proof of (i).

Next, we show (ii). Let $f\in W^{m,X}(\Omega)$.
From \cite[Proposition~8.10]{folland},
we infer that, for any $\alpha\in\mathbb{Z}_+^n$
with $|\alpha|\leq m$, $t\in(0,s)$, and $x\in\Omega_s$,
$$
D^\alpha(f*\xi_t)(x)=D^\alpha f*\xi_t(x).
$$
By this and an argument similar to that used
in the estimation of \eqref{955} with $f$
replaced by $D^\alpha f$, we conclude that
\begin{align*}
&\lim_{t\to0^+}\left\|f*\xi_t-f\right\|_{W^{m,X}(\Omega_s)}\\
&\quad=\lim_{t\to0^+}\sum_{\alpha\in\mathbb{Z}_+^n,\,
|\alpha|\leq m}\left\|D^\alpha(f*\xi_t)-D^\alpha f\right\|_{X(\Omega_s)}\\
&\quad=\lim_{t\to0^+}\sum_{\alpha\in\mathbb{Z}_+^n,\,
|\alpha|\leq m}\left\|D^\alpha f*\xi_t-D^\alpha f\right\|_{X(\Omega_s)}
=0.
\end{align*}
This finishes the proof of (ii) and hence Lemma~\ref{943}.
\end{proof}

The following extrapolation lemma is a simple
corollary of \cite[Lemma~4.7 and Remark~4.8]{dlyyz.arxiv}.

\begin{lemma}\label{4.6}
Let $X(\mathbb{R}^n)$ be a ball Banach
function space and $p\in[1,\infty)$.
Assume that $X^\frac{1}{p}(\mathbb{R}^n)$ is a
ball Banach function space and that
the Hardy--Littlewood maximal operator
$\mathcal{M}$ is bounded on $[X^\frac{1}{p}(\mathbb{R}^n)]'$
with its operator norm denoted by
$\|\mathcal{M}\|_{[X^\frac{1}{p}(\mathbb{R}^n)]'
\to[X^\frac{1}{p}(\mathbb{R}^n)]'}$.
Then, for any $f\in X(\mathbb{R}^n)$,
\begin{align}\label{2120}
\|f\|_{X(\mathbb{R}^n)}\leq\sup_{\|g\|_{[X^\frac{1}{p}(\mathbb{R}^n)]'}=1}
\left[\int_{\mathbb{R}^n}
\left|f(x)\right|^pR_{[X^\frac{1}{p}(\mathbb{R}^n)]'}
g(x)\,dx\right]^\frac{1}{p}
\leq2^\frac{1}{p}\|f\|_{X(\mathbb{R}^n)},
\end{align}
where, for any $g\in[X^\frac{1}{p}(\mathbb{R}^n)]'$,
\begin{align*}
R_{[X^\frac{1}{p}(\mathbb{R}^n)]'}g:=\sum_{k=0}^\infty
\frac{\mathcal{M}^kg}{2^k\|\mathcal{M}\|^k_{
[X^\frac{1}{p}(\mathbb{R}^n)]'\to
[X^\frac{1}{p}(\mathbb{R}^n)]'}}\in A_1(\mathbb{R}^n)
\end{align*}
and
$$
\left[R_{[X^\frac{1}{p}(\mathbb{R}^n)]'}g\right]_{A_1(\mathbb{R}^n)}\leq2
\|\mathcal{M}\|_{[X^\frac{1}{p}(\mathbb{R}^n)]'\to
[X^\frac{1}{p}(\mathbb{R}^n)]'}
$$
and where, for any $k\in\mathbb{N}$,
$\mathcal{M}^k$ is the $k$ iterations of $\mathcal{M}$
and $\mathcal{M}^0g:=|g|$.
\end{lemma}

\begin{remark}\label{waicha}
Let all the symbols be the same as in Lemma~\ref{4.6}.
Let $\Omega$ be an open subset of $\mathbb{R}^n$
and $X(\Omega)$ the restrictive space of $X(\mathbb{R}^n)$ on $\Omega$.
Then, by \eqref{2120} and Proposition~\ref{norm},
we conclude that, for any $f\in X(\Omega)$,
\begin{align*}
\|f\|_{X(\Omega)}\leq\sup_{\|g\|_{[X^\frac{1}{p}(\mathbb{R}^n)]'}=1}
\left[\int_{\Omega}
\left|f(x)\right|^pR_{[X^\frac{1}{p}(\mathbb{R}^n)]'}
g(x)\,dx\right]^\frac{1}{p}
\leq2^\frac{1}{p}\|f\|_{X(\Omega)}.
\end{align*}
\end{remark}

Now, we use Lemmas~\ref{943} and~\ref{4.6} to show Theorem~\ref{2247}.

\begin{proof}[Proof of Theorem~\ref{2247}]
We only show (i) by borrowing some ideas
from the proof of \cite[Theorem~6.1]{c1992}
with $\{E^{p_i}_{\omega_i,k_i}(\mathcal{D})\}_{i=0}^N$
replaced by $\{\dot{W}^{i,p}_\omega(\Omega)\}_{i=0}^m$
and hence $\bigcap_{i=0}^NE^{p_i}_{\omega_i,k_i}(\mathcal{D})$ becomes
$$
\bigcap_{i=0}^m\dot{W}^{i,p}_\omega(\Omega)
=W^{m,p}_\omega(\Omega),
$$
where $\omega\in A_1(\mathbb{R}^n)$.

Let $f\in W^{m,X}(\Omega)$ and fix $\eta\in(0,\infty)$.
From Remark~\ref{waicha},
we deduce that
\begin{align}\label{1944}
\|f\|_{W^{m,X}(\Omega)}&=\sum_{|\alpha|\leq m}
\left\|D^\alpha f\right\|_{X(\Omega)}
\sim\left\|\sum_{|\alpha|\leq m}\left|D^\alpha f\right|\right\|_{X(\Omega)}\\
&\sim\sup_{\|g\|_{[X^\frac{1}{p}(\mathbb{R}^n)]'}=1}
\sum_{|\alpha|\leq m}\left[\int_{\Omega}
\left|D^\alpha f(x)\right|^pR_{[X^\frac{1}{p}(\mathbb{R}^n)]'}
g(x)\,dx\right]^\frac{1}{p},\nonumber
\end{align}
which, combined with Lemma~\ref{4.6},
further implies that $f\in W^{m,p}_\omega(\Omega)$
with $\omega:=R_{[X^\frac{1}{p}(\mathbb{R}^n)]'}g$
for any $g\in[X^\frac{1}{p}(\mathbb{R}^n)]'$
satisfying $\|g\|_{[X^\frac{1}{p}(\mathbb{R}^n)]'}=1$.
Let $\mathfrak{R}:=\{R_j\}_{j\in I}$ and
$\{\varphi_j\}_{j\in I}\subset C_{\mathrm{c}}^\infty(\mathbb{R}^n)$
be the same as in \cite[p.\,1049]{c1992}
with $\mathcal{D}$ replaced by $\Omega$,
where $I$ is a set of indices.
For any $s\in(0,\infty)$,
let $\Omega_s$ be the same as in \eqref{1506} and,
for any set $K\subset\mathbb{R}^n$,
$K^s:=\{x+y:\ x\in K,\,y\in B(\mathbf{0},s)\}$.
By the inner regularity of the Lebesgue measure
(see, for instance, \cite[Theorem~2.14(d)]{rudin}),
the assumption that $X(\mathbb{R}^n)$ has an
absolutely norm,
and \cite[Lemma~5.6.14]{lyh2320}
(see also \cite[p.\,16, Proposition~3.6]{bs1988}),
we find that there exists a compact set $K_{(X)}\subset\Omega$
such that
\begin{align}\label{6.21}
\left\|f\right\|_{W^{m,X}(\Omega\setminus K_{(X)})}\leq\eta.
\end{align}
Fix $s\in(0,1)$ such that
\begin{align}\label{1530}
K^{3s/2}_{(X)}\subset\Omega_{s/2}
\end{align}
because $\Omega$ is open and
$K_{(X)}$ is a compact subset of $\Omega$.
Next, choose $\psi_{(X)}\in C_{\mathrm{c}}^\infty(\mathbb{R}^n)$
such that, for any $\alpha\in\mathbb{Z}_+^n$
with $|\alpha|\leq m$,
$$
\mathbf{1}_{K^s_{(X)}}\leq\psi_{(X)}
\leq\mathbf{1}_{K^{3s/2}_{(X)}}
\ \text{and}\ \left|D^\alpha\psi_{(X)}\right|
\leq C_{(\alpha)}s^{-|\alpha|},
$$
where the positive constant $C_{(\alpha)}$ depends on $\alpha$.
Let $\xi\in C_{\mathrm{c}}^\infty(\mathbb{R}^n)$
and $\xi_t$ be the same as in Lemma~\ref{943}.
From \eqref{1530} and
Lemma~\ref{943}(ii), we deduce that there exists $t\in(0,s/2)$
such that
\begin{align}\label{6.22}
\left\|f-f*\xi_t\right\|_{W^{m,X}(K^{3s/2})}
\leq\left\|f-f*\xi_t\right\|_{W^{m,X}(\Omega_{s/2})}
\leq\eta s^m.
\end{align}
For any $j\in I$, let
$P_j:=\pi_v^m(R_j)f$,
where the projection
$\pi_v^m(R_j):\ W^{m,X}(\Omega)\to\mathcal{P}_{m-1}(\Omega)$
[here and thereafter,
we denote by $\mathcal{P}_{m-1}(\Omega)$ the set of all the
polynomials of degree not higher than $m-1$ on $\Omega$]
is the same as in \cite[(4.2)]{c1992} with $k$ and $Q$ replaced,
respectively, by $m$ and $R_j$.
Let
\begin{align*}
f_\eta:=\left(\sum_{j\in I}P_j\varphi_j\right)(1-\psi_{(X)})
+\left(f*\xi_t\right)\psi_{(X)}.
\end{align*}
Then it is easy to show that $f_\eta\in C^\infty(\overline{\Omega})$
because $P_j$ for any $j\in I$ is a polynomial of degree
not higher than $m-1$ and $\varphi_j,\psi_{(X)},
(f*\xi_t)\psi_{(X)}\in C^\infty(\overline{\Omega})$.
By Remark~\ref{waicha}, \eqref{6.21}, \eqref{6.22},
and \cite[Proposition~7.1.5(6)]{g2014},
we conclude that,
for any $g\in[X^\frac{1}{p}(\mathbb{R}^n)]'$ with
$\|g\|_{[X^\frac{1}{p}(\mathbb{R}^n)]'}=1$
and for any $\alpha\in\mathbb{Z}_+^n$
with $l:=|\alpha|\in\mathbb{N}\cap[0,m]$,
\begin{align}\label{621}
\left\|D^\alpha f\right\|_{L^p_{\omega_g}(\Omega\setminus K_{(X)})}
&\leq\left\|f\right\|_{W^{m,p}_{\omega_g}(\Omega\setminus K_{(X)})}\\
&\leq\sup_{\|g\|_{[X^\frac{1}{p}(\mathbb{R}^n)]'}=1}
\sum_{\beta\in\mathbb{Z}_+^n,\,|\beta|\leq m}
\left[\int_{\Omega\setminus K_{(X)}}
\left|D^\beta f(x)\right|^p\omega_g(x)\,dx\right]^\frac{1}{p}
\nonumber\\
&\sim\left\|f\right\|_{W^{m,X}(\Omega\setminus K_{(X)})}\leq\eta
\nonumber
\end{align}
and
\begin{align}\label{622}
&\left\|f-f*\xi_t\right\|_{W^{l,p}_{\omega_g}(K^{3s/2}_{(X)})}\\
&\quad\leq\left\|f-f*\xi_t\right\|_{W^{m,p}_{\omega_g}(K^{3s/2}_{(X)})}
\nonumber\\
&\quad\leq\sup_{\|g\|_{[X^\frac{1}{p}(\mathbb{R}^n)]'}=1}
\sum_{\beta\in\mathbb{Z}_+^n,\,|\beta|\leq m}
\left[\int_{K^{3s/2}_{(X)}}
\left|D^\beta(f-f*\xi_t)(x)\right|^p\omega_g(x)\,dx\right]^\frac{1}{p}
\nonumber\\
&\quad\sim\left\|f-f*\xi_t\right\|_{W^{m,X}(K^{3s/2}_{(X)})}
\leq\eta s^m\leq\eta s^l,\nonumber
\end{align}
where
\begin{align}\label{omegag1}
\omega_g:=R_{[X^\frac{1}{p}(\mathbb{R}^n)]'}g\in A_1(\mathbb{R}^n)
\subset A_p(\mathbb{R}^n)
\end{align}
with
\begin{align}\label{omegag2}
[\omega_g]_{A_p(\mathbb{R}^n)}
\leq[\omega_g]_{A_1(\mathbb{R}^n)}\leq2
\|\mathcal{M}\|_{[X^\frac{1}{p}(\mathbb{R}^n)]'\to
[X^\frac{1}{p}(\mathbb{R}^n)]'}.
\end{align}
From this, the proven conclusion that $f\in W^{m,p}_\omega(\Omega)$
with $\omega:=\omega_g$,
and an argument similar to that used
in the proof of \cite[Theorem~6.1]{c1992}
(to be more precise, the estimation of the inequality
between \cite[(6.1)]{c1992} and \cite[(6.2)]{c1992})
with \cite[(6.2)]{c1992}
replaced by both \eqref{621} and \eqref{622},
we deduce that,
for any $g\in[X^\frac{1}{p}(\mathbb{R}^n)]'$ with
$\|g\|_{[X^\frac{1}{p}(\mathbb{R}^n)]'}=1$
and for any $\alpha\in\mathbb{Z}_+^n$
with $l:=|\alpha|\in\mathbb{N}\cap[0,m]$,
\begin{align*}
\left\|D^\alpha\left(f-f_\eta\right)\right\|_{L^p_{\omega_g}(\Omega)}
\lesssim\eta,
\end{align*}
where the implicit positive constant is independent of $\omega_g$
but depends on $\|\mathcal{M}\|_{[X^\frac{1}{p}(\mathbb{R}^n)]'\to
[X^\frac{1}{p}(\mathbb{R}^n)]'}$.
By this and Remark~\ref{waicha},
we conclude that
\begin{align*}
\left\|f-f_\eta\right\|_{W^{m,X}(\Omega)}
&=\sum_{\alpha\in\mathbb{Z}_+^n,\,|\alpha|\leq m}
\left\|D^\alpha\left(f-f_\eta\right)\right\|_{X(\Omega)}\\
&\leq\sum_{\alpha\in\mathbb{Z}_+^n,\,|\alpha|\leq m}
\sup_{\|g\|_{[X^\frac{1}{p}(\mathbb{R}^n)]'}=1}
\left[\int_{\Omega}
\left|D^\alpha\left(f-f_\eta\right)(x)\right|^p
\omega_g(x)\,dx\right]^\frac{1}{p}\\
&\lesssim\sum_{\alpha\in\mathbb{Z}_+^n,\,|\alpha|\leq m}
\eta\sim\eta,
\end{align*}
which completes the proof of (i).

Since the proof of (ii) is
similar to that of (i), which uses an argument
similar to that used in the proof of \cite[Theorem~6.1]{c1992}
with $\{E^{p_i}_{\omega_i,k_i}(\mathcal{D})\}_{i=0}^N$
replaced by $\dot{W}^{m,p}_{\omega}(\Omega)$,
we omit the details.
This finishes the proof of Theorem~\ref{2247}.
\end{proof}

\subsection{Extension Theorem on Inhomogeneous
Ball Banach Sobolev Spaces}
\label{sub3.2}

In this subsection,
we establish the following extension theorem
on inhomogeneous ball Banach Sobolev spaces.

\begin{theorem}\label{extension}
Let $\Omega\subset\mathbb{R}^n$ be an $(\varepsilon,\delta)$-domain
with $\varepsilon\in(0,1]$, $\delta\in(0,\infty]$,
and $\mathrm{rad\,}(\Omega)\in(0,\infty]$.
Let $m\in\mathbb{N}$
and $X(\mathbb{R}^n)$ be a ball Banach function space
satisfying the assumptions (a) and (b) of Theorem~\ref{2247}.
Then there exists a linear extension operator
$$
\Lambda:\ W^{m,X}(\Omega)\to W^{m,X}(\mathbb{R}^n)
$$
such that,
for any $f\in W^{m,X}(\Omega)$,
\begin{enumerate}
\item[\textup{(i)}]
$\Lambda(f)=f$ almost everywhere in $\Omega$;
\item[\textup{(ii)}]
$\|\Lambda(f)\|_{W^{m,X}(\mathbb{R}^n)}
\leq C\|f\|_{W^{m,X}(\Omega)}$,
where the positive constant $C$ is independent of $f$.
\end{enumerate}
\end{theorem}

To show Theorem~\ref{extension},
we need the following extension lemma on
inhomogeneous weighted Sobolev spaces,
which can be found in \cite[Theorem~1.1]{c1992}.

\begin{lemma}\label{chua1992}
Let $\Omega\subset\mathbb{R}^n$ be an $(\varepsilon,\delta)$-domain
with $\varepsilon\in(0,1]$, $\delta\in(0,\infty]$,
and $\mathrm{rad\,}(\Omega)\in(0,\infty]$.
Let $m\in\mathbb{N}$,
$p\in[1,\infty)$, and $\omega\in A_p(\mathbb{R}^n)$.
Then there exists a linear extension operator
$$
\Lambda:\ W^{m,p}_\omega(\Omega)\to W^{m,p}_\omega(\mathbb{R}^n)
$$
such that, for any $f\in W^{m,p}_\omega(\Omega)$,
$\Lambda(f)=f$ almost everywhere in $\Omega$ and
$$
\left\|\Lambda(f)\right\|_{W^{m,p}_\omega(\mathbb{R}^n)}
\leq C\left\|f\right\|_{W^{m,p}_\omega(\Omega)},
$$
where the positive constant $C$ depends only on $\varepsilon$, $\delta$,
$[\omega]_{A_p(\mathbb{R}^n)}$,
$m$, $p$, $n$, and $\mathrm{rad\,}(\Omega)$.
\end{lemma}

\begin{proof}[Proof of Theorem~\ref{extension}]
First assume that $f\in C^\infty(\overline{\Omega})\cap W^{m,X}(\Omega)$.
By Lemma~\ref{4.6},
Remark~\ref{waicha},
and an argument similar
to that used in the estimation of \eqref{1944},
we find that,
for any given $g\in[X^\frac{1}{p}(\mathbb{R}^n)]'$
with $\|g\|_{[X^\frac{1}{p}(\mathbb{R}^n)]'}=1$,
$f\in W^{m,p}_{\omega_g}(\Omega)$,
where $\omega_g$ is the same as in \eqref{omegag1}
satisfying \eqref{omegag2}.
Let $\Lambda$ be the same extension operator
as in \cite[Theorem~1.1]{c1992}.
The proof of \cite[Theorem~1.1]{c1992}
shows that $\Lambda(f)$
for any $f\in C^\infty(\overline{\Omega})\cap W^{m,X}(\Omega)$
is independent of $\omega$
(see \cite[p.\,1056]{c1992}).
Then, from Lemma~\ref{chua1992} and \eqref{omegag2},
we deduce that, for any given $g\in[X^\frac{1}{p}(\mathbb{R}^n)]'$
with $\|g\|_{[X^\frac{1}{p}(\mathbb{R}^n)]'}=1$,
\begin{align}\label{re}
\Lambda(f)=f
\end{align}
almost everywhere in $\Omega$
and
$$
\left\|\Lambda(f)\right\|_{W^{m,p}_{\omega_g}(\mathbb{R}^n)}
\lesssim\left\|f\right\|_{W^{m,p}_{\omega_g}(\Omega)},
$$
where the implicit positive constant
is independent of $g$ but depends on
$\|\mathcal{M}\|_{[X^\frac{1}{p}(\mathbb{R}^n)]'\to
[X^\frac{1}{p}(\mathbb{R}^n)]'}$. By this, Remark~\ref{waicha},
and \eqref{1944},
we conclude that
\begin{align}\label{1550}
\left\|\Lambda(f)\right\|_{W^{m,X}(\mathbb{R}^n)}
&=\sum_{|\alpha|\leq m}\left\|D^\alpha(\Lambda(f))\right\|_{X(\Omega)}
\sim\left\|\sum_{|\alpha|\leq m}
\left|D^\alpha(\Lambda(f))\right|\,\right\|_{X(\Omega)}\\
&\sim\sup_{\|g\|_{[X^\frac{1}{p}(\mathbb{R}^n)]'}=1}
\sum_{|\alpha|\leq m}\left[\int_{\Omega}
\left|D^\alpha(\Lambda(f))(x)\right|^pR_{[X^\frac{1}{p}(\mathbb{R}^n)]'}
g(x)\,dx\right]^\frac{1}{p}\nonumber\\
&=\sup_{\|g\|_{[X^\frac{1}{p}(\mathbb{R}^n)]'}=1}
\sum_{|\alpha|\leq m}
\left\|D^\alpha(\Lambda(f))\right\|_{L^p_{\omega_g}(\mathbb{R}^n)}\nonumber\\
&=\sup_{\|g\|_{[X^\frac{1}{p}(\mathbb{R}^n)]'}=1}
\left\|\Lambda(f)\right\|_{W^{m,p}_{\omega_g}(\mathbb{R}^n)}
\lesssim\sup_{\|g\|_{[X^\frac{1}{p}(\mathbb{R}^n)]'}=1}
\left\|f\right\|_{W^{m,p}_{\omega_g}(\Omega)}\nonumber\\
&\sim\left\|f\right\|_{W^{m,X}(\Omega)},\nonumber
\end{align}
where the implicit positive constants depend only
on $\varepsilon$, $\delta$,
$\|\mathcal{M}\|_{[X^\frac{1}{p}(\mathbb{R}^n)]'\to
[X^\frac{1}{p}(\mathbb{R}^n)]'}$, $m$,
$p$, $n$, and $\mathrm{rad\,}(\Omega)$,
which, combined with \eqref{re},
further implies that the desired conclusion of
the present theorem holds true for any
$f\in C^\infty(\overline{\Omega})\cap W^{m,X}(\Omega)$.

Now, let $f\in W^{m,X}(\Omega)$.
From Theorem~\ref{2247},
we deduce that there exists a sequence
$\{f_k\}_{k\in\mathbb{N}}\subset C^\infty(\overline{\Omega})\cap W^{m,X}(\Omega)$
such that
\begin{align}\label{1604}
\left\|f-f_k\right\|_{W^{m,X}(\Omega)}\to0
\end{align}
as $k\to\infty$.
Using \eqref{1550} with $f:=f_s-f_k$,
we find that, for any $k,s\in\mathbb{N}$,
\begin{align*}
\left\|\Lambda(f_k)-\Lambda(f_s)\right\|_{W^{m,X}(\mathbb{R}^n)}
\lesssim\left\|f_k-f_s\right\|_{W^{m,X}(\Omega)},
\end{align*}
which, combined with \eqref{1604}, further
implies that $\{\Lambda(f_k)\}_{k\in\mathbb{N}}$
is a Cauchy sequence in $W^{m,X}(\mathbb{R}^n)$.
This and the completeness of $W^{m,X}(\mathbb{R}^n)$
[which can be deduced from an argument similar to
that used in the proof of \cite[p.\,262, Theorem 2]{evans}
with $L^p(U)$ replaced by $X(\mathbb{R}^n)$]
imply that there exists a unique function in $W^{m,X}(\mathbb{R}^n)$,
denoted by $\Lambda(f)$,
such that
\begin{align}\label{1610}
\left\|\Lambda(f)-\Lambda(f_k)\right\|_{W^{m,X}(\mathbb{R}^n)}\to0
\end{align}
as $k\to\infty$.
From this, Definition~\ref{1659}(v),
\eqref{1550}, and \eqref{1604},
we infer that
\begin{align*}
\left\|\Lambda(f)\right\|_{W^{m,X}(\mathbb{R}^n)}
&\leq\left\|\Lambda(f)-\Lambda(f_k)\right\|_{W^{m,X}(\mathbb{R}^n)}
+\left\|\Lambda(f_k)\right\|_{W^{m,X}(\mathbb{R}^n)}\\
&\lesssim\left\|\Lambda(f)-\Lambda(f_k)\right\|_{W^{m,X}(\mathbb{R}^n)}
+\left\|f_k\right\|_{W^{m,X}(\Omega)}\\
&\leq\left\|\Lambda(f)-\Lambda(f_k)\right\|_{W^{m,X}(\mathbb{R}^n)}
+\left\|f_k-f\right\|_{W^{m,X}(\Omega)}+
\left\|f\right\|_{W^{m,X}(\Omega)}\\
&\to\left\|f\right\|_{W^{m,X}(\Omega)}
\end{align*}
as $k\to\infty$.
Thus, $\Lambda(f)$ satisfies (ii) of the present theorem.
Moreover, by Proposition~\ref{1533}(v),
\eqref{1035},
\eqref{re}, \eqref{1610}, and \eqref{1604},
we conclude that
\begin{align*}
\left\|\Lambda(f)|_\Omega-f\right\|_{X(\Omega)}
&\leq\left\|\Lambda(f)|_\Omega-\Lambda(f_k)|_\Omega\right\|_{X(\Omega)}
+\left\|\Lambda(f_k)|_\Omega-f\right\|_{X(\Omega)}\\
&\leq\left\|\Lambda(f)-\Lambda(f_k)\right\|_{X(\mathbb{R}^n)}
+\left\|f_k-f\right\|_{X(\Omega)}\\
&\leq\left\|\Lambda(f)-\Lambda(f_k)\right\|_{W^{m,X}(\mathbb{R}^n)}
+\left\|f_k-f\right\|_{W^{m,X}(\Omega)}\to0
\end{align*}
as $k\to\infty$,
which, combined with Proposition~\ref{1533}(i),
further implies that $\Lambda(f)|_\Omega=f$ almost everywhere in $\Omega$.
Thus, $\Lambda(f)$ satisfies (i) of the present theorem.
This finishes the proof of Theorem~\ref{extension}.
\end{proof}

As a corollary of Theorem~\ref{extension},
we have the following conclusion.

\begin{corollary}\label{restr}
Let $\Omega\subset\mathbb{R}^n$ be an $(\varepsilon,\delta)$-domain
with $\varepsilon\in(0,1]$, $\delta\in(0,\infty]$,
and $\mathrm{rad\,}(\Omega)\in(0,\infty]$.
Let $m\in\mathbb{N}$
and $X(\mathbb{R}^n)$ be a ball Banach function space
satisfying the assumptions (a) and (b) of Theorem~\ref{2247}.
Then $f\in W^{m,X}(\Omega)$ if and only if there exists
$g\in W^{m,X}(\mathbb{R}^n)$ such that
$g=f$ almost everywhere in $\Omega$;
moreover, there exists a positive constant $C$,
independent of both $f$ and $g$, such that
$$
\left\|f\right\|_{W^{m,X}(\Omega)}
\leq\left\|g\right\|_{W^{m,X}(\mathbb{R}^n)}
\leq C\left\|f\right\|_{W^{m,X}(\Omega)}.
$$
\end{corollary}

\begin{proof}
We first show the necessity.
By Theorem~\ref{extension},
we find that there exists a linear extension operator
$\Lambda$ such that, for any $f\in W^{m,X}(\Omega)$,
$\Lambda(f)\in W^{m,X}(\mathbb{R}^n)$,
$\Lambda(f)=f$ almost everywhere in $\Omega$,
and
$$
\left\|\Lambda(f)\right\|_{W^{m,X}(\mathbb{R}^n)}
\lesssim\|f\|_{W^{m,X}(\Omega)}.
$$
This finishes the proof of the necessity.

Next, we show the sufficiency.
Let $f\in W^{m,X}(\mathbb{R}^n)$.
Then, for any
$\alpha\in\mathbb{Z}_+^n$ with $|\alpha|\leq m$,
we have $D^\alpha f\in X(\mathbb{R}^n)$.
Let $f^{(\alpha)}:=D^\alpha f|_\Omega$
be the restriction of $D^\alpha f$ on $\Omega$;
particular, $f^{(\mathbf{0})}:=f|_\Omega$.
From this and \eqref{1035}, we infer that,
for any
$\alpha\in\mathbb{Z}_+^n$ with $|\alpha|\leq m$,
\begin{align}\label{2110}
\left\|f^{(\alpha)}\right\|_{X(\Omega)}\leq
\left\|D^\alpha f\right\|_{X(\mathbb{R}^n)}.
\end{align}
On the other hand, by \eqref{wpd},
the fact that $\widetilde{\phi}\in C_{\mathrm{c}}^\infty(\mathbb{R}^n)$
for any $\phi\in C_{\mathrm{c}}^\infty(\Omega)$,
and the definition of $f^{(\alpha)}$, we conclude that,
for any $\phi\in C_{\mathrm{c}}^\infty(\Omega)$
and $\alpha\in\mathbb{Z}_+^n$ with $|\alpha|\leq m$,
\begin{align*}
\int_{\Omega}f^{(\mathbf{0})}(x)D^\alpha\phi(x)\,dx
&=\int_{\mathbb{R}^n}f(x)\widetilde{D^\alpha\phi}(x)\,dx
=\int_{\mathbb{R}^n}f(x)D^\alpha\widetilde{\phi}(x)\,dx\\
&=(-1)^{|\alpha|}\int_{\mathbb{R}^n}D^\alpha f(x)\widetilde{\phi}(x)\,dx\\
&=(-1)^{|\alpha|}\int_{\Omega}f^{(\alpha)}(x)\phi(x)\,dx,
\end{align*}
where $\widetilde{D^\alpha f}$ and $\widetilde{\phi}$
are defined the same as in \eqref{1448} with $f$ replaced by
$D^\alpha f$ and $\phi$,
which further implies that $f^{(\alpha)}$
is just the $\alpha^\mathrm{th}$-weak partial derivative
of $f^{(\mathbf{0})}$.
From this and \eqref{2110}, we deduce that
\begin{align*}
\left\|f^{(\mathbf{0})}\right\|_{W^{m,X}(\Omega)}
&=\sum_{|\alpha|\leq m}
\left\|D^\alpha(f^{(\mathbf{0})})\right\|_{X(\Omega)}
=\sum_{|\alpha|\leq m}\left\|f^{(\alpha)}\right\|_{X(\Omega)}\\
&\leq\sum_{|\alpha|\leq m}\left\|D^\alpha f\right\|_{X(\mathbb{R}^n)}
=\left\|f\right\|_{W^{m,X}(\mathbb{R}^n)}<\infty.\nonumber
\end{align*}
Thus, $f^{(\mathbf{0})}\in W^{m,X}(\Omega)$.
This finishes the proof of the
sufficiency and hence Corollary~\ref{restr}.
\end{proof}

\subsection{Extension Theorem on Homogeneous
Ball Banach Sobolev Spaces}
\label{sub3.3}

In this subsection,
we establish the following extension theorem
on homogeneous ball Banach Sobolev spaces.

\begin{theorem}\label{extension2}
Let $\Omega\subset\mathbb{R}^n$ be an $(\varepsilon,\infty)$-domain
with $\varepsilon\in(0,1]$.
Let $m\in\mathbb{N}$
and $X(\mathbb{R}^n)$ be a ball Banach function space
satisfying the assumptions (a) and (b) of Theorem~\ref{2247}.
Then there exists a linear extension operator
$$
\Lambda:\ \dot{W}^{m,X}(\Omega)\to\dot{W}^{m,X}(\mathbb{R}^n)
$$
such that,
for any $f\in\dot{W}^{m,X}(\Omega)$,
\begin{enumerate}
\item[\textup{(i)}]
$\Lambda(f)=f$ almost everywhere in $\Omega$;
\item[\textup{(ii)}]
$\|\Lambda(f)\|_{\dot{W}^{m,X}(\mathbb{R}^n)}
\leq C\|f\|_{\dot{W}^{m,X}(\Omega)}$,
where the positive constant $C$ is independent of $f$.
\end{enumerate}
\end{theorem}

To show Theorem~\ref{extension2}, we need
the following extension lemma on homogeneous
weighted Sobolev spaces,
which is just \cite[Theorem~1.2]{c1992}.

\begin{lemma}\label{chua1992'}
Let $\Omega\subset\mathbb{R}^n$ be an $(\varepsilon,\infty)$-domain
with $\varepsilon\in(0,1]$.
Let $m\in\mathbb{N}$,
$p\in[1,\infty)$, and $\omega\in A_p(\mathbb{R}^n)$.
Then there exists a linear extension operator
$$
\Lambda:\ \dot{W}^{m,p}_\omega(\Omega)\to\dot{W}^{m,p}_\omega(\mathbb{R}^n)
$$
such that, for any $f\in\dot{W}^{m,p}_\omega(\Omega)$,
$\Lambda(f)=f$ almost everywhere in $\Omega$ and
$$
\left\|\Lambda(f)\right\|_{\dot{W}^{m,p}_\omega(\mathbb{R}^n)}
\leq C\left\|f\right\|_{\dot{W}^{m,p}_\omega(\Omega)},
$$
where the positive constant $C$ depends only on $\varepsilon$,
$[\omega]_{A_p(\mathbb{R}^n)}$,
$m$, $p$, and $n$.
\end{lemma}

\begin{proof}[Proof of Theorem~\ref{extension2}]
First assume that $f\in C^\infty(\overline{\Omega})\cap
\dot{W}^{m,X}(\Omega)$. By
Lemma~\ref{4.6}, Remark~\ref{waicha},
and an argument similar
to that used in the estimation of \eqref{1944},
we find that, for any given $g\in[X^\frac{1}{p}(\mathbb{R}^n)]'$
with $\|g\|_{[X^\frac{1}{p}(\mathbb{R}^n)]'}=1$,
$f\in\dot{W}^{m,p}_{\omega_g}(\Omega)$,
where $\omega_g$ is the same as in \eqref{omegag1}
satisfying \eqref{omegag2}.
Let $\Lambda$ be the extension operator
same as in \cite[Theorem~1.2]{c1992}.
The proof of \cite[Theorem~1.2]{c1992}
shows that $\Lambda(f)$
for any $f\in C^\infty(\overline{\Omega})\cap
\dot{W}^{m,X}(\Omega)$
is independent of $\omega$;
see \cite[p.\,1061]{c1992} for the
definition of $\Lambda(f)$ in the case that $\Omega$
is unbounded and \cite[p.\,1062]{c1992} for the
definition of $\Lambda(f)$ in the case that $\Omega$
is bounded.
Then, from Lemma~\ref{chua1992'} and an argument similar
to that used in the estimation of \eqref{1550},
we deduce that both (i) and (ii) of the present theorem
hold true for any $f\in C^\infty(\overline{\Omega})\cap
\dot{W}^{m,X}(\Omega)$.

Next, let $f\in\dot{W}^{m,X}(\Omega)$.
By Theorem~\ref{2247},
we find that there exists a sequence
$\{f_k\}_{k\in\mathbb{N}}\subset C^\infty(\overline{\Omega})
\cap\dot{W}^{m,X}(\Omega)$
such that
\begin{align}\label{1611}
\left\|f-f_k\right\|_{\dot{W}^{m,X}(\Omega)}\to0
\end{align}
as $k\to\infty$.
From the proven conclusion
[(ii) of the present theorem for $f_s-f_k$],
we infer that, for any $k,s\in\mathbb{N}$,
\begin{align*}
\left\|\Lambda(f_k)-\Lambda(f_s)\right\|_{\dot{W}^{m,X}(\mathbb{R}^n)}
\lesssim\left\|f_k-f_s\right\|_{\dot{W}^{m,X}(\Omega)}.
\end{align*}
By this and \eqref{1611}, we find that
$\{\Lambda(f_k)\}_{k\in\mathbb{N}}$
is a Cauchy sequence in $\dot{W}^{m,X}(\mathbb{R}^n)/
\mathcal{P}_{m-1}(\mathbb{R}^n)$,
where
$\dot{W}^{m,X}(\mathbb{R}^n)/\mathcal{P}_{m-1}(\mathbb{R}^n)$
is the \emph{quotient space}
equipped with the norm $\|\cdot\|_{\dot{W}^{m,X}(\mathbb{R}^n)}$.
This and the completeness of
$$
\left(\dot{W}^{m,X}(\mathbb{R}^n)/\mathcal{P}_{m-1}(\mathbb{R}^n),
\|\cdot\|_{\dot{W}^{m,X}(\mathbb{R}^n)}\right)
$$
(which can be deduced from an argument similar to that
used in the proof of \cite[p.\,22, Theorem~1]{ma2011}
with $\Omega$ and $L^p$ replaced, respectively,
by $\mathbb{R}^n$ and $X$)
imply that there exists a function
$g\in\dot{W}^{m,X}(\mathbb{R}^n)/\mathcal{P}_{m-1}(\mathbb{R}^n)$
such that
\begin{align*}
\left\|g-\Lambda(f_k)\right\|_{\dot{W}^{m,X}(\mathbb{R}^n)}\to0
\end{align*}
as $k\to\infty$.
From this, Proposition~\ref{1533}(v),
the proven conclusion that $\Lambda(f_k)=f_k$
almost everywhere in $\Omega$, and \eqref{1611}, we infer that
\begin{align*}
\left\|g|_\Omega-f\right\|_{\dot{W}^{m,X}(\Omega)}
&\leq\|g|_\Omega-\Lambda(f_k)|_\Omega\|_{\dot{W}^{m,X}(\Omega)}
+\|\Lambda(f_k)|_\Omega-f\|_{\dot{W}^{m,X}(\Omega)}\\
&\leq\|g-\Lambda(f_k)\|_{\dot{W}^{m,X}(\mathbb{R}^n)}
+\|f_k-f\|_{\dot{W}^{m,X}(\Omega)}\to0
\end{align*}
as $k\to\infty$ and hence
$\|g|_\Omega-f\|_{\dot{W}^{m,X}(\Omega)}=0$,
which further implies that there
exists a polynomial $P$ of degree less than $m$
such that $g=f+P$ almost everywhere in $\Omega$.
Define $\Lambda(f):=g-P$.
Then $\Lambda(f)$ satisfies (i) of the present theorem.
Notice that, for any $\alpha\in\mathbb{Z}_+^n$ with $|\alpha|=m$,
$D^\alpha(\Lambda(f))=D^\alpha g$.
By this, Definition~\ref{1659}(v),
Proposition~\ref{1533}(v),
\eqref{1550}, and \eqref{1604},
we conclude that
\begin{align*}
\left\|\Lambda(f)\right\|_{\dot{W}^{m,X}(\mathbb{R}^n)}
&=\sum_{\alpha\in\mathbb{Z}_+^n,\,|\alpha|=m}
\left\|D^\alpha(\Lambda(f))\right\|_{X(\mathbb{R}^n)}\\
&=\sum_{\alpha\in\mathbb{Z}_+^n,\,|\alpha|=m}
\left\|D^\alpha g\right\|_{X(\mathbb{R}^n)}
=\|g\|_{\dot{W}^{m,X}(\mathbb{R}^n)}\\
&\leq\left\|g-\Lambda(f_k)\right\|_{\dot{W}^{m,X}(\mathbb{R}^n)}
+\left\|\Lambda(f_k)\right\|_{\dot{W}^{m,X}(\mathbb{R}^n)}\\
&\lesssim\left\|g-\Lambda(f_k)\right\|_{\dot{W}^{m,X}(\mathbb{R}^n)}
+\left\|f_k\right\|_{\dot{W}^{m,X}(\Omega)}\\
&\leq\left\|g-\Lambda(f_k)\right\|_{\dot{W}^{m,X}(\mathbb{R}^n)}
+\left\|f_k-f\right\|_{\dot{W}^{m,X}(\Omega)}+
\left\|f\right\|_{\dot{W}^{m,X}(\Omega)}\\
&\to\left\|f\right\|_{\dot{W}^{m,X}(\Omega)}
\end{align*}
as $k\to\infty$.
Thus, $\Lambda(f)$ satisfies (ii) of the present theorem,
which completes the proof of Theorem~\ref{extension2}.
\end{proof}

As a corollary of Theorem~\ref{extension2},
we have the following conclusion
whose proof is similar to that of Corollary~\ref{restr};
we omit the details.

\begin{corollary}\label{restr2}
Let $\Omega\subset\mathbb{R}^n$ be an $(\varepsilon,\infty)$-domain
with $\varepsilon\in(0,1]$.
Let $m\in\mathbb{N}$
and $X(\mathbb{R}^n)$ be a ball Banach function space
satisfying the assumptions (a) and (b) of Theorem~\ref{2247}.
Then $f\in\dot{W}^{m,X}(\Omega)$ if and only if there exists
$g\in\dot{W}^{m,X}(\mathbb{R}^n)$ such that
$g=f$ almost everywhere in $\Omega$;
moreover, there exists a positive constant $C$,
independent of both $f$ and $g$, such that
$$
\left\|f\right\|_{\dot{W}^{m,X}(\Omega)}
\leq\left\|g\right\|_{\dot{W}^{m,X}(\mathbb{R}^n)}
\leq C\left\|f\right\|_{\dot{W}^{m,X}(\Omega)}.
$$
\end{corollary}

\section{Asymptotics
of $\dot{W}^{1,X}(\Omega)$ Functions in Terms of
Ball Banach Sobolev\\ Norms}
\label{S3}

This section is devoted to establishing the asymptotics
of $\dot{W}^{1,X}(\Omega)$ functions in terms of
ball Banach Sobolev norms.
To this end, we begin with
the following definition of the
radial decreasing approximation of the identity
on $\mathbb{R}^n$, which
can be found in \cite[Definition~2.3]{dgpyyz2022}.

\begin{definition}\label{ATI}
Let $\nu_0\in(0,\infty)$.
A family $\{\rho_\nu\}_{\nu\in(0,\nu_0)}$
of nonnegative and locally integrable functions on $(0,\infty)$
is called a $\nu_0$-\emph{radial
decreasing approximation of the identity
on $\mathbb{R}^n$} (for short, a $\nu_0$-RDATI)
if $\{\rho_\nu\}_{\nu\in(0,\nu_0)}$ satisfies that
\begin{enumerate}
\item[\textup{(i)}]
for any $\nu\in(0,\nu_0)$,
$\rho_\nu$ is decreasing on $(0,\infty)$;
\item[\textup{(ii)}]
for any $\nu\in(0,\nu_0)$,
\begin{align*}
\int_0^\infty\rho_\nu(r)r^{n-1}\,dr=1;
\end{align*}
\item[\textup{(iii)}]
for any $\delta\in(0,\infty)$,
\begin{align*}
\lim_{\nu\to0^+}\int_\delta^\infty
\rho_\nu(r)r^{n-1}\,dr=0.
\end{align*}
\end{enumerate}
\end{definition}

Recall that
$\nu\to0^+$ means that there exists $\nu_0\in(0,\infty)$
such that $\nu\in(0,\nu_0)$ and $\nu\to0$.
Now, we present the main theorem of this section as follows.

\begin{theorem}\label{2045}
Let $\Omega\subset\mathbb{R}^n$ be a bounded
$(\varepsilon,\infty)$-domain
with $\varepsilon\in(0,1]$
and $\{\rho_\nu\}_{\nu\in(0,\nu_0)}$
a $\nu_0$-{\rm RDATI} on $\mathbb{R}^n$
with $\nu_0\in(0,\infty)$.
Let $X(\mathbb{R}^n)$ be a ball Banach function space
and $p\in[1,\infty)$.
Assume that
\begin{enumerate}
\item[\rm(i)]
$X(\mathbb{R}^n)$ has an absolutely continuous norm;
\item[\rm(ii)]
$X^\frac{1}{p}(\mathbb{R}^n)$
is a ball Banach function space;
\item[\rm(iii)]
the Hardy--Littlewood maximal operator $\mathcal{M}$
is bounded on $[X^\frac{1}{p}(\mathbb{R}^n)]'$.
\end{enumerate}
Then, for any $f\in\dot{W}^{1,X}(\Omega)$,
\begin{align}\label{2031}
\lim_{\nu\to0^+}
\left\|\left[\int_\Omega\frac{|f(\cdot)-f(y)|^p}{
|\cdot-y|^p}\rho_\nu(|\cdot-y|)\,dy
\right]^\frac{1}{p}\right\|_{X(\Omega)}
=\left[\kappa(p,n)\right]^\frac{1}{p}
\left\|\,\left|\nabla f\right|\,\right\|_{X(\Omega)},
\end{align}
where the constant
$\kappa(p,n)$ is the same as in \eqref{kappaqn}.
\end{theorem}

\subsection{Asymptotics
of $C^\infty(\overline{\Omega})$ Functions in Terms of
Ball Banach Sobolev Norms}
\label{sub4.1}

To show Theorem~\ref{2045}, in this subsection,
we first establish the following
asymptotics of $C^\infty(\overline{\Omega})$ functions.

\begin{proposition}\label{2043}
Let $\Omega\subset\mathbb{R}^n$ be
a bounded $(\varepsilon,\infty)$-domain
with $\varepsilon\in(0,1]$.
Let $X(\mathbb{R}^n)$ be a ball Banach function space and
$\{\rho_\nu\}_{\nu\in(0,\nu_0)}$
a $\nu_0$-{\rm RDATI} on $\mathbb{R}^n$
with $\nu_0\in(0,\infty)$.
Assume that $X(\mathbb{R}^n)$ has an absolutely continuous norm.
Then, for any $p\in(0,\infty)$
and $f\in C^\infty(\overline{\Omega})$,
\eqref{2031} holds true.
\end{proposition}

To show Proposition~\ref{2043}, we need
the following well-known inequality
(see, for instance, \cite[p.\,699]{b2002}).

\begin{lemma}\label{1111}
Let $q\in(0,\infty)$.
Then, for any $\theta\in(0,1)$,
there exists a positive constant $C_{(\theta)}$ such that,
for any $a,b\in(0,\infty)$,
$$
(a+b)^q\leq(1+\theta)a^q+C_{(\theta)}b^q.
$$
\end{lemma}

\begin{proof}[Proof of Proposition~\ref{2043}]
Let $f\in C^\infty(\overline{\Omega})$,
$x\in\Omega$, and
$R_x:=\mathrm{dist\,}(x,\Omega^\complement)$.
Then $R_x\in(0,\infty)$
because $\Omega$ is a bounded open set.
We write
\begin{align}\label{I1+I2}
&\int_\Omega\frac{|f(x)-f(y)|^p}{|x-y|^p}
\rho_\nu(|x-y|)\,dy\\
&\quad=\int_{B(x,R_x)}\frac{|f(x)-f(y)|^p}{|x-y|^p}
\rho_\nu(|x-y|)\,dy
+\int_{\Omega\setminus B(x,R_x)}\ldots\nonumber\\
&\quad=:I_1(x)+I_2(x).\nonumber
\end{align}
To estimate $I_1(x)$,
applying the Taylor expansion and the mean value theorem,
we find that there exists $L\in(0,\infty)$ such that,
for any $y\in B(x,R_x)$,
$$
\left|f(x)-f(y)-\nabla f(x)\cdot(x-y)\right|\leq L|x-y|^2,
$$
which further implies that, for any $y\neq x$,
\begin{align}\label{2129}
&\frac{|\nabla f(x)\cdot(x-y)|}{|x-y|}
-L|x-y|\\
&\quad\leq
\frac{\left|f(x)-f(y)\right|}{|x-y|}
\leq\frac{|\nabla f(x)\cdot(x-y)|}{|x-y|}
+L|x-y|.\nonumber
\end{align}
On the one hand,
by the second inequality of \eqref{2129}
and Lemma~\ref{1111},
we find that,
for any given $\theta_1\in(0,\infty)$,
there exists a positive constant $C_{(\theta_1)}$ such that
\begin{align}\label{I1}
I_1(x)
&\leq(1+\theta_1)\int_{B(x,R_x)}
\frac{|\nabla f(x)\cdot(x-y)|^p}{|x-y|^p}
\rho_\nu(|x-y|)\,dy\\
&\quad+C_{(\theta_1)}L^p\int_{B(x,R_x)}
|x-y|^p
\rho_\nu(|x-y|)\,dy\nonumber\\
&=:(1+\theta_1)I_{1,1}(x)+C_{(\theta_1)}
L^pI_{1,2}(x).\nonumber
\end{align}
To estimate $I_{1,1}(x)$, from the change of variable, the polar
coordinate, and both (ii) and (iii) of Definition~\ref{ATI},
we deduce that
\begin{align}\label{I11}
I_{1,1}(x)&=\left|\nabla f(x)\right|^p
\int_{\mathbb{S}^{n-1}}\left|\frac{\nabla f(x)}{|\nabla f(x)|}
\cdot\omega\right|^p\,d\sigma(\omega)\int_0^{R_x}
\rho_\nu(r)r^{n-1}\,dr\\
&=\kappa(p,n)\left|\nabla f(x)\right|^p
\left[\int_0^{\infty}
\rho_\nu(r)r^{n-1}\,dr-
\int_{R_x}^\infty\ldots\right]\nonumber\\
&=\kappa(p,n)\left|\nabla f(x)\right|^p
\left[1-\int_{R_x}^\infty\rho_\nu(r)r^{n-1}\,dr\right]
\to\kappa(p,n)\left|\nabla f(x)\right|^p\nonumber
\end{align}
as $\nu\to0^+.$
To estimate $I_{1,2}(x)$, by the polar
coordinate and both (ii) and (iii) of Definition~\ref{ATI},
we find that, for any $N\in\mathbb{N}$,
\begin{align}\label{I12}
I_{1,2}(x)&=\int_{\mathbb{S}^{n-1}}\int_0^{R_x}
\rho_\nu(r)r^{n-1}r^p\,dr\,d\sigma\\
&\sim\left[\int_0^{\frac{R_x}{N}}
\rho_\nu(r)r^{n-1}r^p\,dr
+\int_{\frac{R_x}{N}}^{R_x}\ldots\right]\nonumber\\
&\leq\left[\left(
\frac{R_x}{N}\right)^p\int_0^{\frac{R_x}{N}}
\rho_\nu(r)r^{n-1}\,dr
+R_x^p\int_{\frac{R_x}{N}}^{R_x}\ldots\right]\nonumber\\
&\leq\left[\left(
\frac{R_x}{N}\right)^p\int_0^{\infty}
\rho_\nu(r)r^{n-1}\,dr
+R_x^p\int_{\frac{R_x}{N}}^{\infty}\ldots\right]\nonumber\\
&=\left[\left(
\frac{R_x}{N}\right)^p
+R_x^p\int_{\frac{R_x}{N}}^{\infty}
\rho_\nu(r)r^{n-1}\,dr\right]\to0\nonumber
\end{align}
via first letting $\nu\to0^+$
and then letting $N\to\infty$.
From this, \eqref{I1}, and \eqref{I11},
we infer that
\begin{align}\label{2204}
I_1(x)\leq(1+\theta_1)I_{1,1}(x)+C_{(\theta_1)}L^pI_{1,2}(x)
\to\kappa(p,n)\left|\nabla f(x)\right|^p
\end{align}
via first letting $\nu\to0^+$
and then letting $\theta_1\to0^+$.
On the other hand, by the first inequality of \eqref{2129}
and Lemma~\ref{1111}, we find that,
for any given $\theta_2\in(0,\infty)$,
there exists a positive constant $C_{(\theta_2)}$ such that
\begin{align*}
I_1(x)
&\ge\frac{1}{1+\theta_2}\int_{B(x,R_x)}
\left|\frac{\nabla f(x)\cdot(x-y)}{|x-y|}\right|^p
\rho_\nu(|x-y|)\,dy\\
&\quad-\frac{C_{(\theta_2)}}{1+\theta_2}L^p\int_{B(x,R_x)}
|x-y|^p
\rho_\nu(|x-y|)\,dy\\
&=\frac{1}{1+\theta_2}I_{1,1}(x)
-\frac{C_{(\theta_2)}L^p}{1+\theta_2}I_{1,2}(x),
\end{align*}
which, combined with both \eqref{I11} and \eqref{I12},
further implies that
\begin{align*}
I_1(x)\ge\frac{1}{1+\theta_2}I_{1,1}(x)
-\frac{C_{(\theta_2)}L^p}{1+\theta_2}I_{1,2}(x)
\to\kappa(p,n)\left|\nabla f(x)\right|^p
\end{align*}
via first letting $\nu\to0^+$
and then letting $\theta_2\to0^+$.
From this and \eqref{2204}, we infer that
\begin{align}\label{2206}
I_1(x)\to\kappa(p,n)\left|\nabla f(x)\right|^p
\end{align}
as $\nu\to0^+$.

Next, we turn to estimate $I_2(x)$.
By $f\in C^\infty(\overline{\Omega})$
and the fact that $\Omega$ is a bounded
$(\varepsilon,\infty)$-domain
[here, we use both (ii)
and (iii) of Definition~\ref{2121}],
we find that, for any $x,y\in\Omega$,
\begin{align}\label{949}
\left|f(x)-f(y)\right|
&\leq\int_{\Gamma_{x,y}}\left|\nabla f(s)\right|\,ds
\leq\left\|\,|\nabla f|\,\right\|_{L^\infty(\Omega)}
\ell(\Gamma_{x,y})\\
&\leq\frac{\|\,|\nabla f|\,\|_{L^\infty(\Omega)}}
{\varepsilon}|x-y|,\nonumber
\end{align}
where $\Gamma_{x,y}\subset\mathbb{R}^n$ is
the same rectifiable curve as in Definition~\ref{2121}
connecting $x$ and $y$, which, combined with
the polar coordinate, $R_x\in(0,\infty)$,
and Definition~\ref{ATI}(iii),
further implies that
\begin{align*}
I_2(x)\leq
\left\|\,|\nabla f|\,\right\|_{L^\infty(\Omega)}^p
\varepsilon^{-p}
\int_{\mathbb{S}^{n-1}}\int_{R_x}^\infty
\rho_\nu(r)r^{n-1}\,dr\,d\sigma\to0
\end{align*}
as $\nu\to0^+$.
From this, \eqref{2206}, and \eqref{I1+I2},
we deduce that, for any $x\in\Omega$,
\begin{align}\label{2215}
\int_\Omega\frac{|f(x)-f(y)|^p}{|x-y|^p}
\rho_\nu(|x-y|)\,dy
=I_1(x)+I_2(x)\to\kappa(p,n)\left|\nabla f(x)\right|^p
\end{align}
as $\nu\to0^+$.
Moreover, by \eqref{949},
the polar coordinate,
and Definition~\ref{ATI}(ii), we find that,
for any $x\in\Omega$,
\begin{align*}
&\int_\Omega\frac{|f(x)-f(y)|^p}{|x-y|^p}
\rho_\nu(|x-y|)\,dy\\
&\quad\lesssim
\left\|\,|\nabla f|\,\right\|_{L^\infty(\Omega)}^p
\int_\Omega\rho_\nu(|x-y|)\,dy\\
&\quad\leq\left\|\,|\nabla f|\,\right\|_{L^\infty(\Omega)}^p
\int_{\mathbb{S}^{n-1}}\int_0^\infty
\rho_\nu(r)r^{n-1}\,dr\,d\sigma\\
&\quad\sim
\left\|\,|\nabla f|\,\right\|_{L^\infty(\Omega)}^p,
\end{align*}
the right hand side of which obviously belongs to $X(\Omega)$
because $\Omega$ is bounded.
From this, Proposition~\ref{1533}(v), \eqref{2215},
and Remark~\ref{LdctOmega},
we deduce that
\begin{align*}
&\lim_{\nu\to0^+}
\left\|\left[\int_\Omega\frac{|f(\cdot)-f(y)|^p}{
|\cdot-y|^p}\rho_\nu(|\cdot-y|)\,dy
\right]^\frac{1}{p}\right\|_{X(\Omega)}\\
&\quad\leq\left[\kappa(p,n)\right]^\frac{1}{p}
\left\|\,\left|\nabla f\right|\,\right\|_{X(\Omega)}\\
&\qquad+\lim_{\nu\to0^+}
\left\|\left[\int_\Omega\frac{|f(\cdot)-f(y)|^p}{
|\cdot-y|^p}\rho_\nu(|\cdot-y|)\,dy
\right]^\frac{1}{p}-\left[\kappa(p,n)\right]^\frac{1}{p}
\left|\nabla f(\cdot)\right|\right\|_{X(\Omega)}\\
&\quad=\left[\kappa(p,n)\right]^\frac{1}{p}
\left\|\,\left|\nabla f\right|\,\right\|_{X(\Omega)}.
\end{align*}
This finishes the proof of Proposition~\ref{2043}.
\end{proof}

\subsection{A Key Upper Estimate}
\label{sub4.2}

The main target of this subsection is to establish
the following proposition which gives a key upper estimate
in the proof of Theorem~\ref{2045}.

\begin{proposition}\label{1537}
Let $\Omega\subset\mathbb{R}^n$ be
a bounded $(\varepsilon,\infty)$-domain
with $\varepsilon\in(0,1]$.
Let $X(\mathbb{R}^n)$ be a ball Banach function space,
$p\in[1,\infty)$,
and $\rho\in L_{\mathrm{loc}}^1(0,\infty)$
a nonnegative and decreasing function
satisfying Definition~\ref{ATI}(ii).
Assume that $X^\frac{1}{p}(\mathbb{R}^n)$
is a ball Banach function space
and that the Hardy--Littlewood maximal operator $\mathcal{M}$
is bounded on $[X^\frac{1}{p}(\mathbb{R}^n)]'$.
Then there exists a positive constant $C$ such that,
for any $f\in\dot{W}^{1,X}(\Omega)$,
\begin{align*}
\left\|\left[\int_\Omega\frac{|f(\cdot)-f(y)|^p}{
|\cdot-y|^p}\rho(|\cdot-y|)\,dy
\right]^\frac{1}{p}\right\|_{X(\Omega)}
\leq C\left\|\,\left|\nabla f\right|\,\right\|_{X(\Omega)}.
\end{align*}
\end{proposition}

To show Proposition~\ref{1537},
we need the following lemma
which can be deduced from
an argument similar to the proof of \cite[Lemma~3.23]{dgpyyz2022};
we present the details here for the convenience of the reader.

\begin{lemma}\label{dgpyyz}
Let $p\in[1,\infty)$ and $\rho\in L_{\mathrm{loc}}^1(0,\infty)$
be a nonnegative and decreasing function
satisfying Definition~\ref{ATI}(ii).
Assume that $\omega\in A_1(\mathbb{R}^n)$.
Then there exists a positive constant $C_{(n,p)}$,
depending only on both $n$ and $p$,
such that, for any $f\in\dot{W}^{1,p}_\omega(\mathbb{R}^n)$,
\begin{align}\label{937}
&\int_{\mathbb{R}^n}\int_{\mathbb{R}^n}
\frac{|f(x)-f(y)|^p}{|x-y|^p}\rho(|x-y|)\omega(x)\,dy\,dx\\
&\quad\leq C_{(n,p)}[\omega]_{A_1(\mathbb{R}^n)}^2
\left\|\,\left|\nabla f\right|\,\right\|_{L^p_\omega(\mathbb{R}^n)}^p.\nonumber
\end{align}
\end{lemma}

To show Lemma~\ref{dgpyyz}, we need two technical lemmas.
The following lemma can be deduced from the proof
of \cite[Lemma~3.30]{dgpyyz2022}
which also holds true for any $f\in C^\infty(\mathbb{R}^n)$
with $|\nabla f|\in C_{\mathrm{c}}(\mathbb{R}^n)$;
we omit the details.

\begin{lemma}\label{952}
Let all the symbols be the same as in Lemma~\ref{dgpyyz}.
Then there exists a positive constant $C_{(n,p)}$,
depending only on both $n$ and $p$,
such that, for any $f\in C^\infty(\mathbb{R}^n)$
with $|\nabla f|\in C_{\mathrm{c}}(\mathbb{R}^n)$,
\eqref{937} holds true.
\end{lemma}

The following lemma is a corollary of \cite[Corollary~2.18]{dlyyz.arxiv}
with $X(\mathbb{R}^n):=L^p_\omega(\mathbb{R}^n)$.

\begin{lemma}\label{319}
Let $p\in[1,\infty)$ and $\omega\in A_p(\mathbb{R}^n)$.
Then, for any $f\in\dot{W}^{1,p}_\omega(\mathbb{R}^n)$,
there exists a sequence $\{f_k\}_{k\in\mathbb{N}}
\subset C^\infty(\mathbb{R}^n)$ with
$|\nabla f_k|\in C_{\mathrm{c}}(\mathbb{R}^n)$ for any $k\in\mathbb{N}$
such that, for any $R\in(0,\infty)$,
\begin{align}\label{958}
\lim_{k\to\infty}\left\|\,\left|\nabla\left(f
-f_k\right)\right|\,\right\|_{L^p_\omega(\mathbb{R}^n)}=0
\ \text{and}\
\lim_{k\to\infty}\left\|\left(f-f_k\right)
\mathbf{1}_{B(\mathbf{0},R)}\right\|_{L^p_\omega(\mathbb{R}^n)}=0.
\end{align}
\end{lemma}

\begin{proof}[Proof of Lemma~\ref{dgpyyz}]
Let $f\in\dot{W}^{1,p}_\omega(\mathbb{R}^n)$.
By Lemma~\ref{319}, we find that
there exists a sequence $\{f_k\}_{k\in\mathbb{N}}
\subset C^\infty(\mathbb{R}^n)$ with
$|\nabla f_k|\in C_{\mathrm{c}}(\mathbb{R}^n)$ for any $k\in\mathbb{N}$
such that, for any $R\in(0,\infty)$,
\eqref{958} holds true.
For any $N\in\mathbb{R}^n$, let
\begin{align*}
E_N:&=\left\{(x,y)\in B(\mathbf{0},N)\times B(\mathbf{0},N):\
|x-y|\in(N^{-1},\infty),\right.\\
&\left.\quad\omega(x),\omega(y)\in(N^{-1},N)\right\}.
\end{align*}
From this, a change of variables, the Tonelli theorem,
Definition~\ref{ATI}(ii),
Lemma~\ref{952} with $f:=f_k$, and \eqref{958},
we infer that, for any fixed $N\in\mathbb{N}$,
\begin{align*}
&\iint_{E_N}\frac{|f(x)-f(y)|^p}{|x-y|^p}\rho(|x-y|)\omega(x)\,dy\,dx\\
&\quad\lesssim\iint_{E_N}\frac{|f(x)-f_k(x)|^p}{|x-y|^p}\rho(|x-y|)\omega(x)\,dy\,dx\\
&\qquad+\iint_{E_N}\frac{|f_k(x)-f_k(y)|^p}{|x-y|^p}\rho(|x-y|)\omega(x)\,dy\,dx\\
&\qquad+\iint_{E_N}\frac{|f_k(y)-f(y)|^p}{|x-y|^p}\rho(|x-y|)
\frac{\omega(x)}{\omega(y)}\omega(y)\,dy\,dx\\
&\quad\lesssim N^p\int_{B(\mathbf{0},N)}\left|f(x)-f_k(x)\right|^p
\omega(x)\,dx\int_{\mathbb{R}^n}\rho(|h|)\,dh\\
&\qquad+
\left\|\,\left|\nabla f_k\right|\,\right\|_{L^p_\omega(\mathbb{R}^n)}^p\\
&\qquad+N^{p+2}\int_{B(\mathbf{0},N)}\left|f_k(y)-f(y)\right|^p
\omega(y)\,dy\int_{\mathbb{R}^n}\rho(|h|)\,dh\\
&\quad\lesssim N^{p+2}\left\|\left(f-f_k\right)
\mathbf{1}_{B(\mathbf{0},N)}\right\|_{L^p_\omega(\mathbb{R}^n)}^p\\
&\qquad+\left\|\,\left|\nabla\left(f_k-f\right)\right|\,\right\|_{L^p_\omega(\mathbb{R}^n)}^p
+\left\|\,\left|\nabla f\right|\,\right\|_{L^p_\omega(\mathbb{R}^n)}^p\\
&\quad\to\left\|\,\left|\nabla f\right|\,\right\|_{L^p_\omega(\mathbb{R}^n)}^p
\end{align*}
as $k\to\infty$.
By this and the Levi lemma, we conclude that
\begin{align*}
&\int_{\mathbb{R}^n}\int_{\mathbb{R}^n}
\frac{|f(x)-f(y)|^p}{|x-y|^p}\rho(|x-y|)\omega(x)\,dy\,dx\\
&\quad=\lim_{N\to\infty}\iint_{E_N}\frac{|f(x)-f(y)|^p}{|x-y|^p}\rho(|x-y|)\omega(x)\,dy\,dx\\
&\quad\lesssim
\left\|\,\left|\nabla f\right|\,\right\|_{L^p_\omega(\mathbb{R}^n)}^p,
\end{align*}
which completes the proof of Lemma~\ref{dgpyyz}.
\end{proof}

The following H\"older inequality
when $\Omega=\mathbb{R}^n$ is just
\cite[Lemma~2.13]{dlyyz.arxiv};
see also \cite[p.\,9, Theorem~2.4]{bs1988}.

\begin{lemma}\label{1639}
Let $X(\mathbb{R}^n)$ be a ball Banach function
space and $X'(\mathbb{R}^n)$ its associate
space. Let $X(\Omega)$ and $X'(\Omega)$ be, respectively,
the restrictive spaces
of $X(\mathbb{R}^n)$ and $X'(\mathbb{R}^n)$ on $\Omega$.
If $f\in X(\Omega)$ and $g\in X'(\Omega)$,
then $fg\in L^1(\Omega)$ and
$$
\int_{\Omega}\left|f(x)g(x)\right|\,dx
\leq\|f\|_{X(\Omega)}\|g\|_{X'(\Omega)}.
$$
\end{lemma}

\begin{proof}
Let $f\in X(\Omega)$ and $g\in X'(\Omega)$
and let $\widetilde{f}$ and $\widetilde{g}$
be defined the same as in \eqref{1448}
related, respectively, to $f$ and $g$.
Then, by \cite[Lemma~2.13]{dlyyz.arxiv}
and Proposition~\ref{norm},
we conclude that
\begin{align*}
\int_\Omega f(x)g(x)\,dx
&=\int_{\mathbb{R}^n}
\widetilde{f}(x)\widetilde{g}(x)\,dx
\leq\left\|\widetilde{f}\right\|_{X(\mathbb{R}^n)}
\left\|\widetilde{g}\right\|_{X'(\mathbb{R}^n)}\\
&=\|f\|_{X(\Omega)}\|g\|_{X'(\Omega)},
\end{align*}
which completes the proof of Lemma~\ref{1639}.
\end{proof}

Next, we use Lemmas~\ref{chua1992'},~\ref{dgpyyz},
and~\ref{1639}
to show Proposition~\ref{1537}.

\begin{proof}[Proof of Proposition~\ref{1537}]
Let $\omega\in A_1(\mathbb{R}^n)$.
We first claim that, for any $g\in\dot{W}^{1,p}_\omega(\Omega)$,
\begin{align*}
&\int_{\Omega}\left[\int_{\Omega}
\frac{|g(x)-g(y)|^p}{|x-y|^p}\rho(|x-y|)\,dy\right]
\omega(x)\,dx\lesssim
\left\|\,\left|\nabla g\right|\,\right\|_{L^{p}_\omega(\Omega)}^p,
\end{align*}
where the implicit positive constant depends only
on $p$, $n$, and $[\omega]_{A_1(\mathbb{R}^n)}$.
To this end, let $g\in\dot{W}^{1,p}_\omega(\Omega)$.
By Lemma~\ref{chua1992},
we find that there exists
$\widetilde{g}\in\dot{W}^{1,p}_\omega(\mathbb{R}^n)$
such that $\widetilde{g}=g$ almost everywhere in $\Omega$ and
\begin{align*}
\left\|\,\left|\nabla \widetilde{g}\right|
\,\right\|_{L^p_\omega(\mathbb{R}^n)}
\lesssim\left\|\,\left|\nabla g\right|\,\right\|_{L^{p}_\omega(\Omega)}.
\end{align*}
From this and Lemma~\ref{dgpyyz},
we deduce that
\begin{align*}
&\int_{\Omega}\left[\int_{\Omega}
\frac{|g(x)-g(y)|^p}{|x-y|^p}\rho(|x-y|)\,dy\right]
\omega(x)\,dx\\
&\quad=\int_{\Omega}\left[\int_{\Omega}
\frac{|\widetilde{g}(x)-\widetilde{g}(y)|^p}{|x-y|^p}\rho(|x-y|)\,dy\right]
\omega(x)\,dx\\
&\quad\leq\int_{\mathbb{R}^n}\left[\int_{\mathbb{R}^n}
\frac{|\widetilde{g}(x)-\widetilde{g}(y)|^p}{|x-y|^p}\rho(|x-y|)\,dy\right]
\omega(x)\,dx\\
&\quad\lesssim
\left\|\,\left|\nabla \widetilde{g}\right|
\,\right\|_{L^p_\omega(\mathbb{R}^n)}^p
\lesssim\left\|\,\left|\nabla g\right|\,\right\|_{L^{p}_\omega(\Omega)}^p.
\end{align*}
This finishes the proof of the above claim.

Now, let $f\in\dot{W}^{1,X}(\Omega)$.
For any $g\in[X^\frac{1}{p}(\mathbb{R}^n)]'$
with $\|g\|_{[X^\frac{1}{p}(\mathbb{R}^n)]'}=1$,
let $\omega_g:=R_{[X^\frac{1}{p}(\mathbb{R}^n)]'}g$.
By Remark~\ref{waicha},
the above claim, Lemma~\ref{1639}, and Proposition~\ref{2034},
we conclude that
\begin{align*}
&\left\|\left[\int_\Omega\frac{|f(\cdot)-f(y)|^p}{
|\cdot-y|^p}\rho(|\cdot-y|)\,dy
\right]^\frac{1}{p}\right\|_{X(\Omega)}^p\\
&\quad\leq\sup_{\|g\|_{[X^\frac{1}{p}(\mathbb{R}^n)]'}=1}
\int_{\Omega}\left[\int_{\Omega}
\frac{|f(x)-f(y)|^p}{|x-y|^p}\rho(|x-y|)\,dy\right]
\omega_g(x)\,dx\\
&\quad\lesssim\sup_{\|g\|_{[X^\frac{1}{p}(\mathbb{R}^n)]'}=1}
\int_\Omega\left|\nabla f(x)\right|^p\omega_g(x)\,dx\\
&\quad\leq\sup_{\|g\|_{[X^\frac{1}{p}(\mathbb{R}^n)]'}=1}
\|\omega_g|_\Omega\|_{[X^\frac{1}{p}(\Omega)]'}
\left\||\nabla f|^p\right\|_{X^\frac{1}{p}(\Omega)}\\
&\quad\leq\sup_{\|g\|_{[X^\frac{1}{p}(\mathbb{R}^n)]'}=1}
\|\omega_g\|_{[X^\frac{1}{p}(\mathbb{R}^n)]'}
\left\||\nabla f|^p\right\|_{X^\frac{1}{p}(\Omega)}
\lesssim\left\|\,\left|\nabla f\right|\,\right\|_{X(\Omega)}^p.
\end{align*}
This finishes the proof of Proposition~\ref{1537}.
\end{proof}

\subsection{Proof of Theorem~\ref{2045}}
\label{sub4.3}

In this subsection, we use
Theorem~\ref{2247} and Propositions~\ref{2043}
and~\ref{1537}
to prove Theorem~\ref{2045}.

\begin{proof}[Proof of Theorem~\ref{2045}]
Let $f\in\dot{W}^{1,X}(\Omega)$.
By the assumptions of the present theorem,
Theorem~\ref{2247}, and Remark~\ref{1711},
we find that there exists
a sequence $\{f_k\}_{k\in\mathbb{N}}\subset
C^\infty(\overline{\Omega})\cap\dot{W}^{1,X}(\Omega)$
such that
\begin{align}\label{1033}
\left\|f-f_k\right\|_{\dot{W}^{1,X}(\Omega)}
\sim\left\|\,\left|\nabla(f-f_k
)\right|\,\right\|_{X(\Omega)}\to0
\end{align}
as $k\to\infty$.
From both (ii) and (v) of Proposition~\ref{1533}
and the Minkowski inequality,
we infer that,
for any $\nu\in(0,\nu_0)$
and $k\in\mathbb{N}$,
\begin{align}\label{IIIIII}
&\left|\left\|\left[\int_\Omega\frac{|f(\cdot)-f(y)|^p}{
|\cdot-y|^p}\rho_\nu(|\cdot-y|)\,dy
\right]^\frac{1}{p}\right\|_{X(\Omega)}-\kappa(p,n)
\left\|\,\left|\nabla f\right|\,\right\|_{X(\Omega)}\right|\\
&\quad\leq\left|\left\|\left[\int_\Omega\frac{|f(\cdot)-f(y)|^p}{
|\cdot-y|^p}\rho_\nu(|\cdot-y|)\,dy
\right]^\frac{1}{p}\right\|_{X(\Omega)}\right.\nonumber\\
&\qquad\left.-
\left\|\left[\int_\Omega\frac{|f_k(\cdot)-f_k(y)|^p}{
|\cdot-y|^p}\rho_\nu(|\cdot-y|)\,dy
\right]^\frac{1}{p}\right\|_{X(\Omega)}\right|\nonumber\\
&\qquad+\left|\left\|\left[\int_\Omega\frac{|f_k(\cdot)-f_k(y)|^p}{
|\cdot-y|^p}\rho_\nu(|\cdot-y|)\,dy
\right]^\frac{1}{p}\right\|_{X(\Omega)}-\kappa(p,n)
\left\|\,\left|\nabla f_k\right|\,\right\|_{X(\Omega)}\right|\nonumber\\
&\qquad+\kappa(p,n)\left|\left\|\,\left|
\nabla f_k\right|\,\right\|_{X(\Omega)}
-\left\|\,\left|\nabla f\right|\,\right\|_{X(\Omega)}\right|\nonumber\\
&\quad\leq\left|\left\|\left[\int_\Omega
\frac{|(f-f_k)(\cdot)-(f-f_k)(y)|^p}{
|\cdot-y|^p}\rho_\nu(|\cdot-y|)\,dy
\right]^\frac{1}{p}\right\|_{X(\Omega)}\right.\nonumber\\
&\qquad+\left|\left\|\left[\int_\Omega\frac{|f_k(\cdot)-f_k(y)|^p}{
|\cdot-y|^p}\rho_\nu(|\cdot-y|)\,dy
\right]^\frac{1}{p}\right\|_{X(\Omega)}-\kappa(p,n)
\left\|\,\left|\nabla f_k\right|\,\right\|_{X(\Omega)}\right|\nonumber\\
&\qquad+\kappa(p,n)\left|\left\|\,\left|
\nabla f_k\right|\,\right\|_{X(\Omega)}
-\left\|\,\left|\nabla f\right|\,\right\|_{X(\Omega)}\right|\nonumber\\
&\quad=:\mathrm{I}+\mathrm{II}+\mathrm{III}.\nonumber
\end{align}
To estimate $\mathrm{I}$,
by the assumption that $\{\rho_\nu\}_{\nu\in(0,\nu_0)}$
is a $\nu_0$-RDATI on $\mathbb{R}^n$,
it is easy to show that, for any $\nu\in(0,\nu_0)$,
$\rho_\nu$ satisfies all the assumptions on $\rho$
in Proposition~\ref{1537}.
Using this, Proposition~\ref{1537}, and \eqref{1033},
we conclude that, for any $\nu\in(0,\nu_0)$,
\begin{align}\label{I}
\mathrm{I}\lesssim
\left\|\,\left|\nabla(f-f_k
)\right|\,\right\|_{X(\Omega)}\to0
\end{align}
as $k\to\infty$,
where the implicit positive constant is independent of $\nu$.
To estimate $\mathrm{II}$,
from Proposition~\ref{2043}
with $f:=f_k$, we deduce that,
for any $k\in\mathbb{N}$,
\begin{align}\label{II}
\lim_{\nu\to0^+}\mathrm{II}=0.
\end{align}
To estimate $\mathrm{III}$,
by Proposition~\ref{1533}(v) and \eqref{1033}, we find that
\begin{align*}
\lim_{k\to\infty}\mathrm{III}
\leq\lim_{k\to\infty}\left\|\,\left|\nabla(f-
f_k)\right|\,\right\|_{X(\Omega)}=0,
\end{align*}
which, together with \eqref{IIIIII},
\eqref{I}, and \eqref{II},
further implies that
\begin{align*}
\left|\left\|\left[\int_\Omega\frac{|f(\cdot)-f(y)|^p}{
|\cdot-y|^p}\rho_\nu(|\cdot-y|)\,dy
\right]^\frac{1}{p}\right\|_{X(\Omega)}-\kappa(p,n)
\left\|\,\left|\nabla f\right|\,\right\|_{X(\Omega)}\right|\to0
\end{align*}
via first letting $\nu\to0^+$
and then letting $k\to\infty$.
This finishes the proof of Theorem~\ref{2045}.
\end{proof}

\begin{remark}
In Theorem~\ref{2045}, if $X(\mathbb{R}^n):=L^q(\mathbb{R}^n)$
with $q\in[1,\infty)$ and if $p\in[1,q]$,
then it is easy to show that $L^q(\mathbb{R}^n)$
is a ball Banach function space having an absolutely
continuous norm,
$X^\frac{1}{p}(\mathbb{R}^n)=L^\frac{q}{p}(\mathbb{R}^n)$
is a ball Banach function space,
and the Hardy--Littlewood maximal operator is bounded on
$[X^\frac{1}{p}(\mathbb{R}^n)]'
=L^{(\frac{q}{p})'}(\mathbb{R}^n)$.
Thus, Theorem~\ref{2045} in this case holds true.
In addition, we point out that \cite[(38)]{b2002}
(see also \cite{bbm2001}) is just
Theorem~\ref{2045} with
both $X(\mathbb{R}^n)=L^q(\mathbb{R}^n)$
and $q=p\in[1,\infty)$
and that Theorem~\ref{2045}
with both $X(\mathbb{R}^n)=L^q(\mathbb{R}^n)$
and $1\leq p<q<\infty$ is new.
Moreover, from Remark~\ref{1638}(iii),
we infer that
the assumption on $\Omega$ of Theorem~\ref{2045}
is weaker than the corresponding one in both \cite[Remark~7]{b2002}
and \cite{bbm2001}.
\end{remark}

\section{Bourgain--Brezis--Mironescu-Type Characterization
of $W^{1,X}(\Omega)$}
\label{S4}

This section is devoted to establishing the
Bourgain--Brezis--Mironescu-type characterization
of the inhomogeneous ball Banach Sobolev space $W^{1,X}(\Omega)$.
To this end,
we begin with the following concept of the locally doubling
property of ball quasi-Banach function spaces, which was introduced in
\cite[Definition~2.10]{dgpyyz2022}.

\begin{definition}\label{1658}
Let $\beta\in(0,\infty)$.
A ball quasi-Banach function space
$X(\mathbb{R}^n)$ is said to be \emph{locally $\beta$-doubling}
if there exists a positive constant $C$ such that,
for any $r\in(0,\infty)$ and $\lambda\in[1,\infty)$,
\begin{align*}
\left\|\mathbf{1}_{B(\mathbf{0},\lambda r)}\right\|_{X(\mathbb{R}^n)}
\leq C\lambda^\beta\left\|\mathbf{1}_{
B(\mathbf{0},r)}\right\|_{X(\mathbb{R}^n)}.
\end{align*}
\end{definition}

The following is one of the
main theorems of this section.

\begin{theorem}\label{BBMdomain}
Let $X(\mathbb{R}^n)$ be a ball Banach function space,
$\Omega\subset\mathbb{R}^n$ an open set,
$\{\rho_\nu\}_{\nu\in(0,\nu_0)}$
a $\nu_0$-{\rm RDATI} on $\mathbb{R}^n$
with $\nu_0\in(0,\infty)$,
and $p\in[1,\infty)$.
Assume that
\begin{enumerate}
\item[\textup{(i)}]
$C_{\mathrm{c}}^\infty(\Omega)$ is
dense in $X'(\Omega)$;
\item[\textup{(ii)}]
$X'(\Omega)$ has an absolutely continuous
norm on $\Omega$;
\item[\textup{(iii)}]
$X'(\mathbb{R}^n)$ is locally
$\beta$-doubling with $\beta\in(0,n+1)$.
\end{enumerate}
If $f\in X(\Omega)$ satisfies
\begin{align}\label{953}
\liminf_{\nu\to0^+}
\left\|\left[\int_\Omega\frac{|f(\cdot)-f(y)|^p}{
|\cdot-y|^p}\rho_\nu(|\cdot-y|)\,dy
\right]^\frac{1}{p}\right\|_{X(\Omega)}<\infty,
\end{align}
then $|\nabla f|\in X(\Omega)$ and
\begin{align*}
\left\|\,\left|\nabla f\right|\,\right\|_{X(\Omega)}
\lesssim\liminf_{\nu\to0^+}
\left\|\left[\int_\Omega\frac{|f(\cdot)-f(y)|^p}{
|\cdot-y|^p}\rho_\nu(|\cdot-y|)\,dy
\right]^\frac{1}{p}\right\|_{X(\Omega)},
\end{align*}
where the implicit positive constant is independent of $f$.
\end{theorem}

Applying an argument similar to that used in
the proof of \cite[Lemma~4.4]{dgpyyz2022} with $\mathbb{R}^n$
and \cite[Lemmas~3.18 and~4.3]{dgpyyz2022}
replaced, respectively, by $\Omega$,
Lemma~\ref{1639}, and Remark~\ref{1104},
we obtain the following lemma
which gives a characterization
of $W^{1,X}(\Omega)$; we omit the details.

\begin{lemma}\label{2212}
Let $\Omega\subset\mathbb{R}^n$ be an open set and
$X(\mathbb{R}^n)$ a ball Banach function space.
Assume that
\begin{enumerate}
\item[\textup{(i)}]
$C^\infty_{\mathrm{c}}(\Omega)$ is dense in $X'(\Omega)$;
\item[\textup{(ii)}]
$X'(\Omega)$ has an absolutely continuous norm on $\Omega$.
\end{enumerate}
Let $f\in X(\Omega)$.
Then $|\nabla f|\in X(\Omega)$ if and only if,
for any $j\in\{1,\ldots,n\}$,
\begin{align}\label{2218}
\mathcal{T}_j(f):=\sup\left\{\left|\int_\Omega
f(x)\partial_j\phi(x)\,dx\right|:\ \phi\in
C^\infty_{\mathrm{c}}(\Omega),
\,\|\phi\|_{X'(\Omega)}\leq1\right\}<\infty;
\end{align}
moreover, for such $f$,
$$
\left\|\,|\nabla f|\,\right\|_{X(\Omega)}
\sim\sum_{j=1}^n\mathcal{T}_j(f)
$$
with the positive equivalence constants independent of $f$.
\end{lemma}

The following estimate is just
the antepenultimate formula in the
proof of \cite[Theorem~4.1]{dgpyyz2022}.

\begin{lemma}\label{1459}
Let $X(\mathbb{R}^n)$ be a ball Banach function space
and $\{\rho_\nu\}_{\nu\in(0,\nu_0)}$
a $\nu_0$-{\rm RDATI} on $\mathbb{R}^n$
with $\nu_0\in(0,\infty)$.
Assume that
$X'(\mathbb{R}^n)$ is locally $\beta$-doubling with $\beta\in(0,n+1)$.
Then there exists a positive constant $C$ such that,
for any $j\in\{1,\ldots,n\}$,
$f\in X(\mathbb{R}^n)$, and
$\phi\in C_{\mathrm{c}}^\infty(\mathbb{R}^n)$
with $\|\phi\|_{X'(\mathbb{R}^n)}\leq1$,
\begin{align*}
&\left|\int_{\mathbb{R}^n}f(x)
\partial_j\phi(x)\,dx\right|\\
&\quad\leq C\liminf_{\nu\to0^+}
\int_{\mathbb{R}^n}\left[
\int_{\mathbb{R}^n}\frac{|f(x)-f(y)|}{|x-y|}
\rho_\nu(|x-y|)\,dy
\right]|\phi(x)|\,dx.
\end{align*}
\end{lemma}

Now, we are ready to show Theorem~\ref{BBMdomain}.

\begin{proof}[Proof of Theorem~\ref{BBMdomain}]
Let $f\in X(\Omega)$ satisfy \eqref{953}.
By Lemma~\ref{2212},
we find that, to prove the present theorem,
it suffices to show that, for any $j\in\{1,\ldots,n\}$,
$$
\mathcal{T}_j(f)\lesssim\liminf_{\nu\to0^+}
\left\|\left[\int_\Omega\frac{|f(\cdot)-f(y)|^p}{
|\cdot-y|^p}\rho_\nu(|\cdot-y|)\,dy
\right]^\frac{1}{p}\right\|_{X(\Omega)},
$$
where $\mathcal{T}_j(f)$ is defined in \eqref{2218}.
To this end,
let $\phi\in C_{\mathrm{c}}^\infty(\Omega)$ satisfy
$\|\phi\|_{X'(\Omega)}\leq1$.
Let
$\widetilde{f}$ and $\widetilde{\phi}$
be defined the same as in \eqref{1448}
related, respectively, to $f$ and $\phi$.
Then $\widetilde{f}\in X(\mathbb{R}^n)$
and $\widetilde{\phi}\in C_{\mathrm{c}}^\infty(\mathbb{R}^n)$
with $\|\widetilde{\phi}\|_{X'(\mathbb{R}^n)}
=\|\phi\|_{X'(\Omega)}\leq1$.
From this, the assumption that
$X'(\mathbb{R}^n)$ is locally
$\beta$-doubling with $\beta\in(0,n+1)$,
and Lemma~\ref{1459} with both $f:=\widetilde{f}$
and $\phi:=\widetilde{\phi}$,
we deduce that, for any $j\in\{1,\ldots,n\}$,
\begin{align}\label{2141}
\left|\int_{\Omega}f(x)\partial_j\phi(x)\,dx\right|
&=\left|\int_{\mathbb{R}^n}\widetilde{f}(x)
\partial_j\widetilde{\phi}(x)\,dx\right|\\
&\lesssim\liminf_{\nu\to0^+}
\int_{\mathbb{R}^n}\left[
\int_{\mathbb{R}^n}\frac{|\widetilde{f}(x)-\widetilde{f}(y)|}{|x-y|}
\rho_\nu(|x-y|)\,dy
\right]\left|\widetilde{\phi}(x)\right|\,dx\nonumber\\
&=\liminf_{\nu\to0^+}
\int_{\Omega}\left[
\int_{\mathbb{R}^n}\frac{|\widetilde{f}(x)-\widetilde{f}(y)|}{|x-y|}
\rho_\nu(|x-y|)\,dy
\right]|\phi(x)|\,dx\nonumber\\
&=\liminf_{\nu\to0^+}
\int_{\Omega}\left[
\int_{\Omega}\frac{|f(x)-f(y)|}{|x-y|}
\rho_\nu(|x-y|)\,dy
\right]|\phi(x)|\,dx\nonumber\\
&\quad+\lim_{\nu\to0^+}
\int_{\mathrm{supp\,}(\phi)}\left[
\int_{\mathbb{R}^n\setminus\Omega}
\frac{\rho_\nu(|x-y|)}{|x-y|}\,dy
\right]|f(x)||\phi(x)|\,dx,\nonumber
\end{align}
where we claim that
$$
\lim_{\nu\to0^+}
\int_{\mathrm{supp\,}(\phi)}\left[
\int_{\mathbb{R}^n\setminus\Omega}
\frac{\rho_\nu(|x-y|)}{|x-y|}\,dy
\right]|f(x)||\phi(x)|\,dx=0.
$$
If $\Omega=\mathbb{R}^n$,
then the above claim holds true automatically.
Now, assume that $\Omega\subsetneqq\mathbb{R}^n$.
Let
$$
d:=\mathrm{dist\,}\left(\mathrm{supp\,}(\phi),
\Omega^\complement\right).
$$
Then $d\in(0,\infty)$ because $\Omega$ is open
and $\phi\in C_{\mathrm{c}}^\infty(\Omega)$.
By this, Lemma~\ref{1639},
the polar coordinate, and Definition~\ref{ATI}(iii),
we find that
\begin{align*}
&\int_{\mathrm{supp\,}(\phi)}\left[
\int_{\mathbb{R}^n\setminus\Omega}
\frac{\rho_\nu(|x-y|)}{|x-y|}\,dy
\right]|f(x)||\phi(x)|\,dx\\
&\quad\leq\frac{1}{d}\int_{\mathrm{supp\,}(\phi)}\left[
\int_{|h|\ge d}
\rho_\nu(|h|)\,dh
\right]|f(x)||\phi(x)|\,dx\\
&\quad=\frac{1}{d}\int_{\Omega}|{f}(x)||{\phi}(x)|\,dx
\int_{|h|\ge d}
\rho_\nu(|h|)\,dh\\
&\quad\leq\frac{1}{d}\|f\|_{X(\Omega)}\|\phi\|_{X'(\Omega)}
\int_{\mathbb{S}^{n-1}}\int_{d}^\infty
\rho_\nu(r)r^{n-1}\,dr\,d\sigma
\to0
\end{align*}
as $\nu\to0^+$,
which completes the proof of the above claim.

From the above claim, \eqref{2141}, Lemma~\ref{1639},
the H\"older inequality,
and Definition~\ref{ATI}(ii),
we infer that
\begin{align*}
&\left|\int_{\Omega}f(x)\partial_j\phi(x)\,dx\right|\\
&\quad\lesssim\liminf_{\nu\to0^+}
\int_{\Omega}\left[
\int_{\Omega}\frac{|f(x)-f(y)|}{|x-y|}
\rho_\nu(|x-y|)\,dy
\right]|\phi(x)|\,dx\\
&\quad\leq\liminf_{\nu\to0^+}
\left\|\int_{\Omega}\frac{|f(\cdot)-f(y)|}{|\cdot-y|}
\rho_\nu(|\cdot-y|)\,dy\right\|_{X(\Omega)}
\|\phi\|_{X'(\Omega)}\\
&\quad\leq\liminf_{\nu\to0^+}
\left\|\left[\int_{\Omega}\frac{|f(\cdot)-f(y)|^p}{|\cdot-y|^p}
\rho_\nu(|\cdot-y|)\,dy\right]^\frac{1}{p}
\left[\int_{\mathbb{R}^n}\rho_\nu
(|\cdot-y|)\,dy\right]^\frac{1}{p'}\right\|_{X(\Omega)}\\
&\quad\sim\liminf_{\nu\to0^+}
\left\|\left[\int_{\Omega}\frac{|f(\cdot)-f(y)|^p}{|\cdot-y|^p}
\rho_\nu(|\cdot-y|)\,dy\right]^\frac{1}{p}\right\|_{X(\Omega)}.
\end{align*}
Taking the supremum over all $\phi\in C_{\mathrm{c}}^\infty(\Omega)$
with $\|\phi\|_{X'(\Omega)}\leq1$,
we conclude that, for any $j\in\{1,\ldots,n\}$,
$$
\mathcal{T}_j(f)\lesssim\liminf_{\nu\to0^+}
\left\|\left[\int_{\Omega}\frac{|f(\cdot)-f(y)|^p}{|\cdot-y|^p}
\rho_\nu(|\cdot-y|)\,dy
\right]^\frac{1}{p}\right\|_{X(\Omega)}<\infty,
$$
which, combined with Lemma~\ref{2212},
further implies that $|\nabla f|\in X(\Omega)$
and
$$
\|\,|\nabla f|\,\|_{X(\Omega)}\sim\sum_{j=1}^n\mathcal{T}_j(f)
\lesssim\liminf_{\nu\to0^+}
\left\|\left[\int_{\Omega}\frac{|f(\cdot)-f(y)|^p}{|\cdot-y|^p}
\rho_\nu(|\cdot-y|)\,dy
\right]^\frac{1}{p}\right\|_{X(\Omega)}.
$$
This finishes the proof of Theorem~\ref{BBMdomain}.
\end{proof}

\begin{remark}
In Theorem~\ref{BBMdomain}, if $X(\mathbb{R}^n):=L^q(\mathbb{R}^n)$
with $q\in(1,\infty)$ and if $p\in[1,q]$,
then it is easy to show that $L^q(\mathbb{R}^n)$
is a ball Banach function space,
$L^{q'}(\mathbb{R}^n)$ has an absolutely
continuous norm, $C_{\mathrm{c}}^\infty(\Omega)$
is dense in $L^{q'}(\Omega)$
(using Lemma~\ref{dense} with $X:=L^{q'}$),
and $L^{q'}(\mathbb{R}^n)$ is locally $n$-doubling
(which can be deduced from \cite[Lemma~2.18]{dgpyyz2022})
and hence Theorem~\ref{BBMdomain} in this case holds true.
Thus, we conclude that \cite[(36)]{b2002} is just
Theorem~\ref{BBMdomain} with
both $X(\mathbb{R}^n)=L^q(\mathbb{R}^n)$
and $q=p\in(1,\infty)$
and that Theorem~\ref{BBMdomain}
with both $X(\mathbb{R}^n)=L^q(\mathbb{R}^n)$
and $1\leq p<q<\infty$ is new.
\end{remark}

The following is another main theorem of this
section, which presents a Bourgain--Brezis--Mironescu-type
characterization
of $W^{1,X}(\Omega)$.

\begin{theorem}\label{1622}
Let $\Omega\subset\mathbb{R}^n$ be a bounded
$(\varepsilon,\infty)$-domain
with $\varepsilon\in(0,1]$,
$\{\rho_\nu\}_{\nu\in(0,\nu_0)}$
a $\nu_0$-{\rm RDATI} on $\mathbb{R}^n$
with $\nu_0\in(0,\infty)$,
and $p\in[1,\infty)$.
Assume that
\begin{enumerate}
\item[\textup{(i)}]
both $X(\mathbb{R}^n)$ and $X^\frac{1}{p}(\mathbb{R}^n)$
are ball Banach function spaces;
\item[\textup{(ii)}]
both $X(\mathbb{R}^n)$ and $X'(\mathbb{R}^n)$
have absolutely continuous norms;
\item[\textup{(iii)}]
the Hardy--Littlewood maximal operator
is bounded on $[X^\frac{1}{p}(\mathbb{R}^n)]'$.
\end{enumerate}
Then $f\in W^{1,X}(\Omega)$ if and only if $f\in X(\Omega)$
and
\begin{align}\label{1643}
\liminf_{\nu\to0^+}
\left\|\left[\int_\Omega\frac{|f(\cdot)-f(y)|^p}{
|\cdot-y|^p}\rho_\nu(|\cdot-y|)\,dy
\right]^\frac{1}{p}\right\|_{X(\Omega)}<\infty;
\end{align}
moreover, for such $f$,
\begin{align*}
\lim_{\nu\to0^+}
\left\|\left[\int_\Omega\frac{|f(\cdot)-f(y)|^p}{
|\cdot-y|^p}\rho_\nu(|\cdot-y|)\,dy
\right]^\frac{1}{p}\right\|_{X(\Omega)}
=\left[\kappa(p,n)\right]^\frac{1}{p}
\left\|\,\left|\nabla f\right|\,\right\|_{X(\Omega)},
\end{align*}
where the constant
$\kappa(p,n)$ is the same as in \eqref{kappaqn}.
\end{theorem}

To show Theorem~\ref{1622},
we need the following density lemma
which when $\Omega=\mathbb{R}^n$
is just \cite[Corollary~3.10]{dgpyyz2022}.

\begin{lemma}\label{dense}
Let $\Omega\subset\mathbb{R}^n$ be an open set and
$X(\mathbb{R}^n)$ a ball Banach function space
having an absolutely continuous norm.
Then $C_{\mathrm{c}}^\infty(\Omega)$ is
dense in $X(\Omega)$.
\end{lemma}

\begin{proof}
Without loss of generality,
we may assume that $f\in X(\Omega)$
is a nonnegative measurable function.
Let $\widetilde{f}$ be defined in \eqref{1448}.
By \cite[p.\,18, Theorem~3.11]{bs1988}
(whose proof remains true for ball Banach function spaces) and
the assumption that $X(\mathbb{R}^n)$ has
an absolutely continuous norm,
we find that the set of all
simple functions is dense in $X(\mathbb{R}^n)$,
which further implies that,
for any given $\varepsilon\in(0,\infty)$,
there exist $N\in\mathbb{N}$,
a sequence $\{\lambda_k\}_{k=1}^N\subset(0,\infty)$,
a sequence $\{E_k\}_{k\in\mathbb{N}}$ of
measurable subsets of $\Omega$,
and a nonnegative simple function
$g:=\sum_{k=1}^N\lambda_k\mathbf{1}_{E_k}$ such that
\begin{align}\label{1727}
\left\|\widetilde{f}-g\right\|_{X(\mathbb{R}^n)}<\frac{\varepsilon}{3}.
\end{align}
Let
$\mathring{g}:=\sum_{k=1}^N\lambda_k\mathbf{1}_{\mathring{E_k}}$,
where, for any $k\in\{1,\ldots,N\}$, $\mathring{E_k}$
denotes the interior of $E_k$.
From the inner regularity of the Lebesgue measure
(see, for instance, \cite[Theorem~2.14(d)]{rudin})
and \cite[Lemma~5.6.14]{lyh2320}
(see also \cite[p.\,16, Proposition~3.6]{bs1988}),
we infer that there exists a
simple function $h:=\sum_{k=1}^N\lambda_k\mathbf{1}_{F_k}$
such that
\begin{align}\label{1728}
\left\|\mathring{g}-h\right\|_{X(\mathbb{R}^n)}<\frac{\varepsilon}{3},
\end{align}
where, for any $k\in\{1,\ldots,N\}$,
$F_k\subset\mathring{E_k}$ is a compact set.
Let
$\mathring{h}:=\sum_{k=1}^N\lambda_k\mathbf{1}_{\mathring{F_k}}$.
By the inner regularity of the Lebesgue measure
(see, for instance, \cite[Theorem~2.14(d)]{rudin})
and \cite[Lemma~5.6.14]{lyh2320}
(see also \cite[p.\,16, Proposition~3.6]{bs1988}) again,
we find that there exists a
simple function $u:=\sum_{k=1}^N\lambda_k\mathbf{1}_{U_k}$
such that
\begin{align}\label{1729}
\left\|\mathring{h}-u\right\|_{X(\mathbb{R}^n)}<\frac{\varepsilon}{3},
\end{align}
where, for any $k\in\{1,\ldots,N\}$,
$U_k\subset\mathring{F_k}$ is a compact set.
Applying the Urysohn lemma
(see, for instance, \cite[Lemma~8.18]{folland})
to the open set $\bigcup_{k=1}^N\mathring{F}_k$
and the compact set $\bigcup_{k=1}^NU_k$,
we conclude that there exists
$f_0\in C_{\mathrm{c}}^\infty(\mathbb{R}^n)$
satisfying that, for any $x\in\mathbb{R}^n$,
$$
u(x)\leq f_0(x)\leq\mathring{h}(x).
$$
By this, $F_k\subset\mathring{E}_k$, and $E_k\subset\Omega$,
we conclude that
$$
\mathrm{supp\,}(f_0)\subset\mathrm{supp\,}(h)=\bigcup_{k=1}^N
F_k\subsetneqq\bigcup_{k=1}^N\mathring{E}_k\subset\Omega,
$$
which, together with the fact that
$\bigcup_{k=1}^N{F}_k$ is compact,
further implies that
$f_0|_\Omega\in C_{\mathrm{c}}^\infty(\Omega)$.
Moreover, from
Proposition~\ref{norm},
Definition~\ref{1659}(v), \eqref{1727}, \eqref{1728},
and \eqref{1729}, we deduce that
\begin{align*}
\left\|f_0|_\Omega-f\right\|_{X(\Omega)}
&=\left\|f_0-\widetilde{f}\right\|_{X(\mathbb{R}^n)}\\
&\leq\left\|f_0-\mathring{h}\right\|_{X(\mathbb{R}^n)}
+\left\|\mathring{h}-h\right\|_{X(\mathbb{R}^n)}
+\left\|h-\mathring{g}\right\|_{X(\mathbb{R}^n)}\\
&\quad+\left\|\mathring{g}-g\right\|_{X(\mathbb{R}^n)}
+\left\|g-\widetilde{f}\right\|_{X(\mathbb{R}^n)}\\
&\leq\left\|u-\mathring{h}\right\|_{X(\mathbb{R}^n)}
+\left\|h-\mathring{g}\right\|_{X(\mathbb{R}^n)}
+\left\|g-\widetilde{f}\right\|_{X(\mathbb{R}^n)}<\varepsilon,
\end{align*}
which completes the proof of Lemma~\ref{dense}.
\end{proof}

Now, we use Lemma~\ref{dense}
and Theorems~\ref{2045}
and~\ref{BBMdomain} to show Theorem~\ref{1622}.

\begin{proof}[Proof of Theorem~\ref{1622}]
We first show the necessity.
Let $f\in W^{1,X}(\Omega)$.
Then $f\in X(\Omega)$ and Theorem~\ref{2045} implies that
\begin{align*}
&\lim_{\nu\to0^+}
\left\|\left[\int_\Omega\frac{|f(\cdot)-f(y)|^p}{
|\cdot-y|^p}\rho_\nu(|\cdot-y|)\,dy
\right]^\frac{1}{p}\right\|_{X(\Omega)}\\
&\quad=\left[\kappa(p,n)\right]^\frac{1}{p}
\left\|\,\left|\nabla f\right|\,\right\|_{X(\Omega)}<\infty.
\end{align*}
This finishes the proof of the necessity.

Next, we show the sufficiency.
Let $f\in X(\Omega)$ satisfy \eqref{1643}.
By Lemma~\ref{1009} and the assumptions (i)
and (iii) of the present theorem,
we find that the centered ball average
operators $\{\mathcal{B}_r\}_{r\in(0,\infty)}$
are uniformly bounded on $X(\mathbb{R}^n)$,
which, combined with
\cite[Lemmas~4.9 and~4.10]{dgpyyz2022},
further implies that
$X'(\mathbb{R}^n)$ is locally $n$-doubling.
From Lemma~\ref{dense} with $X(\Omega):=X'(\Omega)$ and
the assumption (ii) of the present theorem,
we infer that $C_{\mathrm{c}}^\infty(\Omega)$
is dense in $X'(\Omega)$.
By Proposition~\ref{abs} and the assumption (ii)
of the present theorem,
we find that $X'(\Omega)$
has an absolutely continuous norm on $\Omega$.
Thus, all the assumptions of Theorem~\ref{BBMdomain}
are satisfied. From Theorem~\ref{BBMdomain},
we deduce that $f\in W^{1,X}(\Omega)$.
This finishes the proof of the sufficiency
and hence Theorem~\ref{1622}.
\end{proof}

\begin{remark}\label{1607}
In Theorem~\ref{1622}, if $X(\mathbb{R}^n):=L^q(\mathbb{R}^n)$
with $q\in(1,\infty)$ and if $p\in[1,q]$,
then it is easy to show that both $L^q(\mathbb{R}^n)$
and $L^{q'}(\mathbb{R}^n)$
are a ball Banach function space having an absolutely
continuous norm
and that the Hardy--Littlewood maximal operator
is bounded on $[X^\frac{1}{p}(\mathbb{R}^n)]'
=L^{(\frac{q}{p})'}(\mathbb{R}^n)$
and hence Theorem~\ref{1622} in this case holds true.
Moreover, we point out that \cite[Theorem~2']{b2002} is just
Theorem~\ref{1622} with both $X(\mathbb{R}^n)=L^q(\mathbb{R}^n)$
and $q=p\in(1,\infty)$
and that Theorem~\ref{1622} with both $X(\mathbb{R}^n)=L^q(\mathbb{R}^n)$
and $1\leq p<q<\infty$ is new.
In addition, from Remark~\ref{1638}(iii),
we infer that
the assumption on $\Omega$ in Theorem~\ref{1622}
is weaker than that in \cite[Remark~7]{b2002}.
\end{remark}

Using Theorem~\ref{1622},
we obtain the following corollary.

\begin{corollary}\label{2001}
Let $p\in[1,\infty)$,
$\Omega\subset\mathbb{R}^n$ be a bounded
$(\varepsilon,\infty)$-domain
with $\varepsilon\in(0,1]$,
and $X(\mathbb{R}^n)$ a ball Banach function space
satisfying the same assumptions in Theorem~\ref{1622}.
Then $f\in W^{1,X}(\Omega)$ if and only if $f\in X(\Omega)$
and
\begin{align*}
\liminf_{s\to1^-}(1-s)^\frac{1}{p}\left\|\left[\int_{\Omega}
\frac{|f(\cdot)-f(y)|^p}{|\cdot-y|^{n+sp}}
\,dy\right]^{\frac{1}{p}}\right\|_{X(\Omega)}<\infty;
\end{align*}
moreover, for such $f$,
\begin{align*}
\lim_{s\to1^-}(1-s)^\frac{1}{p}\left\|\left[\int_{\Omega}
\frac{|f(\cdot)-f(y)|^p}{|\cdot-y|^{n+sp}}
\,dy\right]^{\frac{1}{p}}\right\|_{X(\Omega)}
=\left[\frac{\kappa(p,n)}{p}\right]^\frac{1}{p}
\left\|\,\left|\nabla f\right|\,\right\|_{X(\Omega)},
\end{align*}
where the constant
$\kappa(p,n)$ is the same as in \eqref{kappaqn}.
\end{corollary}

\begin{proof}
Since $\Omega$ is bounded,
it follows that there exists $R\in(0,\infty)$ such that
$\Omega\subset B(\mathbf{0},R/2)$.
Let $\nu_0:=\min\{n/p,\,1\}$.
For any $\nu\in(0,\nu_0)$ and $r\in(0,\infty)$, define
\begin{align}\label{rhov}
\rho_\nu(r):=\nu p(2R)^{-\nu p}
r^{-n+\nu p}\mathbf{1}_{(0,2R]}(r).
\end{align}
Next, we claim that
$\{\rho_\nu\}_{\nu\in(0,\nu_0)}$
is a $\nu_0$-RDATI in Definition~\ref{ATI}.
By \eqref{rhov}, it is easy to show that,
for any $\nu\in(0,\nu_0)$,
$\rho_\nu$ is a locally integrable,
nonnegative, and decreasing function on $(0,\infty)$.
From a simple computation, we deduce that
\begin{align*}
\int_0^\infty\rho_\nu(r)r^{n-1}\,dr=1
\end{align*}
for any $\nu\in(0,\nu_0)$ and
\begin{align*}
\lim_{\nu\to0^+}
\int_\delta^\infty\rho_\nu(r)r^{n-1}\,dr
=1-\lim_{\nu\to0^+}
\left(\frac{\delta}{2R}\right)^{\nu p}=0
\end{align*}
for any given $\delta\in(0,\infty)$.
Thus, $\{\rho_\nu\}_{\nu\in(0,\nu_0)}$
satisfies all the conditions in Definition~\ref{ATI}
and hence $\{\rho_\nu\}_{\nu\in(0,\nu_0)}$
is a $\nu_0$-RDATI. This finishes the proof of the above claim.
By this claim and the assumptions of the present corollary,
we conclude that all the assumptions of Theorem~\ref{1622}
are satisfied.

On the other hand, from the definition of
$\{\rho_\nu\}_{\nu\in(0,\nu_0)}$
and $\Omega\subset B(\mathbf{0},R/2)$
and letting $s:=1-\nu$,
we infer that
\begin{align*}
&\liminf_{\nu\to0^+}
\left\|\left[\int_{\Omega}\frac{|f(\cdot)-f(y)|^p}{
|\cdot-y|^p}\rho_\nu(|\cdot-y|)\,dy
\right]^\frac{1}{p}\right\|_{X(\Omega)}\\
&\quad=\liminf_{\nu\to0^+}(\nu p)^\frac{1}{p}
(2R)^{-\nu}
\left\|\left[\int_{\Omega}\frac{|f(\cdot)-f(y)|^p}{
|\cdot-y|^{n+(1-\nu)p}}\,dy
\right]^\frac{1}{p}\right\|_{X(\Omega)}\\
&\quad=p^\frac{1}{p}\liminf_{s\to1^-}(2R)^{s-1}
(1-s)^\frac{1}{p}
\left\|\left[\int_{\Omega}\frac{|f(\cdot)-f(y)|^p}{
|\cdot-y|^{n+sp}}\,dy
\right]^\frac{1}{p}\right\|_{X(\Omega)}\\
&\quad=p^\frac{1}{p}
\liminf_{s\to1^-}(1-s)^\frac{1}{p}
\left\|\left[\int_{\Omega}\frac{|f(\cdot)-f(y)|^p}{
|\cdot-y|^{n+sp}}\,dy
\right]^\frac{1}{p}\right\|_{X(\Omega)}.
\end{align*}
By this and Theorem~\ref{1622},
we complete the proof of Corollary~\ref{2001}.
\end{proof}

\begin{remark}
In Corollary~\ref{2001},
if $X(\mathbb{R}^n):=L^q(\mathbb{R}^n)$
with $q\in(1,\infty)$ and $p\in[1,q]$,
then, by an argument similar to that used
in Remark~\ref{1607}, we conclude that
Corollary~\ref{2001} in this case holds true.
Moreover, we point out that \cite[(44)]{b2002}
(see also \cite[Corollary~4]{b2002}) is just
Corollary~\ref{2001} with both $X(\mathbb{R}^n)=L^q(\mathbb{R}^n)$
and $q=p\in(1,\infty)$
and that Corollary~\ref{2001} with
both $X(\mathbb{R}^n)=L^q(\mathbb{R}^n)$
and $1\leq p<q<\infty$ is new.
Moreover, from Remark~\ref{1638}(iii),
we deduce that
the assumption on $\Omega$ in Corollary~\ref{2001}
is weaker than that in \cite[Corollary~4]{b2002}.
\end{remark}

\section{Applications to Specific Function Spaces}
\label{S5}

In this section, we apply Theorem~\ref{1622}
to ten examples of ball
Banach function spaces, namely
weighted Lebesgue and
Morrey (see Subsection~\ref{5.4} below),
Besov--Bourgain--Morrey
(see Subsection~\ref{BBMspace} below),
local (or global) generalized Herz
(see Subsection~\ref{Herz} below),
mixed-norm Lebesgue (see Subsection~\ref{5.2} below),
variable Lebesgue (see Subsection~\ref{5.3} below),
Lorentz (see Subsection~\ref{5.7} below),
Orlicz (see Subsection~\ref{5.5} below),
and Orlicz-slice (see Subsection~\ref{5.6} below) spaces.

\subsection{Weighted Lebesgue Spaces
and Morrey Spaces}\label{5.4}

First, we consider the case of weighted Lebesgue spaces;
see Definition~\ref{1556}(i) for its definition.
As was pointed out in \cite[p.\,86]{shyy2017},
the weighted Lebesgue space
is a ball quasi-Banach function space,
but it may not be a Banach function space
in the sense of Bennett and Sharpley \cite{bs1988}.

By an argument similar to that used in
the proof of \cite[Theorem~4.1]{zyy20231}
with \cite[Theorems~3.1 and 3.11]{zyy20231} replaced by
Theorem~\ref{1622},
we obtain the following conclusion;
we omit the details.

\begin{theorem}\label{2044}
Let $\Omega\subset\mathbb{R}^n$ be a bounded
$(\varepsilon,\infty)$-domain
with $\varepsilon\in(0,1]$ and
$\{\rho_\nu\}_{\nu\in(0,\nu_0)}$
a $\nu_0$-{\rm RDATI} on $\mathbb{R}^n$
with $\nu_0\in(0,\infty)$.
Let $r\in(1,\infty)$, $p\in[1,r]$, and
$\omega\in A_{\frac{r}{p}}(\mathbb{R}^n)$.
Then $f\in W^{1,r}_\omega(\Omega)$ if and only if
$f\in L^{r}_\omega(\Omega)$ and
$$
\liminf_{\nu\to0^+}
\left\|\left[\int_\Omega\frac{|f(\cdot)-f(y)|^p}{
|\cdot-y|^p}\rho_\nu(|\cdot-y|)\,dy
\right]^\frac{1}{p}\right\|_{L^{r}_\omega(\Omega)}<\infty;
$$
moreover, for such $f$,
$$
\lim_{\nu\to0^+}
\left\|\left[\int_\Omega\frac{|f(\cdot)-f(y)|^p}{
|\cdot-y|^p}\rho_\nu(|\cdot-y|)\,dy
\right]^\frac{1}{p}\right\|_{L^{r}_\omega(\Omega)}
=\left[\kappa(p,n)\right]^\frac{1}{p}
\left\|\,\left|\nabla f\right|\,\right\|_{L^{r}_\omega(\Omega)},
$$
where the constant
$\kappa(p,n)$ is the same as in \eqref{kappaqn}.
\end{theorem}

\begin{remark}
\begin{enumerate}
\item[(i)]
To the best of our knowledge,
Theorem~\ref{2044} is completely new.
\item[(ii)]
Let $\omega\in A_1(\mathbb{R}^n)$.
In this case,
since $[L^1_\omega(\mathbb{R}^n)]'
=L^\infty_{\omega^{-1}}(\mathbb{R}^n)$
(see, for instance, \cite[p.\,9]{ins2019})
does not have an absolutely continuous
norm, it is still unclear whether or not
Theorem~\ref{2044} with $r=1$
holds true.
\end{enumerate}
\end{remark}

Next, we turn to study the case of Morrey spaces.
The \emph{Morrey space} $M_r^\alpha(\mathbb{R}^n)$
with $0<r\leq\alpha<\infty$
is defined to be the set of
all the $f\in\mathscr{M}(\mathbb{R}^n)$
having the following finite quasi-norm
\begin{align*}
\|f\|_{M_r^\alpha(\mathbb{R}^n)}:=
\sup_{x\in\mathbb{R}^n}\sup_{r\in(0,\infty)}
\left|B(x,r)\right|^{\frac{1}{\alpha}-
\frac{1}{r}}\|f\|_{L^r(B(x,r))},
\end{align*}
These spaces were introduced by Morrey \cite{m1938}
for the purpose of studying the regularity of solutions
to partial differential equations.
Morrey spaces have important
applications in the theory of partial differential equations,
harmonic analysis,
and potential theory;
see, for instance, the articles \cite{hms2017,hms2020,hss2018,
hs2017,tyy2019}
and the monographs \cite{sdh2020i,sdh20202,ysy2010}.
As was pointed out in \cite[p.\,87]{shyy2017},
$M_r^\alpha(\mathbb{R}^n)$ with $1\leq r\leq\alpha<\infty$
is a ball Banach function space,
but is not a Banach function space
in the sense of Bennett and Sharpley \cite{bs1988}.
Since the Morrey space $M_r^\alpha(\mathbb{R}^n)$
does not have an absolutely continuous norm,
Theorem~\ref{1622} does not seem to be applicable in
the Morrey setting.
However, we can use Proposition~\ref{1537}
and Theorem~\ref{2044} to further
obtain a characterization similar to Theorem~\ref{1622}
for Morrey--Sobolev spaces.

\begin{theorem}\label{Morrey}
Let $\Omega\subset\mathbb{R}^n$ be a bounded
$(\varepsilon,\infty)$-domain
with $\varepsilon\in(0,1]$ and
$\{\rho_\nu\}_{\nu\in(0,\nu_0)}$
a $\nu_0$-{\rm RDATI} on $\mathbb{R}^n$
with $\nu_0\in(0,\infty)$.
Let $1<r\leq\alpha<\infty$ and $p\in[1,r]$.
Then $f\in W^{1,M_r^\alpha}(\Omega)$ if and only if
$f\in M_r^\alpha(\Omega)$ and
\begin{align}\label{1534}
\liminf_{\nu\to0^+}
\left\|\left[\int_\Omega\frac{|f(\cdot)-f(y)|^p}{
|\cdot-y|^p}\rho_\nu(|\cdot-y|)\,dy
\right]^\frac{1}{p}\right\|_{M_r^\alpha(\Omega)}<\infty;
\end{align}
moreover, there exist two constants $C_1,C_2\in(0,\infty)$
such that, for such $f$,
\begin{align*}
C_1\left\|\,\left|\nabla f\right|\,\right\|_{M_r^\alpha(\Omega)}
&\leq\liminf_{\nu\to0^+}
\left\|\left[\int_\Omega\frac{|f(\cdot)-f(y)|^p}{
|\cdot-y|^p}\rho_\nu(|\cdot-y|)\,dy
\right]^\frac{1}{p}\right\|_{M_r^\alpha(\Omega)}\\
&\leq\limsup_{\nu\to0^+}
\left\|\left[\int_\Omega\frac{|f(\cdot)-f(y)|^p}{
|\cdot-y|^p}\rho_\nu(|\cdot-y|)\,dy
\right]^\frac{1}{p}\right\|_{M_r^\alpha(\Omega)}\\
&\leq C_2\left\|\,\left|\nabla f\right|\,\right\|_{M_r^\alpha(\Omega)},
\end{align*}
where the constants $C_1$ and $C_2$ depend only on
$\alpha$, $r$, $n$, $p$, and $\Omega$.
\end{theorem}

\begin{proof}
We first show the necessity.
Assume that $f\in W^{1,M_r^\alpha}(\Omega)$.
By \cite[p.\,87]{shyy2017}, we find that
both $M_r^\alpha(\mathbb{R}^n)$ and
$M_{r/p}^{\alpha/p}(\mathbb{R}^n)$ are ball Banach function spaces.
From \cite[Theorem~4.1]{st2015} and \cite[Theorem~3.1]{ch2014},
we deduce that $\mathcal{M}$
is bounded on $[M_{r/p}^{\alpha/p}(\mathbb{R}^n)]'$.
By this and Proposition~\ref{1537} with both $X:=M_r^\alpha$
and $\rho:=\rho_\nu$, we conclude that,
for any $\nu\in(0,\nu_0)$,
\begin{align*}
\left\|\left[\int_\Omega\frac{|f(\cdot)-f(y)|^p}{
|\cdot-y|^p}\rho_\nu(|\cdot-y|)\,dy
\right]^\frac{1}{p}\right\|_{M_r^\alpha(\Omega)}
\lesssim\left\|\,\left|\nabla f\right|\,\right\|_{M_r^\alpha(\Omega)},
\end{align*}
where the implicit positive constant
is independent of both $\nu$
and $f$. From this, we deduce that
\begin{align*}
&\liminf_{\nu\to0^+}
\left\|\left[\int_\Omega\frac{|f(\cdot)-f(y)|^p}{
|\cdot-y|^p}\rho_\nu(|\cdot-y|)\,dy
\right]^\frac{1}{p}\right\|_{M_r^\alpha(\Omega)}\\
&\quad\leq\limsup_{\nu\to0^+}
\left\|\left[\int_\Omega\frac{|f(\cdot)-f(y)|^p}{
|\cdot-y|^p}\rho_\nu(|\cdot-y|)\,dy
\right]^\frac{1}{p}\right\|_{M_r^\alpha(\Omega)}\nonumber\\
&\quad\lesssim\left\|\,\left|\nabla f\right|\,
\right\|_{M_r^\alpha(\Omega)}<\infty.\nonumber
\end{align*}
This finishes the proof of the necessity.

Next, we show the sufficiency.
Assume that $f\in M_r^\alpha(\Omega)$ satisfies \eqref{1534}.
Fix $\theta\in(1-\frac{r}{\alpha},1)$.
By \cite[Proposition~285]{sdh2020i},
we find that, for any $g\in\mathscr{M}(\mathbb{R}^n)$,
\begin{align}\label{2106}
\left\|g\right\|_{M_r^\alpha(\mathbb{R}^n)}\sim
\sup_{Q\subset\mathbb{R}^n}
|Q|^{\frac{1}{\alpha}-\frac{1}{r}}
\|g\|_{L^r_{[\mathcal{M}(\mathbf{1}_Q)]^\theta}(\mathbb{R}^n)},
\end{align}
where the supremum is taken over all cubes $Q\subset\mathbb{R}^n$
and the positive equivalence constants depend only on
$r$, $\alpha$, $n$, and $\theta$.
Moreover, we have, for any $Q\subset\mathbb{R}^n$
and $\nu\in(0,1)$,
\begin{align*}
&|Q|^{\frac{1}{\alpha}-\frac{1}{r}}
\left\|\left[\int_\Omega\frac{|f(\cdot)-f(y)|^p}{
|\cdot-y|^p}\rho_\nu(|\cdot-y|)\,dy
\right]^\frac{1}{p}\right\|_{
L^r_{[\mathcal{M}(\mathbf{1}_Q)]^\theta}(\Omega)}\\
&\quad\lesssim\left\|\left[\int_\Omega\frac{|f(\cdot)-f(y)|^p}{
|\cdot-y|^p}\rho_\nu(|\cdot-y|)\,dy
\right]^\frac{1}{p}\right\|_{M_r^\alpha(\Omega)},
\end{align*}
which further implies that
\begin{align}\label{1642}
&\liminf_{\nu\to0^+}|Q|^{\frac{1}{\alpha}-\frac{1}{r}}
\left\|\left[\int_\Omega\frac{|f(\cdot)-f(y)|^p}{
|\cdot-y|^p}\rho_\nu(|\cdot-y|)\,dy
\right]^\frac{1}{p}\right\|_{
L^r_{[\mathcal{M}(\mathbf{1}_Q)]^\theta}(\Omega)}\\
&\quad\lesssim\liminf_{\nu\to0^+}
\left\|\left[\int_\Omega\frac{|f(\cdot)-f(y)|^p}{
|\cdot-y|^p}\rho_\nu(|\cdot-y|)\,dy
\right]^\frac{1}{p}\right\|_{M_r^\alpha(\Omega)}.\nonumber
\end{align}
On the other hand,
from \cite[Theorem~281]{sdh2020i} and $\theta\in(0,1)$,
we deduce that,
for any cube  $Q\subset\mathbb{R}^n$,
$[\mathcal{M}(\mathbf{1}_Q)]^\theta\in A_1(\mathbb{R}^n)$
and
$[\{\mathcal{M}(\mathbf{1}_Q)\}^\theta
]_{A_1(\mathbb{R}^n)}\lesssim1$,
where the implicit positive constant depends only on $\theta$.
By this, \eqref{2106}, Theorem~\ref{2044}, and \eqref{1642},
we conclude that,
for any $f\in M_r^\alpha(\Omega)$,
\begin{align*}
\left\|\,|\nabla f|\,\right\|_{M_r^\alpha(\Omega)}
&\sim\sup_{Q\subset\mathbb{R}^n}
|Q|^{\frac{1}{\alpha}-\frac{1}{r}}
\left\|\,|\nabla f|\,\right\|_
{L^r_{[\mathcal{M}(\mathbf{1}_Q)]^\theta}(\Omega)}\\
&\sim\sup_{Q\subset\mathbb{R}^n}
|Q|^{\frac{1}{\alpha}-\frac{1}{r}}\\
&\quad\times\lim_{\nu\to0^+}
\left\|\left[\int_\Omega\frac{|f(\cdot)-f(y)|^p}{
|\cdot-y|^p}\rho_\nu(|\cdot-y|)\,dy
\right]^\frac{1}{p}\right\|_{L^r_{
[\mathcal{M}(\mathbf{1}_Q)]^\theta}(\Omega)}
\\
&\lesssim\liminf_{\nu\to0^+}
\left\|\left[\int_\Omega\frac{|f(\cdot)-f(y)|^p}{
|\cdot-y|^p}\rho_\nu(|\cdot-y|)\,dy
\right]^\frac{1}{p}\right\|_{M_r^\alpha(\Omega)}.
\end{align*}
This finishes the proof of Theorem~\ref{Morrey}.
\end{proof}

\begin{remark}
To the best of our knowledge,
Theorem~\ref{Morrey} is completely new.
\end{remark}

\subsection{Besov--Bourgain--Morrey Spaces}
\label{BBMspace}

Recall that Bourgain--Morrey spaces
(see \cite[Definition~1.1]{hnsh2022} for its definition),
whose special case was introduced by Bourgain \cite{bourgain},
play an important role in the study of both the
nonlinear Schr\"odinger equation and the
Strichartz estimate.
As a generalization of the Bourgain--Morrey space,
the Besov--Bourgain--Morrey space was
introduced by Zhao et al. \cite{zstyy2022},
which is a bridge connecting Bourgain--Morrey spaces
with amalgam-type spaces;
see \cite[p.\,525]{f1989} for the definition
of the amalgam-type space.
For more studies on Bourgain--Morrey-type spaces,
we refer the reader to
\cite{hnsh2022,hly2023,m1,m2,zstyy2022}.

Let $j\in\mathbb{Z}$
and $m:=(m_1,\ldots,m_n)\in\mathbb{Z}^n$.
Recall that a dyadic cube $Q_{j,m}\subset\mathbb{R}^n$
is defined by setting
$$
Q_{j,m}:=\prod_{i=1}^n\left[\frac{m_i}{2^j},
\frac{m_i+1}{2^j}\right).
$$
The definition of
the Besov--Bourgain--Morrey space
can be found in
\cite[Definition~1.2]{zstyy2022}.

\begin{definition}\label{2209}
Let $0<q\leq p\leq r\leq\infty$, $\tau\in(0,\infty]$,
and $\{Q_{j,m}\}_{j\in\mathbb{Z},\,
m\in\mathbb{Z}^n}$ be the collection of
all dyadic cubes of $\mathbb{R}^n$.
The \emph{Besov--Bourgain--Morrey space}
$\mathcal{M}\dot{B}^{p,\tau}_{q,r}(\mathbb{R}^n)$
is defined to be the set of all the
$f\in L_{\mathrm{loc}}^q(\mathbb{R}^n)$
such that
$$
\left\|f\right\|_{\mathcal{M}\dot{B}^{p,\tau}_{q,r}(\mathbb{R}^n)}
:=\left\{\sum_{j\in\mathbb{Z}}
\left[\sum_{m\in\mathbb{Z}^n}\left\{
|Q_{j,m}|^{\frac{1}{p}-\frac{1}{q}}
\left[\int_{Q_{j,m}}|f(y)|^q\,dy\right]^\frac{1}{q}\right\}^r
\right]^\frac{\tau}{r}\right\}^\frac{1}{\tau},
$$
with the usual modifications made when
$q=\infty$, $r=\infty$, or $\tau=\infty$,
is finite.
\end{definition}

By an argument similar to that used in
the proof of \cite[Theorem~4.12]{zyy20231}
with \cite[Theorems~3.1 and~3.11]{zyy20231} replaced by
Theorem~\ref{1622},
we obtain the following conclusion;
we omit the details.

\begin{theorem}\label{BBM}
Let $\Omega\subset\mathbb{R}^n$ be a bounded
$(\varepsilon,\infty)$-domain
with $\varepsilon\in(0,1]$ and
$\{\rho_\nu\}_{\nu\in(0,\nu_0)}$
a $\nu_0$-{\rm RDATI} on $\mathbb{R}^n$
with $\nu_0\in(0,\infty)$.
Let $1<q<p<r<\infty$, $\tau\in(1,\infty)$,
and $s\in[1,\min\{q,\,p,\,r,\,\tau\})$.
Then $f\in W^{1,\mathcal{M}\dot{B}^{p,\tau}_{q,r}}(\Omega)$
if and only if
$f\in\mathcal{M}\dot{B}^{p,\tau}_{q,r}(\Omega)$ and
$$
\liminf_{\nu\to0^+}
\left\|\left[\int_\Omega\frac{|f(\cdot)-f(y)|^p}{
|\cdot-y|^p}\rho_\nu(|\cdot-y|)\,dy
\right]^\frac{1}{p}\right\|_{
\mathcal{M}\dot{B}^{p,\tau}_{q,r}(\Omega)}<\infty;
$$
moreover, for such $f$,
$$
\lim_{\nu\to0^+}
\left\|\left[\int_\Omega\frac{|f(\cdot)-f(y)|^p}{
|\cdot-y|^p}\rho_\nu(|\cdot-y|)\,dy
\right]^\frac{1}{p}\right\|_{\mathcal{M}\dot{B}^{p,\tau}_{q,r}(\Omega)}
=\left[\kappa(p,n)\right]^\frac{1}{p}
\left\|\,\left|\nabla f\right|\,\right\|_{
\mathcal{M}\dot{B}^{p,\tau}_{q,r}(\Omega)},
$$
where the constant
$\kappa(p,n)$ is the same as in \eqref{kappaqn}.
\end{theorem}

\begin{remark}
\begin{enumerate}
\item[\textup{(i)}]
To the best of our knowledge,
Theorem~\ref{BBM} is completely new.
\item[\textup{(ii)}]
Let $\tau\in[1,\infty)$ and $1=q<p<r<\infty$ or let
$\tau=1\leq q<p<r<\infty$.
In both cases, since the associate space
of $\mathcal{M}\dot{B}^{p,\tau}_{q,r}(\mathbb{R}^n)$ is unknown,
it is still unclear whether or not Theorem~\ref{BBM} when $s=1$
holds true.
\end{enumerate}
\end{remark}

\subsection{Local and Global Generalized Herz Spaces}
\label{Herz}

We begin with recalling several concepts related to
local and global generalized Herz spaces.
A nonnegative function $\omega$ on $\mathbb{R}_+$
is said to be \emph{almost increasing}
(resp. \emph{almost decreasing})
if there exists a constant $C\in[1,\infty)$ such that,
for any $0<t\leq\tau<\infty$ (resp. $0<\tau\leq t<\infty$),
$\omega(t)\leq C\omega(\tau)$.
The \emph{function class} $M(\mathbb{R}_+)$ is defined to be
the set of all the positive functions $\omega$ on $\mathbb{R}_+$
such that, for any $0<\delta<N<\infty$,
$$
0<\inf_{t\in(\delta,N)}\omega(t)
\leq\sup_{t\in(\delta,N)}\omega(t)<\infty
$$
and that there exist four constants
$\alpha_0,\beta_0,\alpha_\infty,
\beta_\infty\in\mathbb{R}$ such that
\begin{enumerate}
\item[(i)]
for any $t\in(0,1]$,
$t^{-\alpha_0}\omega(t)$ is almost increasing
and $t^{-\beta_0}\omega(t)$ is almost decreasing;
\item[(ii)]
for any $t\in[1,\infty)$,
$t^{-\alpha_\infty}\omega(t)$ is almost increasing
and $t^{-\beta_\infty}\omega(t)$ is almost decreasing.
\end{enumerate}
The \emph{Matuszewska--Orlicz indices} $m_0(\omega)$, $M_0(\omega)$,
$m_\infty(\omega)$, and $M_\infty(\omega)$
of a positive function $\omega$ on $\mathbb{R}_+$
are defined, respectively, by setting
\begin{align*}
m_0(\omega):=\sup_{t\in(0,1)}\frac{\ln\,[\limsup\limits_{h\to0^+}
\frac{\omega(ht)}{\omega(h)}]}{\ln t},\
M_0(\omega):=\inf_{t\in(0,1)}\frac{
\ln\,[\liminf\limits_{h\to0^+}
\frac{\omega(ht)}{\omega(h)}]}{\ln t},
\end{align*}
\begin{align*}
m_\infty(\omega):=\sup_{t\in(1,\infty)}\frac{
\ln\,[\liminf\limits_{h\to\infty}
\frac{\omega(ht)}{\omega(h)}]}{\ln t},
\end{align*}
and
\begin{align*}
M_\infty(\omega):=\inf_{t\in(1,\infty)}
\frac{\ln\,[\limsup\limits_{h\to\infty}
\frac{\omega(ht)}{\omega(h)}]}{\ln t}.
\end{align*}
Next, we recall the concepts of both the local generalized Herz space
and the global generalized Herz space,
which can be found in \cite[Definitions~1.2.1
and  1.7.1]{lyh2320}
(see also \cite[Definition~2.2]{rs2020}).
For any $\xi\in\mathbb{R}^n$ and $k\in\mathbb{Z}$,
let $R_{\xi,k}:=B(\xi,2^k)\setminus B(\xi,2^{k-1})$.

\begin{definition}\label{4.8}
Let $p,q\in(0,\infty]$ and $\omega\in M(\mathbb{R}_+)$.
\begin{enumerate}
\item[(i)]
The \emph{local generalized Herz space}
$\dot{\mathcal{K}}^{p,q}_{\omega,\xi}(\mathbb{R}^n)$,
with a given $\xi\in\mathbb{R}^n$,
is defined to be the set of all the $f\in L_{\mathrm{loc}}^p
(\mathbb{R}^n\setminus\{\xi\})$ having the following finite quasi-norm
\begin{align*}
\|f\|_{\dot{\mathcal{K}}^{p,q}_{\omega,\xi}(\mathbb{R}^n)}
:=\left\{\sum_{k\in\mathbb{Z}}
\left[\omega(2^k)\right]^q
\left\|f\right\|_{L^p(R_{\xi,k})}^q
\right\}^\frac{1}{q}.
\end{align*}
\item[(ii)]
The \emph{global generalized Herz space}
$\dot{\mathcal{K}}^{p,q}_{\omega}(\mathbb{R}^n)$
is defined to be the set of all the $f\in L_{\mathrm{loc}}^p
(\mathbb{R}^n)$ having the following finite quasi-norm
\begin{align*}
\|f\|_{\dot{\mathcal{K}}^{p,q}_{\omega}(\mathbb{R}^n)}
:=\sup_{\xi\in\mathbb{R}^n}
\|f\|_{\dot{\mathcal{K}}^{p,q}_{\omega,\xi}(\mathbb{R}^n)}.
\end{align*}
\end{enumerate}
\end{definition}

Recall that the classical Herz space was originally
introduced by Herz \cite{herz} in order to
study the Bernstein theorem on absolutely
convergent Fourier transforms.
The local and the global generalized Herz spaces
in Definition~\ref{4.8}, introduced by
Rafeiro and Samko \cite{rs2020},
generalize the classical homogeneous
Herz spaces and connect with generalized
Morrey type spaces.
For more studies on Herz spaces,
we refer the reader to \cite{gly1998,hy1999,hwyy2023,ly1996,
lyh2320,rs2020,zyz2022}.
As was pointed out in \cite[Theorem~1.2.46]{lyh2320},
for any $\xi\in\mathbb{R}^n$,
when $p,q\in[1,\infty]$ and $\omega\in M(\mathbb{R}_+)$
satisfy $-\frac{n}{p}<m_0(\omega)\leq M_0(\omega)<n-\frac{n}{p}$,
the local generalized Herz space
$\dot{\mathcal{K}}^{p,q}_{\omega,\xi}(\mathbb{R}^n)$
is a ball Banach function space.
Moreover, as was pointed out in \cite[Theorem~1.2.48]{lyh2320},
when $p,q\in[1,\infty]$ and
$\omega\in M(\mathbb{R}_+)$ satisfy both
$m_0(\omega)\in(-\frac{n}{p},\infty)$
and $M_\infty(\omega)\in(-\infty,0)$,
the global generalized Herz space
$\dot{\mathcal{K}}^{p,q}_{\omega}(\mathbb{R}^n)$ is
a ball Banach function space.

By an argument similar to that used in
the proof of \cite[Theorem~4.15]{zyy20231}
with \cite[Theorems~3.1 and~3.11]{zyy20231} replaced by
Theorem~\ref{1622},
we obtain the following conclusion;
we omit the details.

\begin{theorem}\label{localHerz}
Let $\Omega\subset\mathbb{R}^n$ be a bounded
$(\varepsilon,\infty)$-domain
with $\varepsilon\in(0,1]$ and
$\{\rho_\nu\}_{\nu\in(0,\nu_0)}$
a $\nu_0$-{\rm RDATI} on $\mathbb{R}^n$
with $\nu_0\in(0,\infty)$.
Let $\xi\in\mathbb{R}^n$,
$p,q\in(1,\infty)$, $s\in[1,\min\{p,\,q\}]$, and
$\omega\in M(\mathbb{R}_+)$
satisfy
\begin{align*}
-\frac{n}{p}<m_0(\omega)\leq
M_0(\omega)<n\left(\frac{1}{s}-\frac{1}{p}\right)
\end{align*}
and
\begin{align*}
-\frac{n}{p}<m_\infty(\omega)\leq M_\infty(\omega)
<n\left(\frac{1}{s}-\frac{1}{p}\right).
\end{align*}
Then $f\in W^{1,\dot{\mathcal{K}}^{p,q}_{\omega,
\xi}}(\Omega)$ if and only if
$f\in\dot{\mathcal{K}}^{p,q}_{\omega,
\xi}(\Omega)$ and
$$
\liminf_{\nu\to0^+}
\left\|\left[\int_\Omega\frac{|f(\cdot)-f(y)|^p}{
|\cdot-y|^p}\rho_\nu(|\cdot-y|)\,dy
\right]^\frac{1}{p}\right\|_{\dot{\mathcal{K}}^{p,q}_{\omega,
\xi}(\Omega)}<\infty;
$$
moreover, for such $f$,
$$
\lim_{\nu\to0^+}
\left\|\left[\int_\Omega\frac{|f(\cdot)-f(y)|^p}{
|\cdot-y|^p}\rho_\nu(|\cdot-y|)\,dy
\right]^\frac{1}{p}\right\|_{\dot{\mathcal{K}}^{p,q}_{\omega,
\xi}(\Omega)}
=\left[\kappa(p,n)\right]^\frac{1}{p}
\left\|\,\left|\nabla f\right|\,\right\|_{
\dot{\mathcal{K}}^{p,q}_{\omega,
\xi}(\Omega)},
$$
where the constant
$\kappa(p,n)$ is the same as in \eqref{kappaqn}.
\end{theorem}

\begin{remark}
\begin{enumerate}
\item[(i)]
To the best of our knowledge,
Theorem~\ref{localHerz} is completely new.
\item[(ii)]
Let $p=1$ and $q\in[1,\infty)$
or let $p\in(1,\infty)$ and $q=1$.
Let $\gamma$, $s$, and $\omega$ be
the same as in Theorem~\ref{localHerz}.
In both cases,
since $[\dot{\mathcal{K}}^{p,q}_{\omega,\xi}
(\mathbb{R}^n)]'$
does not have an absolutely continuous norm,
it is still unclear whether or not
Theorem~\ref{localHerz} holds true.
\end{enumerate}
\end{remark}

Next, we turn to study the case of
global generalized Herz spaces.
Since the global generalized Herz space does not
have an absolutely continuous norm,
Theorem~\ref{1622} seems inapplicable
in this setting. However,
applying Proposition~\ref{1537} and Theorem~\ref{localHerz},
we obtain the following characterization
similar to Theorem~\ref{1622} for
global generalized Herz--Sobolev spaces.

\begin{theorem}\label{globalHerz}
Let $\Omega\subset\mathbb{R}^n$ be a bounded
$(\varepsilon,\infty)$-domain
with $\varepsilon\in(0,1]$ and
$\{\rho_\nu\}_{\nu\in(0,\nu_0)}$
a $\nu_0$-{\rm RDATI} on $\mathbb{R}^n$
with $\nu_0\in(0,\infty)$.
Let $p,q\in(1,\infty)$, $s\in[1,\min\{p,\,q\})$, and
$\omega\in M(\mathbb{R}_+)$
satisfy
\begin{align}\label{1859}
-\frac{n}{p}<m_0(\omega)\leq
M_0(\omega)<0
\end{align}
and
\begin{align*}
-\frac{n}{p}<m_\infty(\omega)\leq M_\infty(\omega)<
n\left(\frac{1}{s}-\frac{1}{p}\right).
\end{align*}
Then $f\in W^{1,\dot{\mathcal{K}}^{p,q}_{\omega}}(\Omega)$
if and only if
$f\in\dot{\mathcal{K}}^{p,q}_{\omega}(\Omega)$ and
\begin{align}\label{1911}
\liminf_{\nu\to0^+}
\left\|\left[\int_\Omega\frac{|f(\cdot)-f(y)|^p}{
|\cdot-y|^p}\rho_\nu(|\cdot-y|)\,dy
\right]^\frac{1}{p}\right\|_{
\dot{\mathcal{K}}^{p,q}_{\omega}(\Omega)}<\infty;
\end{align}
moreover, there exists a positive constant $C$,
depending only on $p$, $q$, $s$, $n$, $\omega$, and $\Omega$,
such that, for such $f$,
\begin{align*}
\left\|\,\left|\nabla f\right|\,\right\|_{
\dot{\mathcal{K}}^{p,q}_{\omega}(\Omega)}
&\leq\liminf_{\nu\to0^+}
\left\|\left[\int_\Omega\frac{|f(\cdot)-f(y)|^p}{
|\cdot-y|^p}\rho_\nu(|\cdot-y|)\,dy
\right]^\frac{1}{p}\right\|_{
\dot{\mathcal{K}}^{p,q}_{\omega}(\Omega)}\\
&\leq\limsup_{\nu\to0^+}
\left\|\left[\int_\Omega\frac{|f(\cdot)-f(y)|^p}{
|\cdot-y|^p}\rho_\nu(|\cdot-y|)\,dy
\right]^\frac{1}{p}\right\|_{
\dot{\mathcal{K}}^{p,q}_{\omega}(\Omega)}\\
&\leq C\left\|\,\left|\nabla f\right|\,\right\|_{
\dot{\mathcal{K}}^{p,q}_{\omega}(\Omega)}.
\end{align*}
\end{theorem}

\begin{proof}
We first show the sufficiency.
Assume that $f\in\dot{\mathcal{K}}^{p,q}_{\omega}(\Omega)$
satisfies \eqref{1911}.
By Definition~\ref{4.8}(ii) and Theorem~\ref{localHerz},
we conclude that,
for any $f\in\dot{\mathcal{K}}^{p,q}_{\omega}(\Omega)$,
\begin{align*}
\left\|\,|\nabla f|\,\right\|_{\dot{\mathcal{K}}^{p,q}_{\omega}
(\Omega)}
&=\sup_{\xi\in\mathbb{R}^n}
\left\|\,\left|\nabla f\right|\,\right\|_{
\dot{\mathcal{K}}^{p,q}_{\omega,
\xi}(\Omega)}\\
&=\sup_{\xi\in\mathbb{R}^n}
\lim_{\nu\to0^+}
\left\|\left[\int_\Omega\frac{|f(\cdot)-f(y)|^p}{
|\cdot-y|^p}\rho_\nu(|\cdot-y|)\,dy
\right]^\frac{1}{p}\right\|_{\dot{\mathcal{K}}^{p,q}_{\omega,
\xi}(\Omega)}\\
&\leq\liminf_{\nu\to0^+}
\left\|\left[\int_\Omega\frac{|f(\cdot)-f(y)|^p}{
|\cdot-y|^p}\rho_\nu(|\cdot-y|)\,dy
\right]^\frac{1}{p}\right\|_{
\dot{\mathcal{K}}^{p,q}_{\omega}(\Omega)}
<\infty.
\end{align*}
This finishes the proof of the sufficiency.

Next, we show the necessity.
Assume that $f\in W^{1,\dot{\mathcal{K}}^{p,q}_{\omega}}(\Omega)$.
By \cite[Lemma~1.3.2]{lyh2320},
we find that
$$
\left[\dot{\mathcal{K}}^{p,q}_{\omega}(\mathbb{R}^n)\right]^\frac{1}{s}
=\dot{\mathcal{K}}^{\frac{p}{s},\frac{q}{s}}_{\omega^s}(\mathbb{R}^n),
$$
which, combined with both \eqref{1859} and
\cite[Theorem~1.2.48]{lyh2320},
further implies that both
$\dot{\mathcal{K}}^{p,q}_{\omega}(\mathbb{R}^n)$
and $[\dot{\mathcal{K}}^{p,q}_{\omega}(\mathbb{R}^n)]^\frac{1}{s}$
are ball Banach function spaces.
From \cite[Theorem~2.2.1]{lyh2320},
we infer that the associate space of
the global generalized Herz space is a block space
(see \cite[Definition~2.1.3]{lyh2320}
for the definition of the block space),
on which $\mathcal{M}$ is bounded
(see, for instance, \cite[Corollary~2.3.5]{lyh2320}).
By this and Proposition~\ref{1537} with
both $X:=\dot{\mathcal{K}}^{p,q}_{\omega}$
and $\rho:=\rho_\nu$, we conclude that,
for any $\nu\in(0,\nu_0)$,
\begin{align*}
\left\|\left[\int_\Omega\frac{|f(\cdot)-f(y)|^p}{
|\cdot-y|^p}\rho_\nu(|\cdot-y|)\,dy
\right]^\frac{1}{p}\right\|_{
\dot{\mathcal{K}}^{p,q}_{\omega}(\Omega)}
\lesssim\left\|\,\left|\nabla f\right|\,\right\|_{
\dot{\mathcal{K}}^{p,q}_{\omega}(\Omega)},
\end{align*}
where the implicit positive constant
is independent of both $\nu$
and $f$. From this, we deduce that
\begin{align*}
&\liminf_{\nu\to0^+}
\left\|\left[\int_\Omega\frac{|f(\cdot)-f(y)|^p}{
|\cdot-y|^p}\rho_\nu(|\cdot-y|)\,dy
\right]^\frac{1}{p}\right\|_{\dot{\mathcal{K}}^{p,q}_{\omega}(\Omega)}\\
&\quad\leq\limsup_{\nu\to0^+}
\left\|\left[\int_\Omega\frac{|f(\cdot)-f(y)|^p}{
|\cdot-y|^p}\rho_\nu(|\cdot-y|)\,dy
\right]^\frac{1}{p}\right\|_{
\dot{\mathcal{K}}^{p,q}_{\omega}(\Omega)}\nonumber\\
&\quad\lesssim\left\|\,\left|\nabla f\right|\,\right\|_{
\dot{\mathcal{K}}^{p,q}_{\omega}(\Omega)}<\infty.\nonumber
\end{align*}
This finishes the proof of the necessity
and hence Theorem~\ref{globalHerz}.
\end{proof}

\begin{remark}
To the best of our knowledge,
Theorem~\ref{globalHerz} is completely new.
\end{remark}

\subsection{Mixed-Norm Lebesgue Spaces}\label{5.2}

Let $\vec{r}:=(r_1,\ldots,r_n)
\in(0,\infty]^n$ and
$r_-:=\min\{r_1, \ldots , r_n\}$.
The \emph{mixed-norm Lebesgue
space $L^{\vec{r}}(\mathbb{R}^n)$} is defined to be the
set of all the $f\in\mathscr{M}(\mathbb{R}^n)$
having the following finite quasi-norm
\begin{equation*}
\|f\|_{L^{\vec{r}}(\mathbb{R}^n)}:=\left\{\int_{\mathbb{R}}
\cdots\left[\int_{\mathbb{R}}\left|f(x_1,\ldots,
x_n)\right|^{r_1}\,dx_1\right]^{\frac{r_2}{r_1}}
\cdots\,dx_n\right\}^{\frac{1}{r_n}}
\end{equation*}
with the usual modifications made when $r_i=
\infty$ for some $i\in\{1,\ldots,n\}$.
The mixed-norm Lebesgue space
was systematically studies by
Benedek and Panzone \cite{bp1961}, which
can be traced back to H\"ormander \cite{h1960}.
We refer the reader
to \cite{cgn2017,cgn2019,gjn2017,gn2016,
hlyy2019,hlyy2019b,hy2021,n2019,noss2021,zyz2022}
for more studies on mixed-norm Lebesgue spaces.
As was pointed out in \cite[p.\,2047]{zwyy2021},
$L^{\vec{r}}(\mathbb{R}^n)$
is a ball quasi-Banach function space,
but it
may not be a quasi-Banach function space
in the sense of Bennett and Sharpley \cite{bs1988}.

By an argument similar to that used in
the proof of \cite[Theorem~4.19]{zyy20231}
with \cite[Theorems~3.1 and~3.11]{zyy20231} replaced by
Theorem~\ref{1622},
we obtain the following conclusion;
we omit the details.

\begin{theorem}\label{2116}
Let $\Omega\subset\mathbb{R}^n$ be a bounded
$(\varepsilon,\infty)$-domain
with $\varepsilon\in(0,1]$ and
$\{\rho_\nu\}_{\nu\in(0,\nu_0)}$
a $\nu_0$-{\rm RDATI} on $\mathbb{R}^n$
with $\nu_0\in(0,\infty)$.
Let $\vec{r}:=(r_1,\ldots,r_n)\in(1,\infty)^n$ and $p\in[1,r_-)$.
Then $f\in W^{1,L^{\vec{r}}}(\Omega)$ if and only if
$f\in L^{\vec{r}}(\Omega)$ and
$$
\liminf_{\nu\to0^+}
\left\|\left[\int_\Omega\frac{|f(\cdot)-f(y)|^p}{
|\cdot-y|^p}\rho_\nu(|\cdot-y|)\,dy
\right]^\frac{1}{p}\right\|_{L^{\vec{r}}(\Omega)}<\infty;
$$
moreover, for such $f$,
$$
\lim_{\nu\to0^+}
\left\|\left[\int_\Omega\frac{|f(\cdot)-f(y)|^p}{
|\cdot-y|^p}\rho_\nu(|\cdot-y|)\,dy
\right]^\frac{1}{p}\right\|_{L^{\vec{r}}(\Omega)}
=\left[\kappa(p,n)\right]^\frac{1}{p}
\left\|\,\left|\nabla f\right|\,\right\|_{L^{\vec{r}}(\Omega)},
$$
where the constant
$\kappa(p,n)$ is the same as in \eqref{kappaqn}.
\end{theorem}

\begin{remark}
\begin{enumerate}
\item[(i)]
To the best of our knowledge,
Theorem~\ref{2116} is completely new.
\item[(ii)]
Let $\vec{r}:=(r_1,\ldots,r_n)\in[1,\infty)^n$
with $r_-=1$. In this case,
since the Hardy--Littlewood maximal operator
may not be bounded on $[L^{\vec{r}}(\mathbb{R}^n)]'$
and since $[L^{\vec{r}}(\mathbb{R}^n)]'$
does not have an absolutely continuous norm,
it is still unclear whether or not
Theorem~\ref{2116} with $r_-=1$ holds true.
\end{enumerate}
\end{remark}

\subsection{Variable Lebesgue Spaces}\label{5.3}

For any given nonnegative measurable function $r$
on $\mathbb{R}^n$,
let
\begin{equation*}
\widetilde{r}_-:=\underset{x\in\mathbb{R}^n}{
\mathop\mathrm{\,ess\,inf\,}}\,r(x)\ \ \text{and}\ \
\widetilde{r}_+:=\underset{x\in\mathbb{R}^n}{
\mathop\mathrm{\,ess\,sup\,}}\,r(x).
\end{equation*}
A nonnegative measurable function $r$
is said to be \emph{globally
log-H\"older continuous} if there exist
$r_{\infty}\in\mathbb{R}$ and a positive
constant $C$ such that, for any
$x,y\in\mathbb{R}^n$,
\begin{equation*}
|r(x)-r(y)|\le\frac{C}{\log(e+\frac{1}{|x-y|})}\ \ \text{and}\ \
|r(x)-r_\infty|\le \frac{C}{\log(e+|x|)}.
\end{equation*}
Recall that the \emph{variable Lebesgue space
$L^{r(\cdot)}(\mathbb{R}^n)$}
associated with a nonnegative measurable function
$r$ on $\mathbb{R}^n$ is defined to be the set
of all the $f\in\mathscr{M}(\mathbb{R}^n)$
having the following finite quasi-norm
\begin{equation*}
\|f\|_{L^{r(\cdot)}(\mathbb{R}^n)}:=\inf\left\{\lambda
\in(0,\infty):\ \int_{\mathbb{R}^n}\left[\frac{|f(x)|}
{\lambda}\right]^{r(x)}\,dx\le1\right\}.
\end{equation*}
The variable Lebesgue space
was introduced by Kov\'a$\check{\mathrm{c}}$ik and R\'akosn\'ik
\cite{kr1991}. In \cite{kr1991},
a detailed study of the properties of
variable Lebesgue spaces was given and these spaces were
further applied to study both the mapping properties
of Nemytskii operators and the related nonlinear elliptic
boundary value problems.
For more studies on variable Lebesgue spaces,
we refer the reader to
\cite{b2018,bbd2021,cf2013,cw2014,dhr2009,ns2012,n1950,n1951}.
By \cite[p.\,94]{shyy2017},
we find that $L^{r(\cdot)}(\mathbb{R}^n)$
is a ball quasi-Banach function space.
In particular,
when $1\leq\widetilde r_-\le \widetilde r_+<\infty$,
$L^{r(\cdot)}(\mathbb{R}^n)$ is
a Banach function space
in the sense of Bennett and Sharpley \cite{bs1988}
and hence also a ball Banach function space.

By an argument similar to that used in
the proof of \cite[Theorem~4.21]{zyy20231}
with \cite[Theorems~3.1 and~3.11]{zyy20231} replaced by
Theorem~\ref{1622},
we obtain the following conclusion;
we omit the details.

\begin{theorem}\label{2041}
Let $\Omega\subset\mathbb{R}^n$ be a bounded
$(\varepsilon,\infty)$-domain
with $\varepsilon\in(0,1]$ and
$\{\rho_\nu\}_{\nu\in(0,\nu_0)}$
a $\nu_0$-{\rm RDATI} on $\mathbb{R}^n$
with $\nu_0\in(0,\infty)$.
Let $r:\ \mathbb{R}^n\to(0,\infty)$ be a globally
log-H\"older continuous function with
$1<\widetilde{r}_-\leq\widetilde{r}_+<\infty$
and let $p\in[1,\widetilde{r}_-)$.
Then $f\in W^{1,L^{r(\cdot)}}(\Omega)$ if and only if
$f\in L^{{r}(\cdot)}(\Omega)$ and
$$
\liminf_{\nu\to0^+}
\left\|\left[\int_\Omega\frac{|f(\cdot)-f(y)|^p}{
|\cdot-y|^p}\rho_\nu(|\cdot-y|)\,dy
\right]^\frac{1}{p}\right\|_{L^{{r}(\cdot)}(\Omega)}<\infty;
$$
moreover, for such $f$,
$$
\lim_{\nu\to0^+}
\left\|\left[\int_\Omega\frac{|f(\cdot)-f(y)|^p}{
|\cdot-y|^p}\rho_\nu(|\cdot-y|)\,dy
\right]^\frac{1}{p}\right\|_{L^{{r}(\cdot)}(\Omega)}
=\left[\kappa(p,n)\right]^\frac{1}{p}
\left\|\,\left|\nabla f\right|\,\right\|_{L^{{r}(\cdot)}(\Omega)},
$$
where the constant
$\kappa(p,n)$ is the same as in \eqref{kappaqn}.
\end{theorem}

\begin{remark}
\begin{enumerate}
\item[(i)]
To the best of our knowledge,
Theorem~\ref{2041} is completely new.
\item[(ii)]
Let $r:\ \mathbb{R}^n\to(0,\infty)$ be globally
log-H\"older continuous with
$1=\widetilde{r}_-\leq\widetilde{r}_+<\infty$.
In this case, since $[L^{r(\cdot)}(\mathbb{R}^n)]'$
does not have an absolutely continuous
norm, it is still unclear whether or not
Theorem~\ref{2041} with $\widetilde{r}_-=1$
holds true.
\end{enumerate}
\end{remark}

\subsection{Lorentz Spaces}\label{5.7}

Recall that, for any $r,\tau\in(0,\infty)$,
the \emph{Lorentz space $L^{r,\tau}(\mathbb{R}^n)$}
is defined to be the set of all the
$f\in\mathscr{M}(\mathbb{R}^n)$ having
the following finite quasi-norm
\begin{equation*}
\|f\|_{L^{r,\tau}(\mathbb{R}^n)}
:=\left\{\int_0^{\infty}
\left[t^{\frac{1}{r}}f^*(t)\right]^\tau
\frac{\,dt}{t}\right\}^{\frac{1}{\tau}},
\end{equation*}
where $f^*$ denotes the \emph{decreasing rearrangement of $f$},
defined by setting, for any $t\in[0,\infty)$,
\begin{equation*}
f^*(t):=\inf\left\{s\in(0,\infty):\ \left|
\left\{x\in\mathbb{R}^n:\ |f(x)|>s\right\}\right|\leq t\right\}
\end{equation*}
with the convention $\inf \emptyset = \infty$.
As a generalization of Lebesgue spaces,
Lorentz spaces were originally studied by Lorentz \cite{l1950,l1951}
in the early 1950's, which are recognized to be
the intermediate spaces of Lebesgue spaces in the real interpolation method
(see, for instance, \cite{c1964}).
We refer the reader to \cite{CF1,CF2,Hunt,osttw2012,st2001}
for more studies on Lorentz spaces
and to \cite{lwyy2019,lyy2016,lyy2017,
lyy2018,zhy2020}
for more studies on Hardy--Lorentz spaces.
When $r,\tau\in(0,\infty)$,
$L^{r,\tau}(\mathbb{R}^n)$ is a quasi-Banach function space
in the sense of Bennett and Sharpley \cite{bs1988}
and hence a ball quasi-Banach function space
(see, for instance, \cite[Theorem 1.4.11]{g2014});
when $r,\tau\in(1,\infty)$,
$L^{r,\tau}(\mathbb{R}^n)$ is a Banach function space
and hence a ball Banach function space
(see, for instance, \cite[p.\,87]{shyy2017}
and \cite[p.\,74]{g2014}).

By an argument similar to that used in
the proof of \cite[Theorem~4.23]{zyy20231}
with \cite[Theorems~3.1 and~3.11]{zyy20231} replaced by
Theorem~\ref{1622},
we obtain the following conclusion;
we omit the details.

\begin{theorem}\label{2047}
Let $\Omega\subset\mathbb{R}^n$ be a bounded
$(\varepsilon,\infty)$-domain
with $\varepsilon\in(0,1]$ and
$\{\rho_\nu\}_{\nu\in(0,\nu_0)}$
a $\nu_0$-{\rm RDATI} on $\mathbb{R}^n$
with $\nu_0\in(0,\infty)$.
Let $r,\tau\in(1,\infty)$ and $p\in[1,\min\{r,\tau\})$.
Then $f\in W^{1,L^{r,\tau}}(\Omega)$ if and only if
$f\in L^{r,\tau}(\Omega)$ and
$$
\liminf_{\nu\to0^+}
\left\|\left[\int_\Omega\frac{|f(\cdot)-f(y)|^p}{
|\cdot-y|^p}\rho_\nu(|\cdot-y|)\,dy
\right]^\frac{1}{p}\right\|_{L^{r,\tau}(\Omega)}<\infty;
$$
moreover, for such $f$,
$$
\lim_{\nu\to0^+}
\left\|\left[\int_\Omega\frac{|f(\cdot)-f(y)|^p}{
|\cdot-y|^p}\rho_\nu(|\cdot-y|)\,dy
\right]^\frac{1}{p}\right\|_{L^{r,\tau}(\Omega)}
=\left[\kappa(p,n)\right]^\frac{1}{p}
\left\|\,\left|\nabla f\right|\,\right\|_{L^{r,\tau}(\Omega)},
$$
where the constant
$\kappa(p,n)$ is the same as in \eqref{kappaqn}.
\end{theorem}

\begin{remark}
\begin{enumerate}
\item[\textup{(i)}]
To the best of our knowledge,
Theorem~\ref{2047} is new.
\item[\textup{(ii)}]
Let $r=1$ and $\tau\in[1,\infty)$
or let $r\in(1,\infty)$ and $\tau=1$.
In both cases, since $L^{r,\tau}(\mathbb{R}^n)$ might not be
a ball Banach function space and since
$[L^{r,\tau}(\mathbb{R}^n)]'$ does not have an absolutely
continuous norm, it is still unclear whether or not
Theorem~\ref{2047} holds true.
\end{enumerate}
\end{remark}

\subsection{Orlicz Spaces}\label{5.5}

A non-decreasing function $\Phi:\ [0,\infty)
\ \to\ [0,\infty)$ is called an \emph{Orlicz function} if
$\Phi(0)= 0$,
$\Phi(t)\in(0,\infty)$ for any $t\in(0,\infty)$,
and $\lim_{t\to\infty}\Phi(t)=\infty$.
An Orlicz
function $\Phi$ is said to be of \emph{lower}
(resp. \emph{upper}) \emph{type} $r$ for some
$r\in\mathbb{R}$ if there exists a constant
$C_{(r)}\in(0,\infty)$ such that,
for any $t\in(0,\infty)$ and
$s\in(0,1)$ [resp. $s\in[1,\infty)$],
$\Phi(st)\le C_{(r)} s^r\Phi(t)$.
Throughout this subsection, we always assume that
$\Phi:\ [0,\infty)\ \to\ [0,\infty)$
is an Orlicz function with both positive lower
type $r_{\Phi}^-$ and positive upper type
$r_{\Phi}^+$.
The \emph{Orlicz space $L^\Phi(\mathbb{R}^n)$}
is defined to be the set of all the
$f\in\mathscr{M}(\mathbb{R}^n)$
having the following finite quasi-norm
\begin{equation*}
\|f\|_{L^\Phi(\mathbb{R}^n)}:=\inf\left\{\lambda\in
(0,\infty):\ \int_{\mathbb{R}^n}\Phi\left(\frac{|f(x)|}
{\lambda}\right)\,dx\le1\right\}.
\end{equation*}
The Orlicz space was introduced by Birnbaum and Orlicz \cite{bo1931}
and Orlicz \cite{o}, which is widely used in
various branches of analysis.
For more studies on Orlicz spaces,
we refer the reader to \cite{b2021,dfmn2021,ns2014,rr2002,ylk2017}.
It is easy to show that $L^\Phi(\mathbb{R}^n)$
is a quasi-Banach function space
(see \cite[Section~7.6]{shyy2017})
and hence a ball quasi-Banach function space.
In particular,
when $1\leq r^-_\Phi\leq r^+_\Phi<\infty$,
$L^\Phi(\mathbb{R}^n)$ is a Banach function space
in the sense of Bennett and Sharpley \cite{bs1988}
and hence a ball Banach function space.

By an argument similar to that used in
the proof of \cite[Theorem~4.25]{zyy20231}
with \cite[Theorems~3.1 and~3.11]{zyy20231} replaced by
Theorem~\ref{1622},
we obtain the following conclusion;
we omit the details.

\begin{theorem}\label{2051}
Let $\Omega\subset\mathbb{R}^n$ be a bounded
$(\varepsilon,\infty)$-domain
with $\varepsilon\in(0,1]$ and
$\{\rho_\nu\}_{\nu\in(0,\nu_0)}$
a $\nu_0$-{\rm RDATI} on $\mathbb{R}^n$
with $\nu_0\in(0,\infty)$.
Let $\Phi$ be an Orlicz function with both
lower type $r^-_{\Phi}$
and upper type $r^+_\Phi$,
$1<r^-_{\Phi}\leq r^+_{\Phi}<\infty$,
and $p\in[1,r^-_{\Phi})$.
Then $f\in W^{1,L^\Phi}(\Omega)$ if and only if
$f\in L^{\Phi}(\Omega)$ and
$$
\liminf_{\nu\to0^+}
\left\|\left[\int_\Omega\frac{|f(\cdot)-f(y)|^p}{
|\cdot-y|^p}\rho_\nu(|\cdot-y|)\,dy
\right]^\frac{1}{p}\right\|_{L^{\Phi}(\Omega)}<\infty;
$$
moreover, for such $f$,
$$
\lim_{\nu\to0^+}
\left\|\left[\int_\Omega\frac{|f(\cdot)-f(y)|^p}{
|\cdot-y|^p}\rho_\nu(|\cdot-y|)\,dy
\right]^\frac{1}{p}\right\|_{L^{\Phi}(\Omega)}
=\left[\kappa(p,n)\right]^\frac{1}{p}
\left\|\,\left|\nabla f\right|\,\right\|_{L^{\Phi}(\Omega)},
$$
where the constant
$\kappa(p,n)$ is the same as in \eqref{kappaqn}.
\end{theorem}

\begin{remark}
\begin{enumerate}
\item[\textup{(i)}]
To the best of our knowledge,
Theorem~\ref{2051} is completely new.
\item[\textup{(ii)}]
Let $\Phi$ be an Orlicz function with both
positive lower type $r^-_{\Phi}$
and positive upper type $r^+_\Phi$
such that $1=r^-_{\Phi}\leq r^+_{\Phi}<\infty$.
In this case, since $[L^\Phi(\mathbb{R}^n)]'$
does not have an absolutely continuous norm,
it is still
unclear whether or not Theorem~\ref{2051} holds true.
\end{enumerate}
\end{remark}

\subsection{Orlicz-Slice Spaces}\label{5.6}

Let $\Phi$
be an Orlicz function with both positive
lower type $r_{\Phi}^-$ and positive upper
type $r_{\Phi}^+$. For any given $t,r\in(0,\infty)$,
the \emph{Orlicz-slice space}
$(E_\Phi^r)_t(\mathbb{R}^n)$ is defined to be the set of all
the $f\in\mathscr{M}(\mathbb{R}^n)$ having the following finite quasi-norm
\begin{equation*}
\|f\|_{(E_\Phi^r)_t(\mathbb{R}^n)} :=\left\{\int_{\mathbb{R}^n}
\left[\frac{\|f\mathbf{1}_{B(x,t)}\|_{L^\Phi(\mathbb{R}^n)}}
{\|\mathbf{1}_{B(x,t)}\|_{L^\Phi(\mathbb{R}^n)}}\right]
^r\,dx\right\}^{\frac{1}{r}}.
\end{equation*}
These spaces were originally introduced in
\cite{zyyw2019} as a generalization of both
the slice space of Auscher and Mourgoglou
\cite{am2019,ap2017} and the Wiener amalgam space
in \cite{h2019,h1975,kntyy2007}.
For more studies on Orlicz-slice spaces,
we refer the reader to \cite{hkp2021,hkp2022,zyy2022,zyyw2019}.
By \cite[Lemma 2.28]{zyyw2019},
we find that the Orlicz-slice
space $(E_\Phi^r)_t(\mathbb{R}^n)$ is a
ball quasi-Banach function space.

By an argument similar to that used in
the proof of \cite[Theorem~4.27]{zyy20231}
with \cite[Theorems~3.1 and~3.11]{zyy20231} replaced by
Theorem~\ref{1622},
we obtain the following conclusion;
we omit the details.

\begin{theorem}\label{2055}
Let $\Omega\subset\mathbb{R}^n$ be a bounded
$(\varepsilon,\infty)$-domain
with $\varepsilon\in(0,1]$ and
$\{\rho_\nu\}_{\nu\in(0,\nu_0)}$
a $\nu_0$-{\rm RDATI} on $\mathbb{R}^n$
with $\nu_0\in(0,\infty)$.
Let $t\in(0,\infty)$, $r\in(1,\infty)$,
and $\Phi$ be an Orlicz function with
both positive lower type $r^-_{\Phi}$
and positive upper type $r^+_\Phi$.
Let $1<r^-_{\Phi}\leq r^+_{\Phi}<\infty$
and $p\in[1,\min\{r^-_{\Phi},\,r\})$.
Then $f\in W^{1,(E_\Phi^r)_t}(\Omega)$
if and only if
$f\in (E_\Phi^r)_t(\mathbb{R}^n)$ and
$$
\liminf_{\nu\to0^+}
\left\|\left[\int_\Omega\frac{|f(\cdot)-f(y)|^p}{
|\cdot-y|^p}\rho_\nu(|\cdot-y|)\,dy
\right]^\frac{1}{p}\right\|_{(E_\Phi^r)_t(\Omega)}<\infty;
$$
moreover, for such $f$,
$$
\lim_{\nu\to0^+}
\left\|\left[\int_\Omega\frac{|f(\cdot)-f(y)|^p}{
|\cdot-y|^p}\rho_\nu
(|\cdot-y|)\,dy
\right]^\frac{1}{p}\right\|_{(E_\Phi^r)_t(\Omega)}
=\left[\kappa(p,n)\right]^\frac{1}{p}
\left\|\,\left|\nabla f
\right|\,\right\|_{(E_\Phi^r)_t(\Omega)},
$$
where the constant
$\kappa(p,n)$ is the same as in \eqref{kappaqn}.
\end{theorem}

\begin{remark}
\begin{enumerate}
\item[\textup{(i)}]
To the best of our knowledge,
Theorem~\ref{2055} is completely new.
\item[\textup{(ii)}]
Let $\Phi$ be an Orlicz function with both
positive lower type $r^-_{\Phi}$
and positive upper type $r^+_\Phi$
such that $1=r^-_{\Phi}\leq r^+_{\Phi}<\infty$.
In this case, since $[(E_\Phi^r)_t(\mathbb{R}^n)]'$
does not have an absolutely continuous norm,
it is still
unclear whether or not Theorem~\ref{2055} holds true.
\end{enumerate}
\end{remark}

\bigskip

\noindent  Chenfeng Zhu, Dachun Yang (Corresponding author) and Wen Yuan

\medskip

\noindent Laboratory of Mathematics
and Complex Systems (Ministry of Education of China),
School of Mathematical Sciences,
Beijing Normal University,
Beijing 100875, The People's Republic of China

\smallskip

\noindent{\it E-mails:} \texttt{cfzhu@mail.bnu.edu.cn} (C. Zhu)

\noindent\phantom{{\it E-mails:} }\texttt{dcyang@bnu.edu.cn} (D. Yang)

\noindent\phantom{{\it E-mails:} }\texttt{wenyuan@bnu.edu.cn} (W. Yuan)

\end{document}